\documentclass[a4paper,11pt]{amsart}

\usepackage{amssymb, manfnt}
\usepackage{hyperref}
\usepackage[normalem]{ulem}
\usepackage{tikz}
\usepackage{booktabs}
\usepackage{multirow}
\usepackage{MnSymbol}
\usepackage{graphicx}
\usepackage{tikz-cd}
\usepackage{longtable}
\usepackage{amsmath,amscd}
\usepackage{tkz-graph}
\tikzstyle{vertex}=[circle, draw, inner sep=0pt, minimum size=6pt]
\newcommand{\vertex}{\node[vertex]}
\usepackage[all]{xy}
\graphicspath{ {graphs/} }
\usetikzlibrary{decorations.text}
\usepackage{graphicx}
\usepackage[outdir=./pictures]{epstopdf}
\graphicspath{ {pictures/} }

\def\altdb{\vadjust{\vbox to 0pt{\vss\hbox{\kern \hsize
\quad{\dbend}}\kern\baselineskip\kern-10pt}}}

\newcommand{\arxiv}[1]{\href{http://arxiv.org/abs/#1}{\tt arXiv:\nolinkurl{#1}}}
\newcommand{\doi}[1]{\href{http://dx.doi.org/#1}{{\tt DOI:#1}}}

\newcommand{\googlebooks}[1]{(preview at \href{http://books.google.com/books?id=#1}{google books})}

\makeatletter
\let\@@pmod\pmod
\DeclareRobustCommand{\pmod}{\@ifstar\@pmods\@@pmod}
\def\@pmods#1{\mkern4mu({\operator@font mod}\mkern 6mu#1)}
\makeatother

\newcommand\id{\operatorname{id}}
\newcommand\Aut{\operatorname{Aut}}
\newcommand\TenAut{\operatorname{Eq}}
\newcommand\Rep{\operatorname{Rep}}

\newcommand\Id{\operatorname{Id}}
\newcommand\Hom{\operatorname{Hom}}

\newcommand\ad{\operatorname{Ad}}
\newcommand\BrAut{\operatorname{EqBr}}

\newcommand\Z[1]{ \mathbb{Z}_{#1}}

\newcommand\lie[1]{ \mathfrak{#1}}
\newcommand\cC{ \mathcal{C}}
\newcommand\cF{ \mathcal{F}}
\newcommand\cat[3]{ \cC(   \mathfrak{#1}_{#2} , #3)    }

\theoremstyle{plain}
\newtheorem{theorem}{Theorem}[section]
\newtheorem*{theorem*}{Theorem}
\newtheorem*{prop*}{Proposition}
\newtheorem{cor}[theorem]{Corollary}
\newtheorem{lemma}[theorem]{Lemma}
\newtheorem{prop}[theorem]{Proposition}

\newtheorem{rmk}[theorem]{Remark}
\theoremstyle{remark}

\theoremstyle{definition}
\newtheorem{dfn}[theorem]{Definition}
\newtheorem{cons}[theorem]{Construction}

\newtheorem{ques}[theorem]{Question}

\newcommand\Inv{\operatorname{Inv}}

\usepackage{color}

\numberwithin{equation}{section}

\DeclareRobustCommand*{\nicefrac}{\@UnitsNiceFrac}%
\makeatother

\newcommand{\gra}[1]{\raisebox{-.4cm}{\includegraphics[height=1cm]{PH/UP#1.pdf}}}
\newcommand{\graa}[1]{\raisebox{-.6cm}{\includegraphics[height=1.5cm]{PH/UP#1.pdf}}}

\title{Auto-equivalences of the modular tensor categories  of type $A$, $B$, $C$ and $G$}%$\operatorname{Rep}(U_qSp(N))$}

\author{Cain Edie-Michell}
\address{Cain Edie-Michell\\
Department of Mathematics\\
Vanderbilt University, Nashville\
USA}
\email{cain.edie-michell@vanderbilt.edu}

\allowdisplaybreaks
\begin{document}

\maketitle

\begin{abstract}
We compute the monoidal and braided auto-equivalences of the modular tensor categories $\cat{sl}{r+1}{k}$, $\cat{so}{2r+1}{k}$, $\cat{sp}{2r}{k}$, and $\cat{g}{2}{k}$. Along with the expected simple current auto-equivalences, we show the existence of the charge conjugation auto-equivalence of $\cat{sl}{r+1}{k}$, and exceptional auto-equivalences of $\cat{so}{2r+1}{2}$, $\cat{sp}{2r}{r}$, $\cat{g}{2}{4}$. We end the paper with a section discussing potential applications of these computations, including the relationship of these computations to the program to classify quantum subgroups of the simple Lie algebras. Included is an appendix by Terry Gannon, which discusses the group structure of the auto-equivalences of $\cat{sl}{r+1}{k}$.
\end{abstract}

\section{Introduction}\label{sec:intro}

Modular tensor categories have increasingly become ubiquitous in many areas of modern mathematics. Several important appearances include the superselection sectors of a rational conformal field theory \cite{MR1016869}, the even parts of certain subfactor standard invariant \cite{MR1424954}, and the representation category of a quantum group \cite{MR1227098} (after an appropriate semisimplification has been applied). With these far ranging applications, structure results for modular categories are in high demand. 

Given a modular tensor category $\cC$, a particularly useful invariant one can compute is the group of auto-equivalences of $\cC$; either braided (which we denote $\BrAut(\cC)$), or plain monoidal (which we denote $\TenAut(\cC)$). The usefulness of these two invariants is mainly due to the various constructions one can perform on $\cC$ once these are computed. For example, given a finite group $G$, and a homomorphism $ G \to \BrAut(\cC)$ one can (if certain obstructions vanish) construct a new modular tensor category called the gauging of $\cC$ \cite{MR3555361}. Further, for every braided auto-equivalence of $\cC$, one can construct a ``quantum subgroup'' (in the sense of \cite{MR1907188}) of $\cC$ \cite[Corollary 3.8]{MR3022755}.

Roughly speaking, all known modular tensor categories come from two sources. The first is from Drinfeld centres of spherical fusion categories. The braided auto-equivalences of these examples have been extensively studied, with particular focus on the Drinfeld centres of fusion categories coming from exotic subfactors \cite{MR2909758, MR3449240, 1810.06076}. The second source comes from quantum groups. Let $\mathfrak{g}$ be a simple Lie algebra, and $\hat{\mathfrak{g}}$ the corresponding affine Lie algebra. Then, for every $k\geq 1$, the category of level $k$ integrable representations of $\hat{\mathfrak{g}}$ has the structure of a modular tensor category \cite{MR1227098}. These categories are usually denoted $\cat{g}{}{k}$. Unlike the Drinfeld centre examples, very little is known about the auto-equivalences of these modular tensor categories. To the authors best knowledge, the only case known is when $\mathfrak{g} = \mathfrak{sl}_2$ \cite{Cain-normal}. Hence it is a natural question to compute these auto-equivalence groups for all the higher rank $\mathfrak{g}$.

It has been a long standing open problem to classify the quantum subgroups of the simple Lie algebras, which was initially laid out by Ocneanu \cite{MR1907188} and Gannon \cite{MR1266482}. These quantum subgroups have important applications to the various areas of mathematics that modular tensor categories intertwine. For example, such quantum subgroups classify defect sectors in Chern-Simons topological quantum field theories, and can be used to construct highly non-trivial subfactors of the hyperfinite type $II_1$ factor. These vast applications has meant that an enormous amount of effort has been spent on trying to classify these quantum subgroups. However until recently, only a small amount of progress has been made, with only the $\mathfrak{sl}_2$ and $\mathfrak{sl}_3$ cases being completed \cite{MR1266482}. One of the big stumbling blocks for this program is that for each Lie algebra, there are \textit{a-priori} an infinite number of levels where a quantum subgroup may occur, which makes searching and classifying them very difficult. Recent breakthrough results of Schopieray \cite{MR3808050} and Gannon \cite{Level-Bounds} solve this issue, giving effective bounds on the levels for which exceptional quantum subgroups can occur. With these recent results, we will now very have an list exhaustive list of potential quantum subgroups of the low rank simple Lie algebras. The problem then shifts to constructing these quantum subgroups. The braided auto-equivalences classified in this paper will allow the construction of a whole new plethora of quantum subgroups of the simple Lie algebras, completing a significant portion of the program laid out by to classify all quantum subgroups of the simple Lie algebras.

In this paper we are able to compute both $\TenAut(\cC)$ and $\BrAut(\cC)$, where $\cC \simeq \cat{g}{}{k}$ with $\mathfrak{g}$ a simple Lie algebra of type $A_r$, $B_r$, $C_r$, or $G_2$ for all ranks and levels. Our main theorem is as follows.

\begin{theorem}\label{thm:main}
Let $r$ and $k$ positive integers. We have:
    
    \begin{tabular}{c | l | cc l}
    	\toprule
			\multicolumn{2}{l |}{$\mathcal{C} $}      &       $\TenAut(\mathcal{C})$     & $\BrAut(\mathcal{C})$ \\
	\midrule
	 $\cC(\mathfrak{sl}_{r+1} , k)$     &	$r = 1$ and $k = 2$				  & $\{e\}$  &$\{e\}$ \\
	  					&		$2$ exactly divides $\operatorname{gcd}(r+1,k)$		& $\Z{2}^c\times \Z{n'}^\times \times \Z{2}\times \Z{\frac{n''}{2}}$ & $\Z{2}^{c+p+t}$ \\
	  					&		otherwise					  & $\Z{2}^c\times \Z{n'}^\times \times \Z{n''}$ & $\Z{2}^{c+p+t}$ \\
		              			
		              \midrule
	
  	 $\cC(\mathfrak{so}_{2r+1} , k)$     &	$k =1$							  &$\{e\}$  & $\{e\}$ \\

		              &	$k =2 $ 	and every prime $p$ dividing 				&$\Z{2}^{\omega(2r+1)+1} $ & $ \Z{2}^{\omega(2r+1)-1}$  \\
		              &  $2r+1$ satisfies $p\equiv 1 \pmod 4$ 				&						   &								\\
		              &	$k =2 $ 	and there exists a prime $p$ dividing			&$\Z{2}^{\omega(2r+1)} $ & $ \Z{2}^{\omega(2r+1)-1}$  \\
		              &  $2r+1$ such that $p\nequiv 1 \pmod 4$					&						   &								\\
		               &	$k \geq 3$ and $k \equiv 0 \pmod 2$ 								  &$\Z{2}$  & $\{e\}$ \\
		      		  &	$k \geq 3$ and $k \equiv 1 \pmod 2$ 							      &$\Z{2}$ & $ \Z{2}$  \\    
		              \midrule
  		      $\cC(\mathfrak{sp}_{2r} , k)$   			          &    $r = 2$ and $k=1$												&$\{e\}$ & $ \{e\}$	\\
  		      	 &	$r=k$ and $r \equiv 0 \pmod 2$ 								  &$\Z{2}\times\Z{2}$ & $\{e\}$ \\
		      		  &	$r=k$ and $r \equiv 1 \pmod 2$ 								  &$\Z{2}$ & $ \{e\}$  \\    
		                   &	$r\neq k$ and $rk \equiv 1 \pmod 2$ 										  &$\{e\}$ & $ \{e\}$  \\
			          &	$r\neq k$ and $rk \equiv 0 \pmod 4$ 					          				&$\Z{2}$& $\{e\}$ \\
			          &	$r\neq k$ and $rk \equiv 2 \pmod 4$ 					          				&$\Z{2}$& $\Z{2}$ \\		
	
			          	\midrule
  		       $\cC(\mathfrak{g}_2 , k)$    &	$k=4$								  &$\Z{2}$ & $\Z{2}$ \\
		                   &	$k \neq 4$ 							&$\{e\}$ & $ \{e\}$.  \\

    	\bottomrule

	    \end{tabular}
	    
Where

\begin{itemize}

\item $n'' = \operatorname{gcd}(r+1, k^\infty)$,

\item $n' = \frac{n}{n''}$,

\item $ c :=  \begin{cases}
0, \text{ if } k \leq 2 \text{ or } r = 1 \\
1 , \text{ if } k \geq 3 \text{ and } r \neq 1,
\end{cases}$

\item $ t :=  \begin{cases}
  \resizebox{0.9\hsize}{!}{$0, \text{ if } r \text{ is even, or if } r \text{ is odd and } k \equiv 0 \pmod 4, \text{ or if } k \text{ is odd and } r \equiv 1 \pmod 4$}, \\
1 , \text{ otherwise},
\end{cases}$

\item $p$ is the number of distinct odd primes that divide $r+1$ but not $k$, and

 \item $\omega(n)$ is the number of distinct prime divisors of $n$.
 
 \end{itemize}	   
\end{theorem}

\begin{rmk}
We remark that the result for $\mathfrak{g} = \mathfrak{g}_2$  in the above theorem is entirely reliant on results of Ostrik and Snyder that has yet to appear in print. 
\end{rmk} 

The auto-equivalences we classify in this Theorem can be broken into three classes. The first are the simple current auto-equivalences. These occur whenever there is an invertible object in modular category satisfying some mild conditions on its self-braid eigenvalue. These are the vast majority of the auto-equivalences in the classification. They should not be considered exotic, or even exciting. All constructions one can perform with simple current auto-equivalences result in essentially trivial information. 

The second class of auto-equivalences are the charge-conjugation auto-equivalences, which occur for $\cat{sl}{r+1}{k}$ at all levels for $r \geq 2$. These are the braided auto-equivalences which map $X \mapsto X^*$. These can be considered semi-exceptional, as they give rise to interesting gauged categories, and new quantum subgroups. However, they exist for all $r$ and $k$, and their existence essentially follows from classical Lie theory.

 The final class are the truly exceptional auto-equivalences, which occur sporadically, and can only be shown to exist through a variety of ad-hoc methods. The existence and classification of these exceptional auto-equivalences is the true substance of this paper. We find such exceptional auto-equivalences for the following categories:
\begin{align*}
\cat{so}{2r+1}{2}  & \text{ for all $r\geq 2$},\\
\cat{sp}{2r}{r}	 & \text{ for all $r\geq 2$, and }\\
\cat{g}{2}{4}.   &
\end{align*}
The author is still exploring the exciting consequences of the existence of these exceptional auto-equivalences. For now we wish to draw attention to the quantum subgroups that can be constructed from the exceptional auto-equivalences which are braided. These are the $\Z{2}$ worth of auto-equivalences of $\cat{g}{2}{4}$, and the $\Z{2}^{\omega(2r+1)-1}$ worth of auto-equivalences of $\cat{so}{2r+1}{2}$. For the former case, we explicitly work out the structure of the quantum subgroup of $\mathfrak{g}_2 $ in the final section of this paper. We find that it is a quantum subgroup whose existence was hypothesised in \cite{MR3339174}. For the latter case, the existence of these exceptional braided auto-equivalences gives a new massive family of quantum subgroups of $\mathfrak{so}_{2r+1}$. The general structure of this massive family is not immediately clear to the author, though preliminary analysis of the low rank cases suggests that the quantum subgroups have a very rich structure. As delving further into describing these quantum subgroups is beyond the scope of this (already large) paper, we leave it for a future publication.

Let us briefly try to outline our methods used to prove Theorem~\ref{thm:main}. As the techniques required to deal with each of the cases of Lie type $A$, $B$, $C$, and $G$ are vastly different, it is hard to give a detailed overview in this introduction. Each of these cases is dealt with in its own separate section of the paper, and a detailed outline of the method used for that case can be found at the beginning of each section.

A very rough outline for the general method is as follows. First we show that the group of gauge auto-equivalences (auto-equivalences which are trivial on objects) for each of the examples is trivial. To show this we argue that such a gauge auto-equivalence would imply the existence of a automorphism of a certain planar algebra associated to each category. We can then directly compute the planar algebra automorphisms for the relevant examples, and deduce the desired result.

Once we have shown that the gauge auto-equivalence group of the category is trivial, we know that the auto-equivalence groups are subgroups of the groups of fusion ring automorphisms of the categories. Fortunately for this paper (in fact, one of the main incentives to begin this paper) is the results of \cite{MR1887583}, which compute the fusion ring automorphisms of all the categories $\cat{g}{}{k}$. Hence our task now reduces to showing which of these automorphisms are realised as either braided, or monoidal auto-equivalences. For this task, a variety of ad-hoc techniques are developed. For the details of these techniques, see the beginning introduction of each section where we outline the construction of the auto-equivalences. Of particular note, we find that not every fusion ring automorphism of $\cat{g}{}{k}$ lifts to an auto-equivalence, a fact we found surprising. Examples of this phenomenon can be seen in the $\cat{so}{2r+1}{2}$ and $\cat{g}{2}{3}$ cases where we have fusion ring automorphisms which are not realised as auto-equivalences of the category.

Ideally the goal is to generalise the main theorem to classify the auto-equivalences for all categories $\cat{g}{}{k}$, i.e. also deduce the cases where $\mathfrak{g}\in \{ \mathfrak{so}_{2r}, \mathfrak{e}_6,\mathfrak{e}_7,\mathfrak{e}_8, \mathfrak{f}_4\}$. There are several problems the author is currently unable to deal with to solve this problem. The first is that there is no planar algebra we can study for the remaining exceptionals in order to leverage information about the gauge auto-equivalences groups. Without knowledge of the gauge auto-equivalence group, it is impossible to determine the auto-equivalences. Hence, either planar algebras for the exceptionals have to be deduced, or new techniques for computing gauge auto-equivalences have to be developed. The issue of realisability is also a problem. For the categories $\cat{so}{2r}{2}$ and $\cat{f}{4}{4}$ there exist exceptional auto-equivalences. The author was unable to find ad-hoc methods for constructing these. We expect that the techniques for constructing the exceptional auto-equivalences of $\cat{so}{2r+1}{2}$ will generalise in some sense to the $\cat{so}{2r}{2}$ case. However, there are fundamental differences between the two that need to be worked out. The exceptional auto-equivalence of $\cat{f}{4}{4}$ remains completely mysterious.

We end the paper with a section discussing potential applications of Theorem~\ref{thm:main}. In our main Theorem, we show the type $A$ categories have the charge conjugation braided auto-equivalence. As this auto-equivalence has order 2, we can potentially gauge this $\Z{2}$ action to obtain a new 2-parameter family of modular tensor categories. A quick obstruction argument shows that this gauging is always possible when $r$ is even, and we suspect that gauging is also possible when $r$ is odd. We leave determining the structure of these new modular tensor categories as a future application, as it goes far beyond the scope of this paper. In type $B$ for certain ranks at level 2 we show the existence of an exceptional auto-equivalence with the peculiar property that it is reverse braided. We observe that the categories with these peculiar auto-equivalences are precisely the categories used in the hypothetical ``modular grafting'' procedure to construct the Izumi-Haagerup doubles. We present the evidence for a connection between reverse braided auto-equivalences, and modular grafting, and end with an open question regarding the relationship between these two concepts. In type $C$, we found the existence of the exceptional level-rank duality auto-equivalence when $r = k$. We identify a new planar algebra which is constructed as fixed points of this auto-equivalence. We identify a potential generator, and give some basic relations. We leave it as an open question to determine a full presentation of this new planar algebra. Finally, for type $G$, we use the exceptional braided auto-equivalence at level 4 to construct a new quantum subgroup. The existence of this quantum subgroup was hypothesised in \cite{MR3339174}.

Finally, in an appendix written by Terry Gannon, the explicit group structure of the auto-equivalences of $\cat{sl}{r+1}{k}$ is determined.

\section*{Acknowledgements}
We are grateful to many people for their assistance throughout this project. We wish to thank Noah Snyder for his help regarding several technical aspects of this paper, and Scott Morrison for many helpful conversations, along with his insight into the modular data of the centre of $\mathcal{Z}(EH)$. We also thank Rob Culling, Cesar Galindo, Terry Gannon, and Andrew Schopieray for various conversations that helped in writing this paper. Finally we thank Pinhas Grossman for comments on an early draft of this paper.

This material is based work supported by the National Science Foundation under Grant No. DMS-1440140 while the author was in residence at the Mathematical Sciences Research Institute in Berkeley, California, during the Spring 2020 semester. In addition, the author gratefully acknowledges support from the AMS-Simons travel grant

\section{Preliminaries}

We direct the reader to \cite{EGNO} for the basics on fusion categories. For the purpose of this paper we will work with fusion categories over the field $\mathbb{C}$.

\subsection*{Modular tensor categories} \hspace{1em}

A \textit{braided fusion category} is a fusion category, along with a collection of natural isomorphisms
\[   c_{X,Y}= \raisebox{-.5\height}{ \includegraphics[scale = .3]{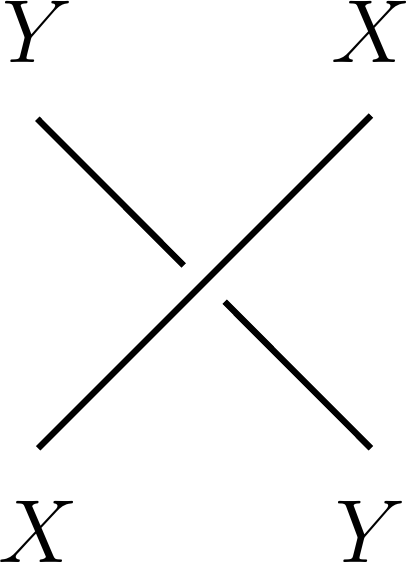}}: X\otimes Y \to Y\otimes X\]
satisfying the topologically implied coherence axioms. 

A \textit{pivotal structure} on a fusion category $\cC$ is a monoidal functor $\psi: ** \to \Id_{\cC}$. We refer to a fusion category along with a pivotal structure as a \textit{pivotal fusion category}. When using graphical calculus for pivotal fusion categories, we will often suppress $\psi$, as it makes for much cleaner pictures, and it is always clear where the component of $\psi$ should belong in a given picture.

Given a pivotal fusion category $\cC$ we can define traces on the endomorphism algebras of $\cC$. Let $f \in \Hom(X,X)$, we have 
\begin{align*} 
 \operatorname{Tr}_+(f) &:=  \raisebox{-.5\height}{ \includegraphics[scale = .3]{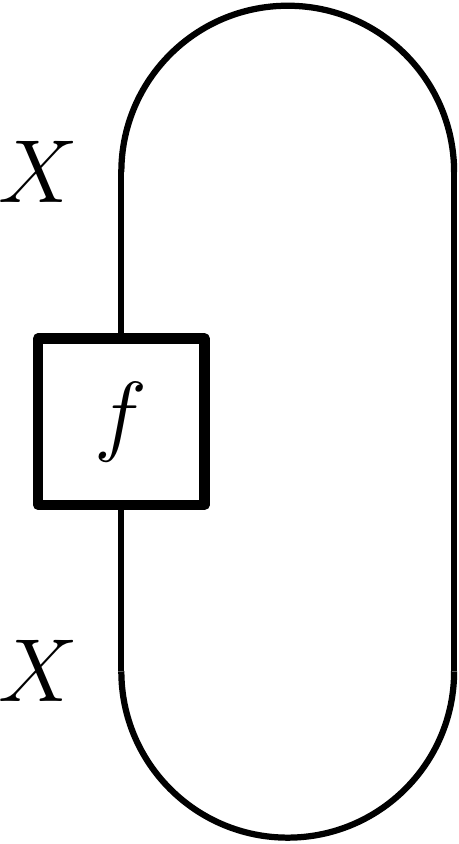}}=   \operatorname{ev}_X \circ (f \otimes \id_{X^*})\circ (\psi_X\otimes \id_{X^*} )\circ \operatorname{coev}_{X^*},\\
\operatorname{Tr}_-(f) &:=  \raisebox{-.5\height}{ \includegraphics[scale = .3]{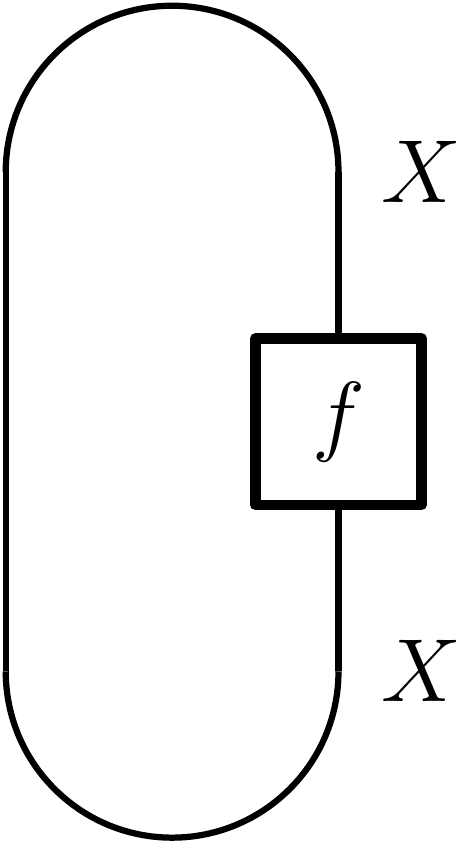}}=\operatorname{ev}_{X^*} \circ ( \id_{X^*}\otimes f )\circ ( \id_{X^*} \otimes \psi_X^{-1} )\circ \operatorname{coev}_{X}.
\end{align*}
We say the pivotal fusion category $\cC$ is \textit{spherical} if $ \operatorname{Tr}_+(f) =  \operatorname{Tr}_-(f)$ for all $f$. For a spherical fusion category, we define the categorical dimension of an object $X$ as
\[ \operatorname{dim}(X) =  \operatorname{Tr}_+(\id_X) =  \operatorname{Tr}_-(\id_X).\]

Given a spherical braided fusion category we can construct two numerical invariants. These are the $S$ and $T$ matrices, defined by
\begin{align*}
   S_{X,Y} \quad &:= \quad \operatorname{Tr}_- \otimes  \operatorname{Tr}_+ (c_{Y,X}\circ c_{X,Y}) \quad  =\quad  \raisebox{-.5\height}{ \includegraphics[scale = .5]{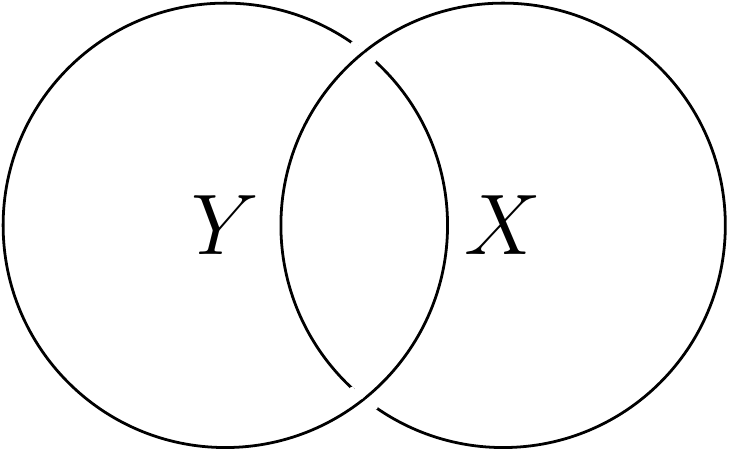}} \\
           \quad T_{X,X}\quad &:= \quad    \frac{  \operatorname{Tr}_+(c_{X,X}) }{  \operatorname{Tr}_+(\id_X)} \quad =\quad \frac{\raisebox{-.5\height}{ \includegraphics[scale = .5]{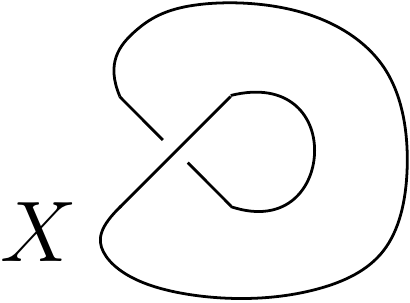}}}{\raisebox{-.5\height}{ \includegraphics[scale = .5]{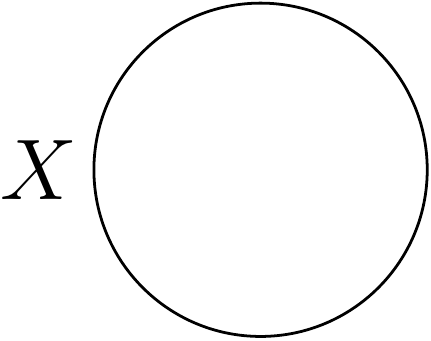}}}  . 
           \end{align*}
If the $S$ matrix of $\cC$ is invertible, we say that $\cC$ is a \textit{modular tensor category}. The main objects of study in this paper will be a certain class of modular tensor categories that we will introduce shortly.

\vspace{3em}

\subsection*{Modular tensor categories from Lie theory} \hspace{1em}

Here we give a brief recap on the modular categories $\cC(\mathfrak{g} , k)$, the main objects of study in this paper. Surprising little knowledge of these categories is required for this paper, so we keep the details to a minimum. For additional details we point the reader to \cite{quan-primer} for an expository explanation of the modular categories $\cC(\mathfrak{g} , k)$, and to \cite{MR1328736,MR1066560} for a detailed exposition.  We mainly restrict our attention to the case when $\mathfrak{g}$ is one of $\mathfrak{sl}_{r+1}$, $\mathfrak{so}_{2r+1}$, $\mathfrak{so}_{2r+1}$, or $\mathfrak{g}_2$, as these are the Lie algebras we deal with in this paper.

Let $\mathfrak{g}$ a simple Lie algebra, $\hat{\mathfrak{g}}$ its corresponding affine Lie algebra, and $k$ a positive integer. The modular category $\cC(\mathfrak{g} , k)$ consists of the level $k$ integrable highest weight modules of $\hat{\mathfrak{g}}$. Let $r$ be the rank of $\mathfrak{g}$, then the simple objects of $\cC(\mathfrak{g} , k)$ are parametrised by the set:
\[ P_+:= \left \{ 	\quad \sum_{j=1}^r  \lambda_j\Lambda_j \quad | \quad  \lambda_j \in \mathbb{N} , \quad  \sum_{j=1}^r \lambda_j a^\vee_j \leq k \right\}. \]
Here $a^\vee_j$ are the \textit{co-labels} of $\mathfrak{g}$, given in Table~\ref{tab:colab} for the relevant Lie algebras in this paper.
\begin{table}[h!]
   
\centering 
    
    \begin{tabular}{c | c }
    	\toprule
			$\mathfrak{g}  $ & $a^\vee_j$ 					   \\
			\midrule
			$\mathfrak{sl}_{r+1}$ & 1 \\
			$\mathfrak{so}_{2r+1}$   & $\begin{cases}
                                  1  & \text{ if } j \in \{1,r\} \\
 								 2  & \text{otherwise}
  \end{cases}$ \\
  $\mathfrak{sp}_{2r+1}$  &      1 \\
  $\mathfrak{g}_2$   & $\begin{cases}
                                  1  & \text{ if } j = 1 \\
 								 2  & \text{ if } j = 2
  \end{cases}$ \\
			          	                 
    	\bottomrule

	    \end{tabular}
	    	\caption{\label{tab:colab}  Co-labels of the Lie algebras $\mathfrak{sl}_{r+1}$, $\mathfrak{so}_{2r+1}$, $\mathfrak{sp}_{2r}$, and $\mathfrak{g}_{2}$}
\end{table}
\newline The fusion rules for the categories $\cC(\mathfrak{g} , k)$ can be computed using the Weyl chamber of $\mathfrak{g}$. The details of this procedure are complicated, so we direct the reader to \cite[Section 5]{quan-primer} for the details. We will recall relevant fusion rules of the categories $\cC(\mathfrak{g} , k)$ as needed throughout the paper. While in theory it is possible to compute the associator for the categories $\cC(\mathfrak{g} , k)$, for this paper we avoid using the associators almost completely. 

The categories $\cC(\mathfrak{g} , k)$ have several choices for pivotal structures, indexed by characters of the universal grading group of the category $\cC(\mathfrak{g} , k)$. For this paper we use the standard Hopf algebra pivotal structure on the categories $\cat{so}{2r+1}{k}$ and $\cat{g}{2}{k}$, and Tureav's unimodal pivotal structure on the categories $\cat{sl}{r+1}{k}$ and $\cat{sp}{2r}{k}$. Importantly for us, this ensures that the object $\Lambda_1$ satisfies certain conditions which allow us to use powerful planar algebra machinery. With these choices of pivotal structures we have that the quantum twists of the simple objects in the categories $\cat{so}{2r+1}{k}$, $\cat{sp}{2r}{k}$ and $\cat{g}{2}{k}$ are given by the formulas:
\begin{align*} 
\cat{sl}{r+1}{k}  : T_{ \sum_{ j=1}^r \lambda_j\Lambda_j } &=  \left( (-1)^{ \sum_{j=1}^r j\lambda_j}\right) e^{\frac{2 \pi i}{2 (1 + k + r)}(\lambda_1  , \lambda_2  ,\cdots , \lambda_r).K.(\lambda_1  +2 , \lambda_2 +2 ,\cdots , \lambda_r + 2)^T} \\   
\cat{so}{2r+1}{k}  : T_{ \sum_{ j=1}^r \lambda_j\Lambda_j } &= e^{\frac{2 \pi i}{4(2r-1 + k)}(\lambda_1  , \lambda_2  ,\cdots , \lambda_r).K.(\lambda_1  +2 , \lambda_2 +2 ,\cdots , \lambda_r + 2)^T} \\
\cat{sp}{2r}{k}  : T_{ \sum_{ j=1}^r \lambda_j\Lambda_j} &= \left( (-1)^{ \sum_{j=1 ,3,\cdots}^r \lambda_j}\right)  \hspace{.1cm} e^{\frac{2 \pi i}{4(r+k  + 1)}(\lambda_1  , \lambda_2  ,\cdots , \lambda_r).K.(\lambda_1  +2 , \lambda_2 +2 ,\cdots , \lambda_r + 2)^T} \\ 
\cat{g}{2}{k}  :  T_{\lambda_1\Lambda_1 + \lambda_2\Lambda_2} &=  e^{\frac{2 \pi i}{6(4+k)}(\lambda_1  , \lambda_2).K.(\lambda_1  +2 , \lambda_2 +2)^T}, 
\end{align*}
where the matrices $K$ are Killing matrices of $\mathfrak{g}$, given in Table~\ref{tab:cotwist} for the relevant Lie algebras for this paper.

\begin{table}[h!]
   
\centering 
    
    \begin{tabular}{c | c }
    	\toprule
			$\mathfrak{g}  $ & $K$ 					   \\
			\midrule
			$\mathfrak{sl}_{r+1}$   & $ \left[ \frac{(r+1)\min(i,j) - ij}{r+1}   \right]_{i,j} $ \\
			$\mathfrak{so}_{2r+1}$   & $ \left[ \frac{2\min(i,j)}{2^{(\delta_{i,r}+\delta_{j,r})}} \right]_{i,j} $ \\
  $\mathfrak{sp}_{2r}$   &    $ \left[ \min(i,j) \right]_{i,j }$ \\
  $\mathfrak{g}_2$   & $\begin{bmatrix}
   2      & 3  \\
   3		 & 6  
\end{bmatrix}$ \\
			          	                 
    	\bottomrule

	    \end{tabular}
	    	\caption{\label{tab:cotwist}  The Killing matrices for the Lie algberas $\mathfrak{sl}_{  r+1}$, $\mathfrak{so}_{2r+1}$, $\mathfrak{sp}_{2r}$, and $\mathfrak{g}_{2}$.}
\end{table}

\vspace{1em}

\subsection*{Monoidal Auto-equivalences} \hspace{1em}

Let $\cC$ be a fusion category. We say  an auto-equivalence $\cF:\cC \to \cC$ is a monoidal auto-equivalence of $\cC$ if there exists a family of natural isomorphisms
\[   \tau_{X,Y}: \cF(X)\otimes\cF(Y) \to \cF(X\otimes Y)\]
satisfying a straightforward coherence condition. We call the collection $\tau$ a \textit{tensorator} for $\cC$.

If $\cC$ is a braided fusion category, then we say $(\cF,\tau)$ is a braided auto-equivalence of $\cC$ if it is a monoidal auto-equivalence, and the following diagram commutes for all $X,Y \in \cC$.
\[\begin{CD}
\cF(X) \otimes \cF(Y) @> c_{\cF(X),\cF(Y)}>>\cF(Y) \otimes \cF(X)\\
@V \tau_{X,Y} VV @VV\tau_{Y,X} V\\
\cF(X\otimes Y)@> \cF(c_{X,Y}) >> \cF(Y\otimes X)
\end{CD} \]

 If $\cC$ is a pivotal fusion category, then we say $(\mathcal{F},\tau)$ is a \textit{pivotal auto-equivalence} if
\[  \delta_{X^*} \circ \mathcal{F}(\psi_X) = \delta_X^* \circ \psi_{\mathcal{F}(X)},\]
where 
\[  \delta_X := (( \mathcal{F}( \operatorname{ev}_X) \circ \tau_{X^*,X}) \otimes \id_{\mathcal{F}(X)^*}) \circ ( \id_{\mathcal{F}(X^*)} \otimes \operatorname{coev}_{\mathcal{F}(X)}). \]

Given a fusion category $\cC$ we write $\underline{\TenAut}(\cC)$ for the monoidal category whose objects are monoidal auto-equivalences of $\cC$, and whose morphisms are monoidal natural transformations. We can truncate this monoidal category to get $\TenAut(\cC)$, the group of monoidal auto-equivalences of $\cC$. In a similar fashion if $\cC$ is a braided fusion category we can define $\underline{\BrAut}(\cC)$ to be the monoidal category of braided auto-equivalences of $\cC$, and $\BrAut(\cC)$, the group of braided auto-equivalences of $\cC$. 

Let $K(\cC)$ be the fusion ring of $\cC$, i.e. the based ring whose objects are isomorphism classes of objects in $\cC$. A \textit{fusion ring automorphism of $\cC$} is a based ring automorphism of $K(\cC)$. We write $\operatorname{FusEq}(\cC)$ for the group of fusion ring automorphisms of $\cC$.  

 Given any monoidal auto-equivalence of $\cC$ we can forget the structure of the tensorator to get a fusion ring automorphism of $\cC$, hence we get a map
\[    \TenAut(\cC) \to  \operatorname{FusEq}(\cC).\]
In general this map is neither injective or surjective. Let us write $\widehat{\operatorname{FusEq}(\cC)}$ for the image of this homomorphism, i.e. the group of fusion ring automorphisms of $\cC$ that are realised by a monoidal auto-equivalence. We call elements of $\widehat{\operatorname{FusEq}(\cC)}$ \textit{realisable fusion ring automorphisms of $\cC$}. As the map $\TenAut(\cC) \to  \widehat{\operatorname{FusEq}(\cC)}$ is surjective we get a short exact sequence
\[     0 \to \text{Gauge}(\cC) \to \TenAut(\cC) \to \widehat{\operatorname{FusEq}(\cC)}\to 0.\]
Here $ \text{Gauge}(\cC)$ is the kernel of the map $\TenAut(\cC) \to  \widehat{\operatorname{FusEq}(\cC)}$. This kernel corresponds to the group of monoidal auto-equivalences of $\cC$ whose underlying functor is the identity i.e the trivial functor equipped with possibly interesting tensorators. We call such auto-equivalences of $\cC$, \textit{gauge auto-equivalences}.

With the above short exact sequence in mind, we see that the process of computing $\TenAut(\cC)$ for a given $\cC$ splits into three parts:
\begin{enumerate}
\item Compute the fusion ring automorphisms of $\cC$.
\item Determine which fusion ring automorphisms of $\cC$ are realisable.
\item Compute the gauge auto-equivalences of $\cC$.
\end{enumerate}
Step (1) is purely combinatorial, and is simply an exercise in ring theory given that the fusion ring of $\cC$ is known. Step (2) is more categorical, as it asks when one can equip a fusion ring automorphism of $\cC$ with a tensorator. There are many invariants one can check to show that a fusion ring automorphism can not be equipped with a tensorator, however there is no general purpose solution for the converse problem. Our tools in this paper are a ad-hoc mix of direct tensorator computations, Hopf algebra methods, and planar algebra arguments. Step (3) is purely categorical, and one can think of the group $\text{Gauge}(\cC)$ as telling us how many distinct tensorators exist for a fusion ring automorphism, given one already exists to begin with. The standard approach to computing $\text{Gauge}(\cC)$ is to use planar algebra techniques. Key to applying these planar algebra techniques will be the following Lemma.

\begin{lemma}\label{lem:gagpiv}
Let $\cC$ a pivotal category, and $(\Id_\cC, \tau)$ a gauge auto-equivalence of $\cC$, then $(\Id_\cC, \tau)$ is a pivotal functor.
\end{lemma}
\begin{proof}
From the definition of a pivotal functor, the statement of this Lemma is equivalent to showing $\delta_{X^*} = \delta^*_X$ for all simple $X$. We compute 
\begin{equation}\label{eq:gag}
\delta^*_X  = \raisebox{-.5\height}{ \includegraphics[scale = .4]{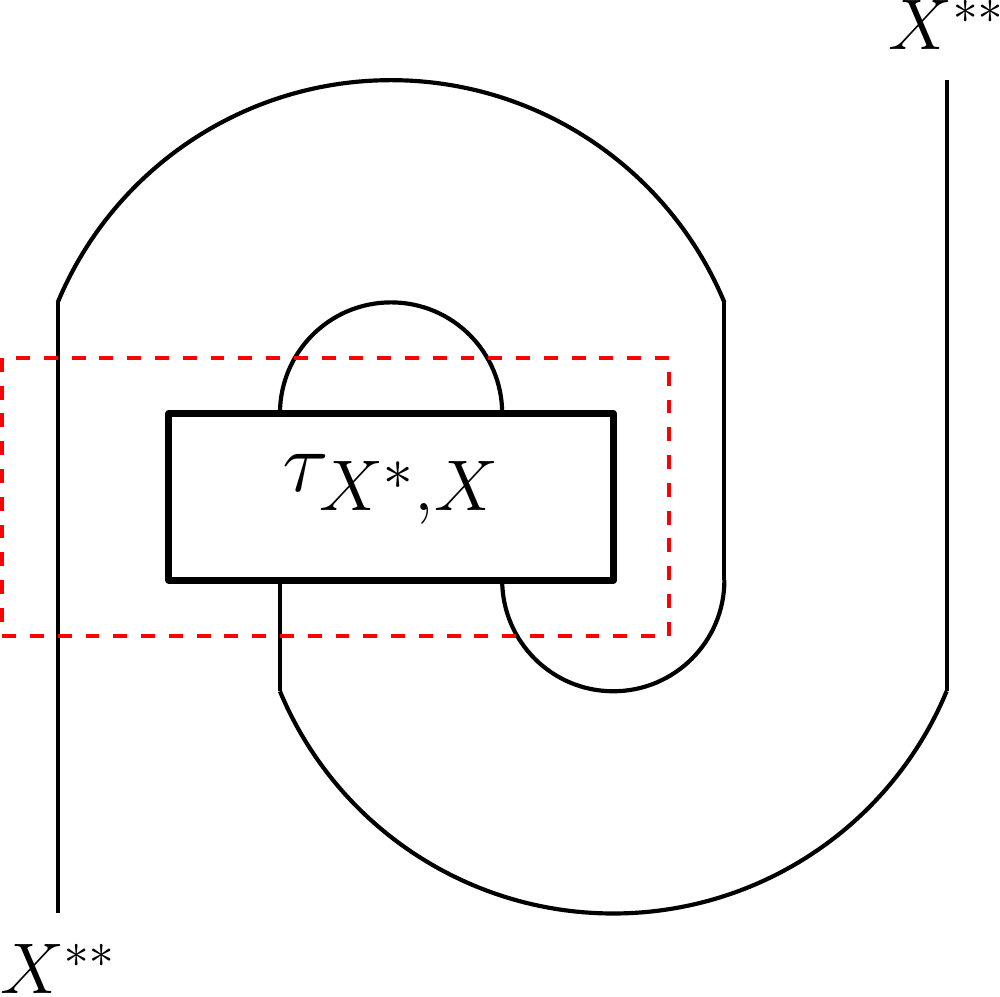}} = \raisebox{-.5\height}{ \includegraphics[scale = .4]{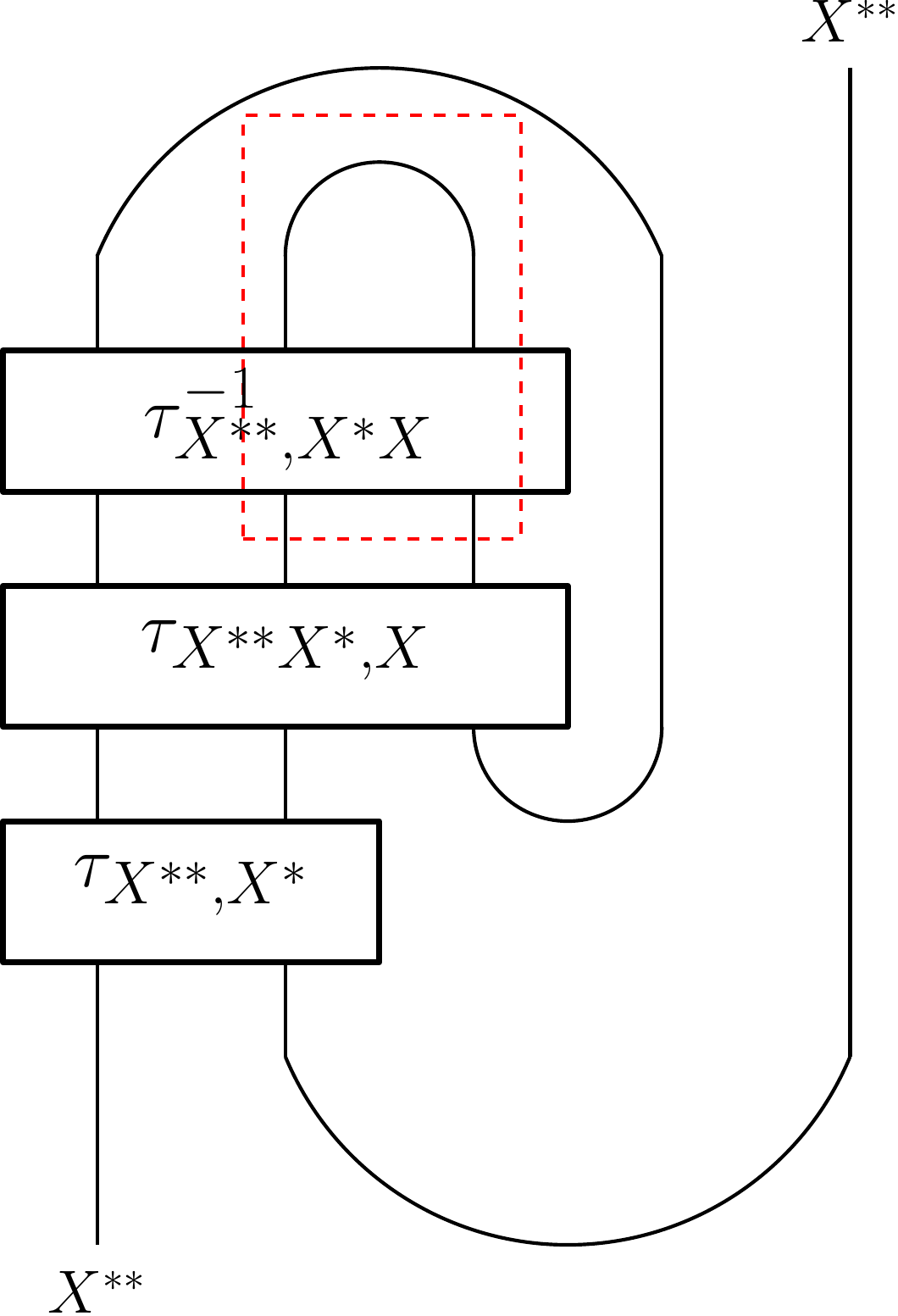}} = \raisebox{-.5\height}{ \includegraphics[scale = .4]{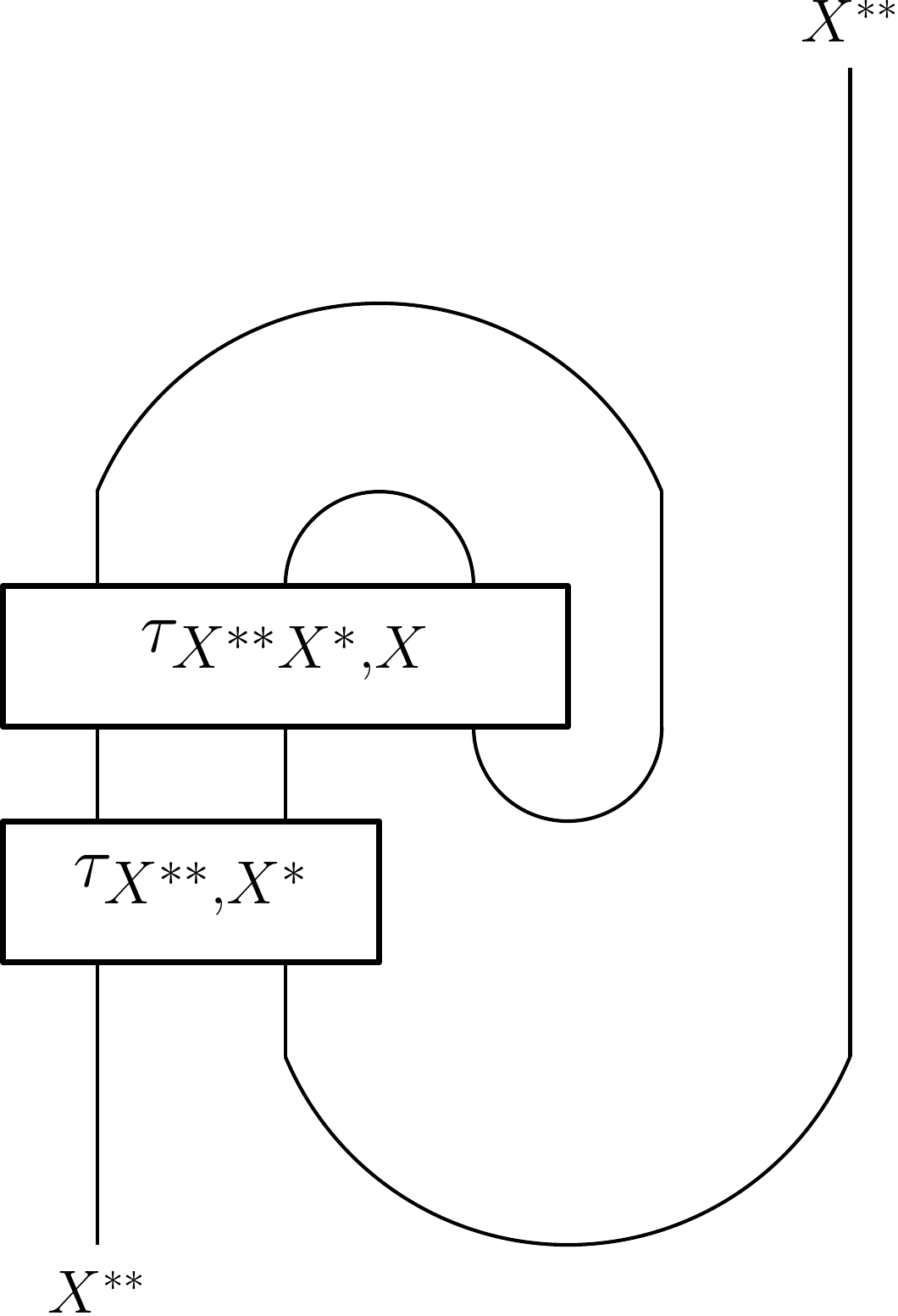}}
\end{equation}

Let us separately consider the \textit{koru} in the last diagram. As $X$ is simple, the hom space $\Hom(X^{**} \otimes X^{*} \to \mathbf{1})$ is 1-dimensional. Thus there exists a scaler $\alpha$ such that 
\[ \raisebox{-.5\height}{ \includegraphics[scale = .4]{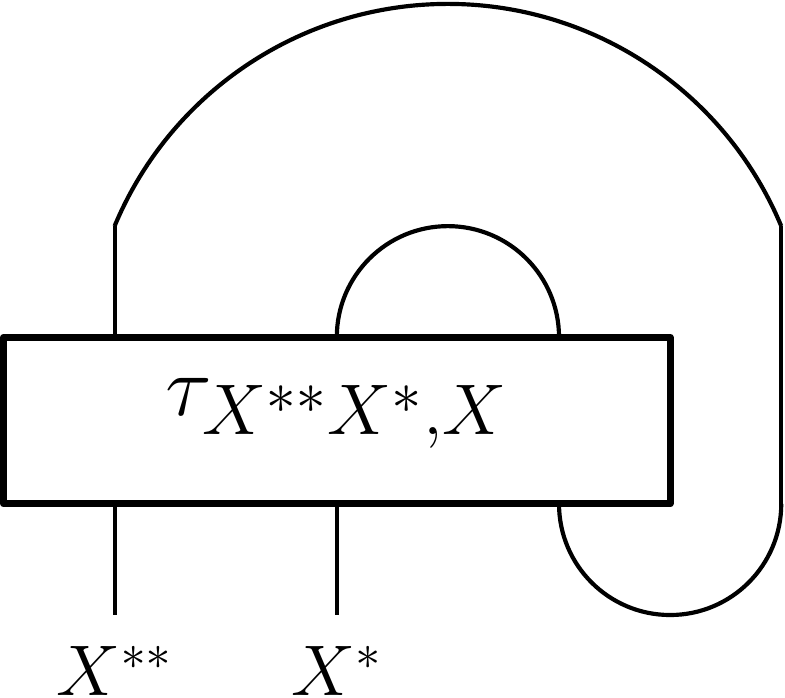}} = \alpha  \raisebox{-.5\height}{ \includegraphics[scale = .4]{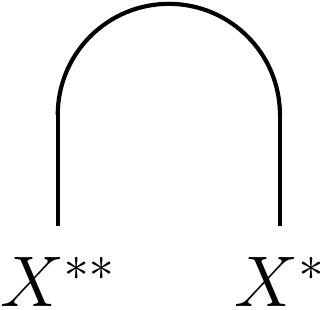}}.\]
Appending the morphism 
\[  \raisebox{-.5\height}{ \includegraphics[scale = .4]{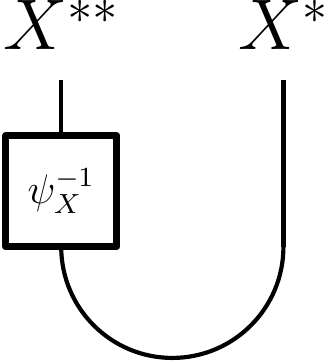}}\]
to both sides gives $\alpha \cdot \dim(X)$ on the right, and $\dim(X)$ on the left (after an application of the naturality of $\tau$). Hence, $\alpha = 1$.

Returning to Equation~\eqref{eq:gag} we thus have that 
\[\raisebox{-.5\height}{ \includegraphics[scale = .4]{gauge4}} = \raisebox{-.5\height}{ \includegraphics[scale = .4]{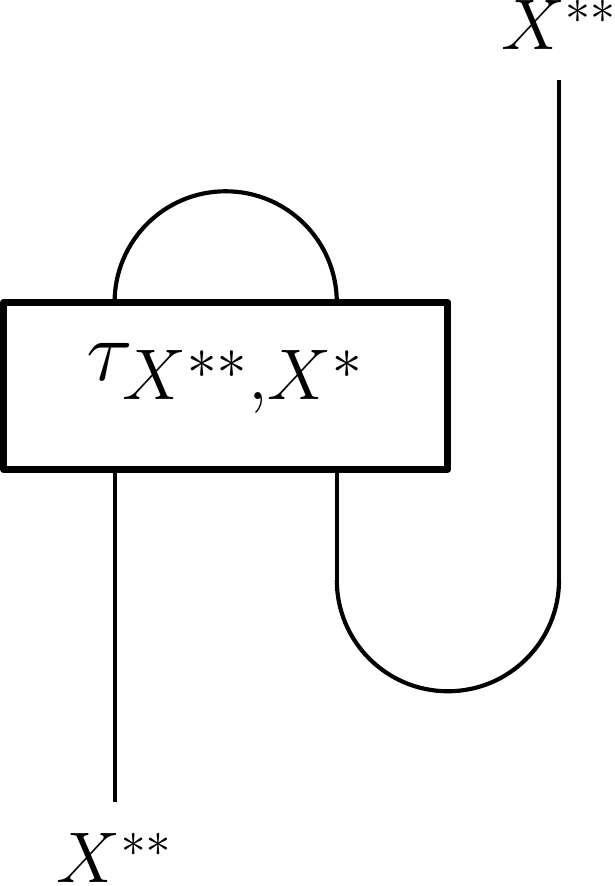}} = \delta_{X^*},\] 
completing the proof.
\end{proof}

This Lemma shows us that every gauge auto-equivalence must appear as a planar algebra automorphism of some planar algebra corresponding to $\cC$. We discuss this more in the planar algebra section of these preliminaries.

Another useful Lemma regarding auto-equivalences is as follows.

\begin{lemma}
Suppose $\cC$ is a spherical braided fusion category, and $(\cF,\tau)$ is a braided auto-equivalence of $\cC$ that preserves categorical dimensions. Then 
\[  T_{X,X} = T_{\cF(X) , \cF(X)}.\]
\end{lemma}
\begin{proof}
Let $X \in \cC$ be a simple object. We have
\begin{align*}
\cF( \operatorname{ev}_X)\circ \tau_{X^*,X}\circ(\delta_X^{-1} \otimes \id_{\cF(X)}) &\in \operatorname{Hom}(\cF(X)^* \otimes \cF(X) \to \mathbf{1}) &\ni  \operatorname{ev}_{\cF(X)}, \\
 ( \id_{\cF(X)^*} \otimes \delta_{X}^{*-1})\circ ( \delta_X \otimes \delta_{X^*})\circ \tau_{X^*,X^{**}}^{-1} \circ \cF( \operatorname{coev}_{X^*}) &\in \operatorname{Hom}(\mathbf{1}\to \cF(X)^* \otimes \cF(X)^{**}) &\ni \operatorname{coev}_{\cF(X)^*} , \\
\cF( \psi_X) \circ \delta_{X^*}^{-1} \circ \delta_X^*&\in \operatorname{Hom}(\cF(X) \to \cF(X)^{**}) &\ni   \psi_{\cF(X)}.
\end{align*}
As $X$ is simple, the Hom spaces above are 1-dimensional. Thus there exist scalers $\alpha, \beta, \gamma$ such that
\begin{align*}
\cF( \operatorname{ev}_X)\circ \tau_{X^*,X}\circ(\delta_X^{-1} \otimes \id_{\cF(X)}) & =    \alpha  \operatorname{ev}_{\cF(X)}, \\
 ( \id_{\cF(X)^*} \otimes \delta_{X}^{*-1})\circ ( \delta_X \otimes \delta_{X^*})\circ \tau_{X^*,X^{**}}^{-1} \circ \cF( \operatorname{coev}_{X^*}) &=    \beta  \operatorname{coev}_{\cF(X)^*}, \\
\cF( \psi_X) \circ \delta_{X^*}^{-1} \circ \delta_X^* &=    \gamma  \psi_{\cF(X)}.
\end{align*}
As the functor $(\cF,\tau)$ preserves categorical dimensions, we have in particular that $\cF(\dim(X)) = \dim({\cF(X)})$. Expanding out the term $\cF(\dim(X))$ reveals that the product $\alpha \beta \gamma$ is equal to $1$.

Now consider the term $\cF(T_{X,X})$. Expanding this gives
\[\cF(T_{X,X}) = \alpha \beta \gamma T_{\cF(X), \cF(X)} =  T_{\cF(X), \cF(X)},\]
 proving the statement of the Lemma.
\end{proof}

As the categorical dimensions of the objects in the categories $\cat{g}{}{k}$ agree with their Frobenius-Perron dimensions, any auto-equivalence will automatically preserve categorical dimensions. Hence we get the following Corollary.

\begin{cor}
Let $(\cF,\tau)$ a braided auto-equivalence of $\cat{g}{}{k}$, then $(\cF,\tau)$ preserves the twists of the simple objects.
\end{cor}

This gives a useful obstruction for the existence of a braided auto-equivalence of $\cat{g}{}{k}$. In fact, for the examples we consider in this paper, we find that a fusion ring automorphism of $\cat{g}{}{k}$ lifts to a braided auto-equivalence precisely when it preserves the twists of the simple objects. This phenomenon is not true for arbitrary braided fusion categories, and we only show this to be true for our examples as a corollary of our main classification theorem.

\vspace{1em}

A useful construction of monoidal auto-equivalences is through the process of simple currents. This process takes an invertible object in a modular tensor category $\cC$, and constructs a monoidal auto-equivalence. This process was first investigated for fusion rings in the language of conformal field theory in \cite{MR1065268}. For modular tensor categories, the process was studied in \cite{1711.00645,1902.09498}. The construction given in \cite{1902.09498} is somewhat awkward to use for the examples considered in this paper. Instead we present the following construction, which can be proved in much the same fashion as \cite{1902.09498}.

\begin{lemma}\label{lem:simplecurrent}
Let $\cC$ be a modular tensor category, and $g$ an invertible object of order $M$. Set $q$ equal to the unique integer (modulo $2M$) such that 
\[    \sigma_{g,g} = e^{2\pi i \frac{q}{2M}}\id_{g\otimes g},\]
and choose $a \in \{0,1,\cdots , M-1\}$ such that
\[  1 + aq \quad \text{ is coprime to } M.\]
Then there exists a monoidal auto-equivalence $\cF_{g,a}$ of $\cC$ defined on objects by 
\[   \cF_{g,a}(X) = g^{-an}\otimes X , \]
where $n$ is the unique integer (modulo $M$) such that $\sigma_{X,g}\sigma_{g,X} = e^{2\pi i \frac{n}{M}}\id_{g\otimes X}$. The monoidal auto-equivalence $\cF_{g,a}$ is braided if and only if 
\[   a + \frac{a^2q}{2} \equiv 0 \pmod M.\]
\end{lemma}
\begin{proof}
Nearly identical to the proof of \cite{1902.09498}.
\end{proof}

\vspace{3em}

\subsection*{Hopf algebras} \hspace{1em}

A useful tool for constructing auto-equivalences of tensor categories is by studying the automorphisms of an associated Hopf algebra. Here we define Hopf algebras, and explain how symmetries of certain Hopf algebras give rise to symmetries of certain tensor categories. Our motivation for introducing these techniques is to construct the \textit{charge conjugation} auto-equivalences of $\cC(\mathfrak{sl}_n , k)$.

A Hopf algebra $H$ is an associative algebra, along with a coassociative colagebra structure $\Delta: H \to H\otimes H$, a counit map $\varepsilon: H \to \mathbb{C}$, and an antipode map $S: H\to H$, all of which must satisfy certain conditions. Further, we say that the Hopf algebra $H$ is quasi-triangular if there exists an element $R \in H\otimes H$ satisfying the Yang-Baxtor equation. Full definitions can be found in \cite{MR1243637}. We are intentionally light on the definition of a Hopf algebra, as we are interested in specific concrete examples coming from deformed Lie algebras. These are the so called quantum groups in the sense of Drinfeld \cite{MR934283}.

\begin{dfn}\cite{MR1321145}
Let $\mathfrak{g}$ be a simple Lie algebra with simple roots $\Delta$, and let $A$ the Cartan matrix of $\mathfrak{g}$. The quasi-triangular Hopf algebra $U_q(\mathfrak{g})$ is defined by the generators: 
\[   X_i, Y_i, H_i   \text{        for  } i \in \Delta\]
These generators satisfy the relations:
\begin{align*}
[ H_i , H_j] &= 0 \\
[ X_i , X_j] &=  \delta_{i,j}\frac{ \sinh( qd_iH_i/2)}{\sinh(qd_i/2)} \\
[ H_i , X_j] &=  A_{i,j} X_j \\
[ X_i , Y_j] &= -A_{i,j} Y_j  \\
\sum_{k=0}^{1 - A_{i,j}} (-1)^k \left[   \begin{smallmatrix}
1 - A_{i,j}\\
k
\end{smallmatrix}
 \right]_{h_i} X_i^kX_jX_i^{1-A_{i,j} - k} &= 0 \text{ for } i\neq j\\
 \sum_{k=0}^{1 - A_{i,j}} (-1)^k \left[   \begin{smallmatrix}
1 - A_{i,j}\\
k
\end{smallmatrix}
 \right]_{h_i} Y_i^kY_jY_i^{1-A_{i,j} - k} &= 0 \text{ for } i\neq j
\end{align*}

The comultiplication $\Delta: U_q(\mathfrak{g}) \to U_q(\mathfrak{g})\otimes U_q(\mathfrak{g})$ is defined by
\begin{align*}
\Delta(H_i) &= H_i \otimes 1 + 1 \otimes H_i, \\
\Delta(X_i) &= X_i\otimes e^{qd_iH_i/4} + e^{-qd_iH_i/4}\otimes X_i, \\
\Delta(Y_i) &= Y_i\otimes e^{qd_iH_i/4} + e^{-qd_iH_i/4}\otimes Y_i. \\
\end{align*}

The conunit $\varepsilon : U_q(\mathfrak{g}) \to \mathbb{C}$ is defined by 
\[     \varepsilon(H_i) =  \varepsilon(X_i) = \varepsilon(Y_i) = 0.\]

The antipode $S :  U_q(\mathfrak{g}) \to U_q(\mathfrak{g})$ is defined by 
\[  S(H_i) = - H_i \quad S(X_i) = -e^{hd_i/2}X_i \quad S(Y_i) = -e^{hd_i/2}Y_i.\]

This Hopf algebra has an $R$-matrix, however we neglect to write it down.
\end{dfn} 

The representation theory of $U_q(\mathfrak{g})$ gives a braided tensor category $\operatorname{Rep}(U_q(\mathfrak{g}))$. We will be interested in the special case of when $q$ is a root of unity. In this case the category $\operatorname{Rep}(U_q(\mathfrak{g}))$ is not semisimple. However, we can take the semi-simplification (as in \cite{1801.04409}) to obtain the braided fusion category $\overline{\operatorname{Rep}(U_q(\mathfrak{g}))}$. 

We now have two ways of constructing a braided fusion category from a simple Lie algebra $\mathfrak{g}$. We can either form $\cC(\mathfrak{g} , k)$ or $\overline{\operatorname{Rep}(U_q(\mathfrak{g}))}$. A result of Finkelberg \cite{MR1384612} shows that these two constructions are essentially the same, by providing the following braided equivalence of categories:
\[   \cC(\mathfrak{g} , k) \simeq  \overline{\operatorname{Rep}\left(U_{e^{\frac{\pi i }{m(k +\check{h})}}}(\mathfrak{g})\right)}\]
where $\check{h}$ is the dual coxeter number of $\mathfrak{g}$, and $m= 1$ for the Lie algebras $ADE$, $m=2$ for the Lie algebras $BCF$, and $m=3$ for the Lie algebra $G$. We remark that there are four exceptions to the above equivalence, which are $(\mathfrak{g}, k ) = (E_6, 1) , (E_7, 1), (E_8, 1)$, or $(E_8, 2)$. These four exceptions will not feature in the results or proofs of this paper.

The above equivalence will be useful to us, as it will allow us to perform certain computations regarding the categories $\cC(\mathfrak{g} , k)$ in the Hopf algebra $U_q(\mathfrak{g})$. In particular we have that Hopf algebra automorphisms (see \cite{MR1243637}) will induce monoidal and braided auto-equivalences of the category $\operatorname{Rep}(U_q(\mathfrak{g}))$, and hence also of $\overline{\operatorname{Rep}(U_q(\mathfrak{g}))}$. To make this precise, we have the following naturality result of Hopf algebras.

\begin{theorem}
Let $H$ be a quasi-triangular Hopf algebra, and let $\Aut(H)$ be the group of Hopf algebra automorphisms of $H$ that preserve the $R$-matrix. Then there is an injection
\[  \Aut(H) \to \BrAut(\operatorname{Rep}(H)).\]
\end{theorem}

Applying this Theorem to the quasi-triangular Hopf algebras $U_q(\mathfrak{g})$ gives the following Corollary.

\begin{cor}\label{cor:cc}
Let $\mathfrak{g}$ be a simple Lie algebra, with Dynkin diagram $\Gamma_\mathfrak{g}$. Then there is an injection
\[  \Aut(\Gamma_\mathfrak{g} ) \to  \BrAut(\overline{\operatorname{Rep}(U_q(\mathfrak{g}))}).\]
\end{cor}
\begin{proof}

Let $\Delta$ be the set of simple roots of $\mathfrak{g}$. A symmetry of $\Gamma_\mathfrak{g}$ induces a symmetry $\sigma$ of $\Delta$ which preserves the Cartan matrix of $\mathfrak{g}$, i.e. $A_{i,j} = A_{\sigma(i) , \sigma(j)}$. Using the generators and relations presentation of $U_q(\mathfrak{g})$ we can see that $\sigma$ induces a Hopf algebra automorphism $\phi_\sigma: U_q(\mathfrak{g}) \to U_q(\mathfrak{g})$ defined by
\begin{align*}
\phi_\sigma(H_i) :=& H_{\sigma(i)} \\
\phi_\sigma(X_i^{\pm}) :=& X_{\sigma(i)}^{\pm}. \\
\end{align*}

We still have to show that $\phi_\sigma$ preserves the $R$-matrix of $U_q(\mathfrak{g})$. To do this we will use the results of \cite{MR2506324}. For this next portion of the proof, we will freely use the notation and conventions of this cited paper.

Following \cite{MR2506324} we can write the $R$-matrix of $U_q(\mathfrak{g})$ in the form $\widehat{R}\cdot K$, where $\widehat{R}\in T_q(b_+ \otimes b_-)$, and $K = q^{ \sum_{i,j \in \Delta} A^{-1}_{i,j}(H_i \otimes H_j )}$. Clearly $\phi_\sigma$ preserves $K$, and $\phi_\sigma( \widehat{R} )$ is again an element of $T_q(b_+ \otimes b_-)$ (as $\phi_\sigma$ preserves the signs of the generators). Hence $\phi_\sigma(R)$ is of the form $\widehat{R}'\cdot K$ for $R' \in T_q(b_+ \otimes b_-)$. As $\phi_\sigma$ preserves the comultiplication of $U_q(\mathfrak{g})$, we can apply \cite[Theorem 7.1]{MR2506324} to see that $R$ and $R'$ only differ by a scalar. From the standard properties of the $R$-matrix, we can further deduce that this scalar is the identity. Hence $\phi_\sigma(R) = R$. Hence we get an injection  

\[  \Aut(\Gamma_\mathfrak{g} ) \to  \BrAut(\operatorname{Rep}(U_q(\mathfrak{g}))).\]

Finally, any auto-equivalence of $\operatorname{Rep}(U_q(\mathfrak{g}))$ will preserve the negligible ideal. Thus the above map descends to an injection  
\[  \Aut(\Gamma_\mathfrak{g} ) \to  \BrAut(\overline{\operatorname{Rep}(U_q(\mathfrak{g}))}).\]

\end{proof}

The main purpose of this Corollary is that it gives the construction of the \textit{charge conjugation} auto-equivalences of $\cC(\mathfrak{sl}_n , k)$. These auto-equivalences occur due to the $\Z{2}$-symmetry in the Dynkin diagram for $\mathfrak{sl}_n$. 

\vspace{3em} 

\subsection*{Planar algebras} \hspace{1em}

Roughly speaking a planar algebra $P$ is a collection of vector spaces $\{P_n : n\in \mathbb{N} \}$, along with a multi-linear action of planar tangles. The full definition is fairly involved, so we point the reader towards \cite{math.QA/9909027} for the full definition. Planar algebras have appeared in many different contexts, and under many different names. Slightly different axiomitisations for the general idea of a planar algebra include monoidal algebras \cite{MR2132671}, spiders \cite{MR1403861}, and towers of algebras \cite{MR999799}.

It is well known that there is a one to one correspondence between planar algebras, and pivotal fusion categories with a distinguished $\otimes$-generating symmetrically self-dual object. Recently it was shown in that this correspondence is functorial. This functorial relationship is very useful for this paper, as it will allow us to construct and classify certain auto-equivalences of pivotal fusion categories by studying planar algebra automorphisms of the corresponding planar algebra. 

Let us describe this functorial relationship between planar algebras and pivotal categories. 

On one hand we have the category of based pivotal (braided) fusion categories. The objects of this category are pairs $(\cC,X)$, where $\cC$ is a pivotal (braided) category, and $X$ is a $\otimes$-generating symmetrically self-dual generating object. The morphisms $(\cC_1,X_1)\to (\cC_2, X_2)$ in this category are pivotal (braided) functors $\cF: \cC_1$ to $\cC_2$ such that $\mathcal{F}(X_1) \cong X_2$.

On the other hand we have the category of (braided) planar algebras. The objects of this category are (braided) planar algebras with $P_0 \cong \mathbb{C}$, and the morphisms are (braided) planar algebra homomorphisms, considered up to planar algebra natural isomorphism, see \cite[Appendix A]{MR3808052}.

Given a based pivotal (braided) category $(\cC,X)$ we can construct a (braided) planar algebra $\text{PA}(\cC, X)$ by
\[  \text{PA}(\cC, X)_n := \Hom_{\cC}(X^{\otimes n} \to \mathbf{1}).\]
Furthermore, given a pivotal (braided) functor $(\mathcal{F},\phi) : (\cC_1,X_1) \to (\cC_2, X_2)$ we define a planar algebra morphism
\[ \phi_{X_1,X_2}\phi_{X_1^{\otimes 2},X_1}\cdots \phi_{X_1^{\otimes n-1},X_1}\mathcal{F} : \text{PA}(\cC_1, X_1)_n\to \text{PA}(\cC_2, X_2)_n. \] 

Conversely, given a planar algebra $P$ we can define the based pivotal category $\cC_P$ of idempotents in $P$, and given a planar algebra homomorphism $P_1 \to P_2$, we can define a based pivotal (braided) functor $\cC_{P_1} \to \cC_{P_2}$. We are light on details of these converse constructions as they don't feature much in this paper. Additional details can be found in \cite{1607.06041}.

A powerful theorem of Henrqiques, Penneys, and Tener \cite[Theorem A]{1607.06041} shows that the above constructions give equivalences between the category of based (braided) pivotal categories, and the category of (braided) planar algebras. In particular this equivalence gives the following key Lemma for this paper.

\begin{lemma}\label{lem:gaugeauto}
Let $\cC$ be a pivotal fusion category, and $X$ a symmetrically self-dual $\otimes$-generating object. Then there is an injection
\[ \text{Gauge}(\cC) \to \Aut(   \text{PA}(\cC, X)            ).\]
\end{lemma}
\begin{proof}
By Lemma~\ref{lem:gagpiv}, every gauge auto-equivalence of $\cC$ is pivotal. The result then follows by applying \cite[Theorem A]{1607.06041}.
\end{proof}

Every planar algebra $P$ has the (possible trivial) planar ideal consisting of \textit{negligible elements}. We say an element $f \in P$ is negligible if
\[  \raisebox{-.5\height}{ \includegraphics[scale = .5]{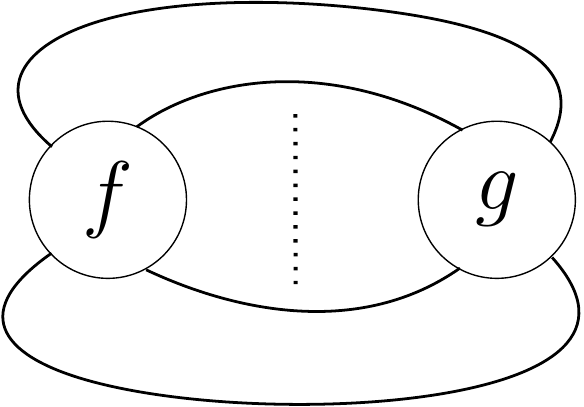}} = 0 \text{ for all } g \in P.\]
It is a straightforward exercise to show the set of all negligible elements of $P$ form a planar ideal.

\begin{dfn}
We write $\overline{P}$ for the planar quotient of $P$ by the negligible ideal.
\end{dfn}

The planar algebra $\overline{P}$ should be thought of as the ``semisimplification'' of $P$, as this construction corresponds to taking the semisimplification (in the sense of \cite{1801.04409}) of the idempotent category of $P$. Giving an explicit description of the planar ideal of $P$ can be difficult in practice. In the bulk of this paper we aim to construct planar algebra automorphisms of several examples of semisimplified planar algebras. As we won't have explicit descriptions of the negligible ideals for these examples, we instead apply the following result, which shows that every planar algebra automorphism descends to the semisimplification.

\begin{prop}\label{prop:neg}
Let $P$ be a planar algebra, and $\phi$ a planar algebra automorphism of $P$. Then $\phi$ preserves the negligible ideal of $P$.
\end{prop}
\begin{proof}
For each $f$ in the negligible planar ideal of $P$, we have to show that $\phi(f)$ is also in the negligible planar ideal. Fix such a $f$, and let $g \in P$. As $\phi$ is a planar algebra automorphism, there exists some $h \in P$ such that $\phi(h) = g$. Therefore
\[\raisebox{-.5\height}{ \includegraphics[scale = .5]{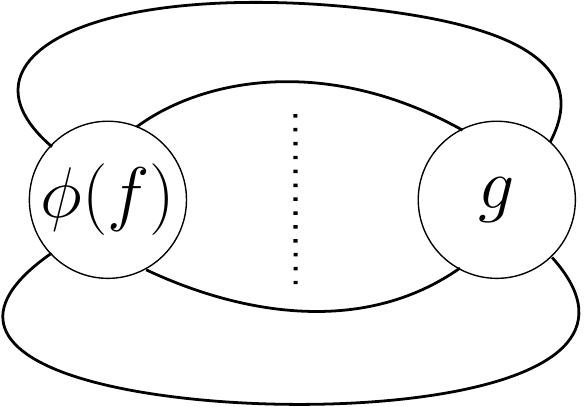}} = \raisebox{-.5\height}{ \includegraphics[scale = .5]{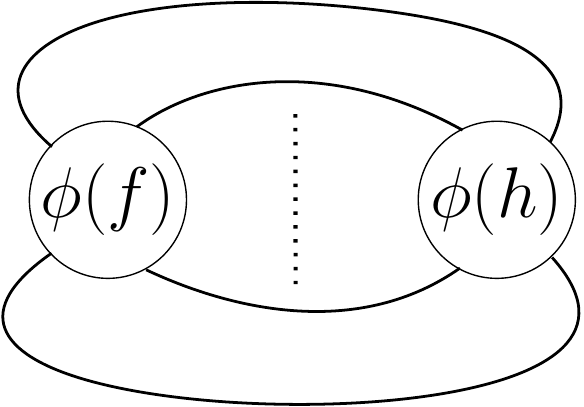}} = \phi\left( \raisebox{-.5\height}{ \includegraphics[scale = .5]{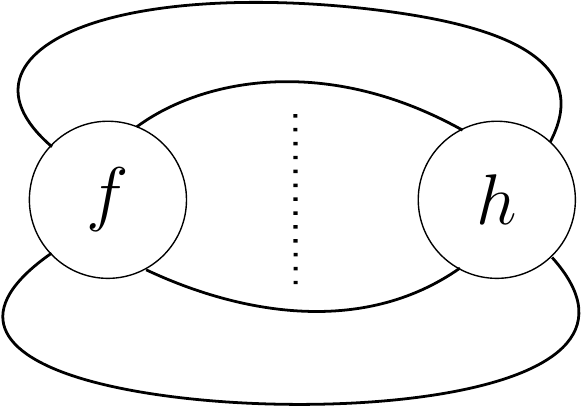}}\right) = \phi(0) = 0.\]
Thus $\phi(f)$ is in the negligible ideal of $P$.
\end{proof}

We now introduce the $\mathcal{P}^H$, BMW and $G_2$ braided planar algebras, and explain their connection to the modular categories $\cat{sl}{r+1}{k}$, $\cat{so}{2r+1}{k}$, $\cat{sp}{2r}{k}$, and $\cat{g}{2}{k}$.

\subsubsection*{The planar algebra $\mathcal{P}^H(\delta, \gamma)$}\hspace{1em}

Here we define the planar algebra $\mathcal{P}^H(\delta, \gamma)$, and explain its connection to certain endomorphism algebras in $\cat{sl}{r+1}{k}$.

\begin{dfn}\label{def:PH}
Let $\delta$ and $\gamma$ non-zero complex numbers. The braided planar algebra $\mathcal{P}^H(\delta, \gamma)$ has two generators $\gra{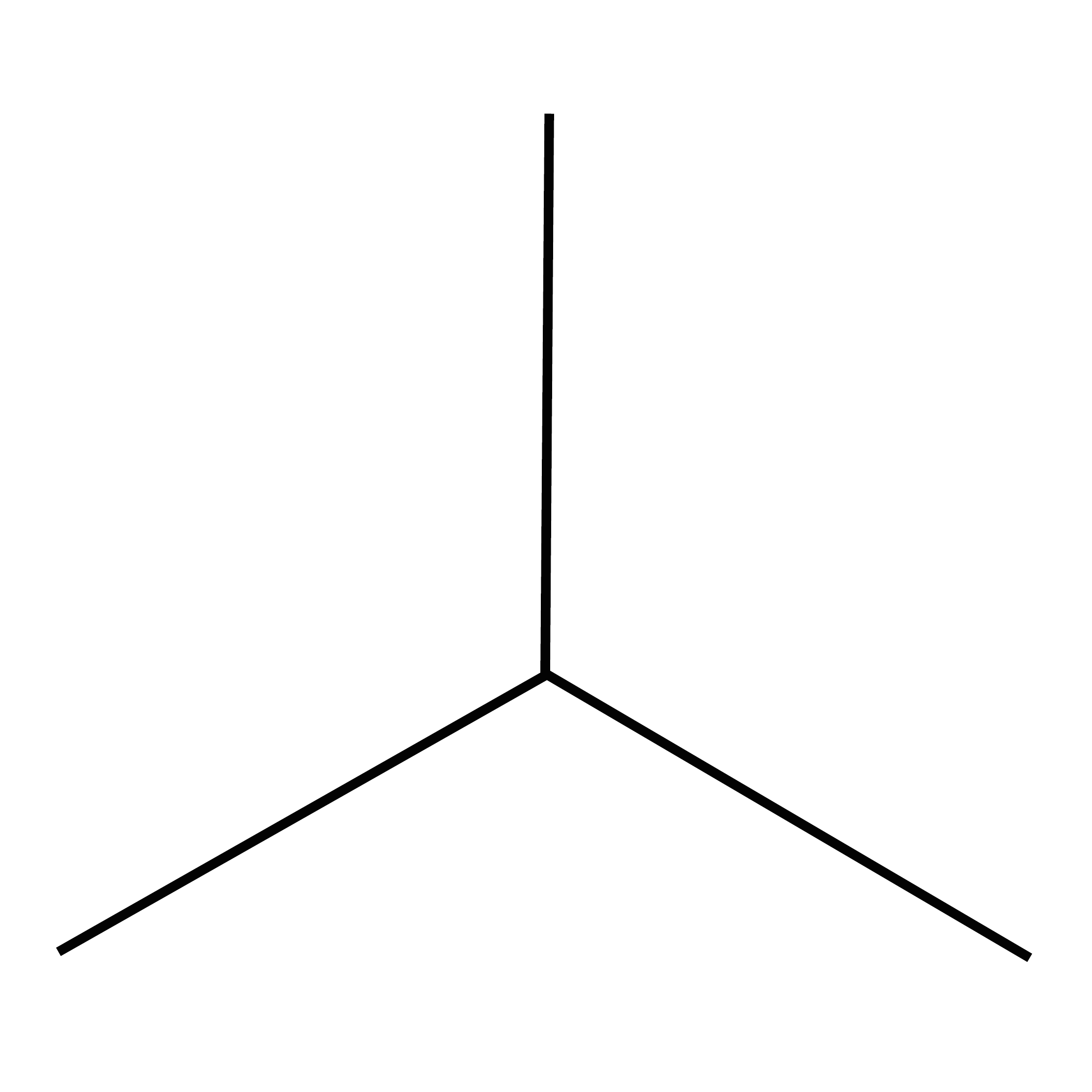} ,\gra{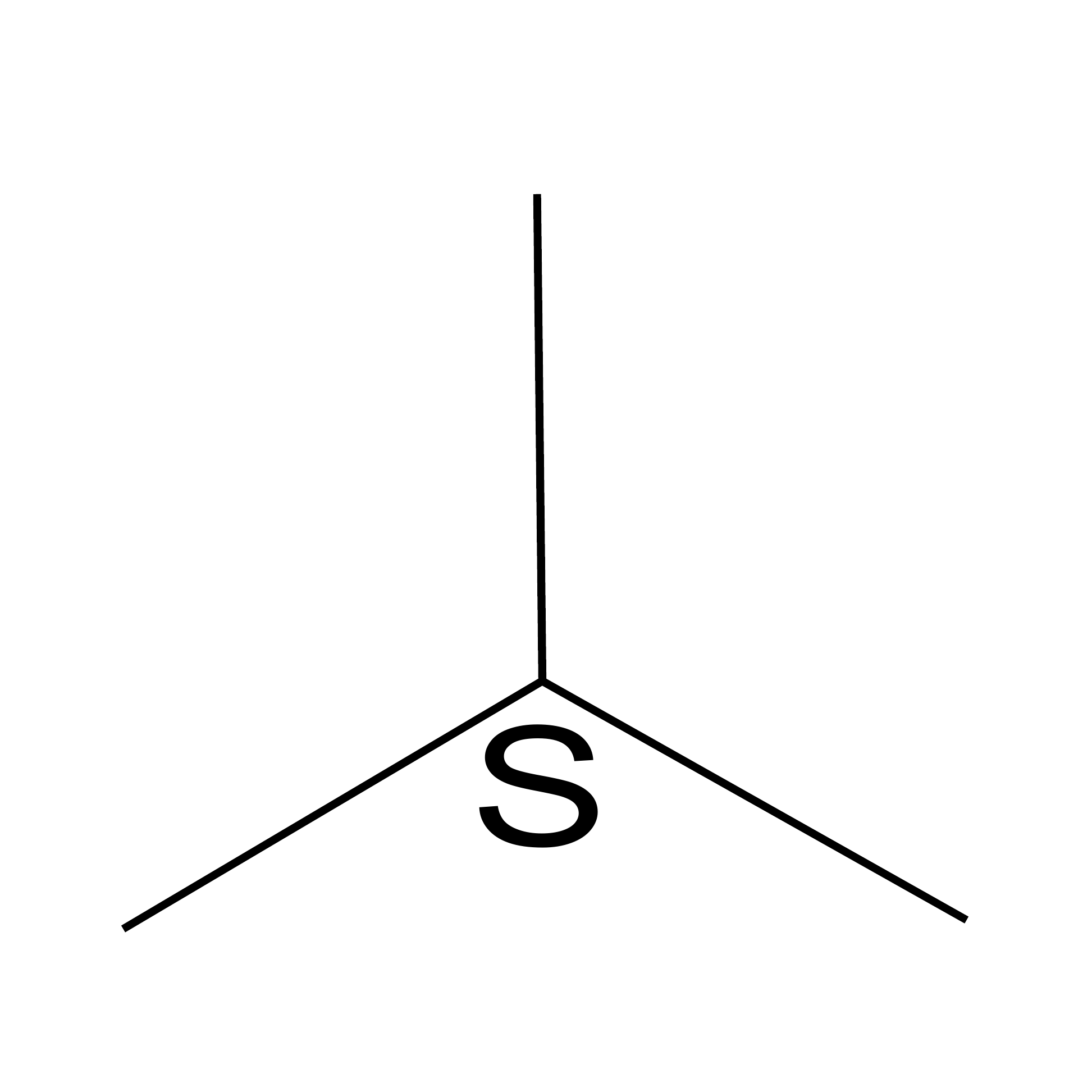} \in \mathcal{P}^H(\delta, \gamma)_3$, along with the relations:

\begin{itemize}
\item[(i)]$~~~\gra{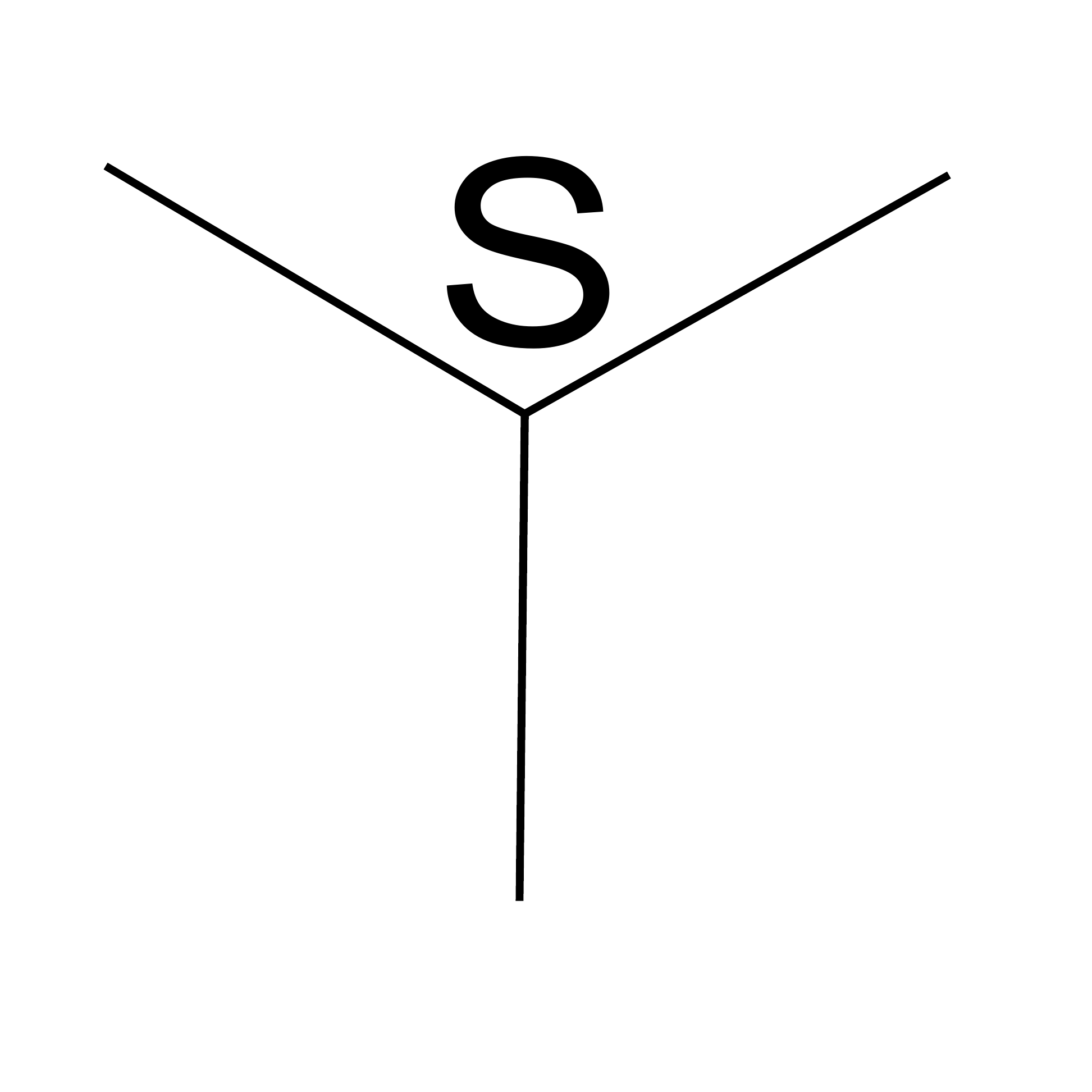}=\gra{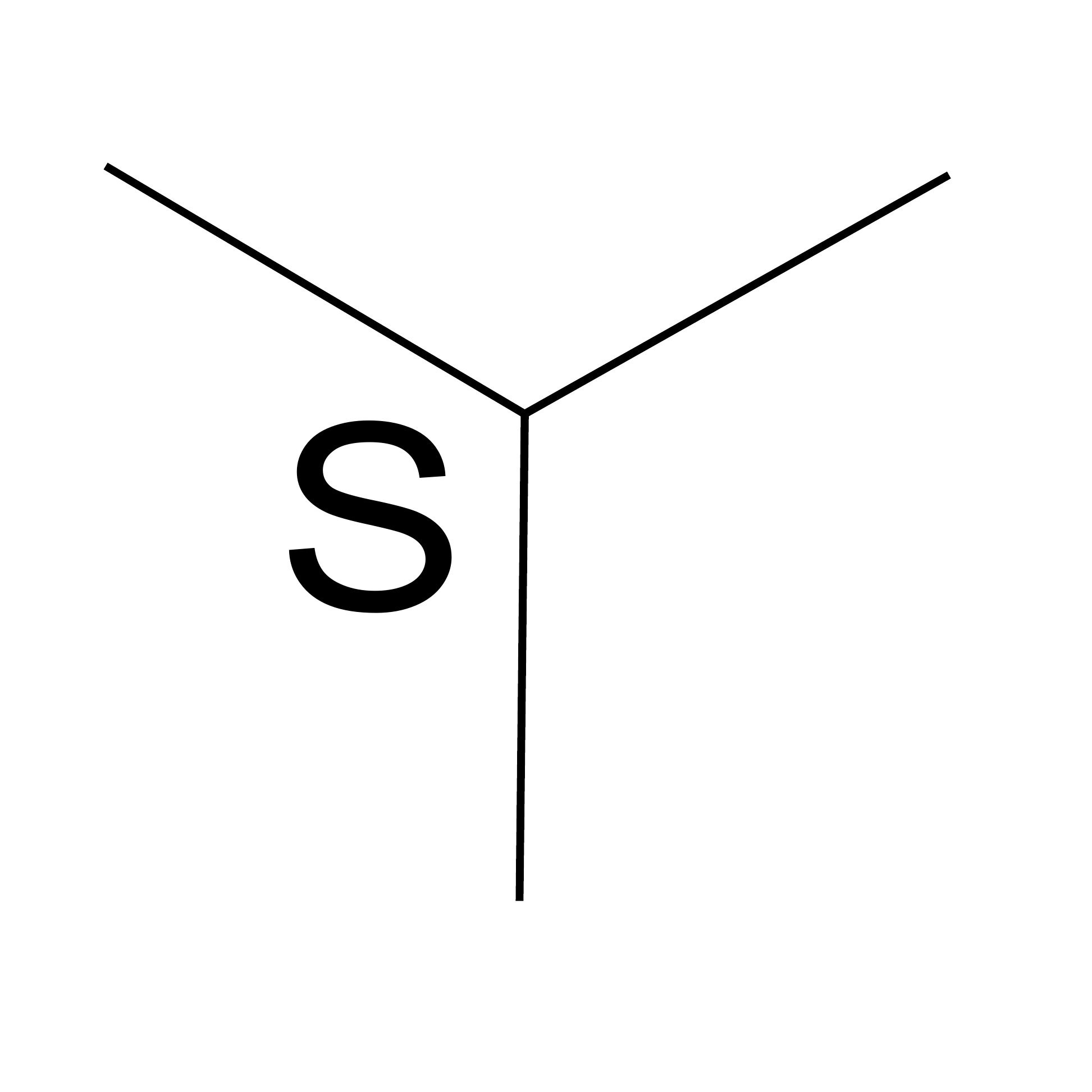};$\\
\item[(ii)]$~~\gra{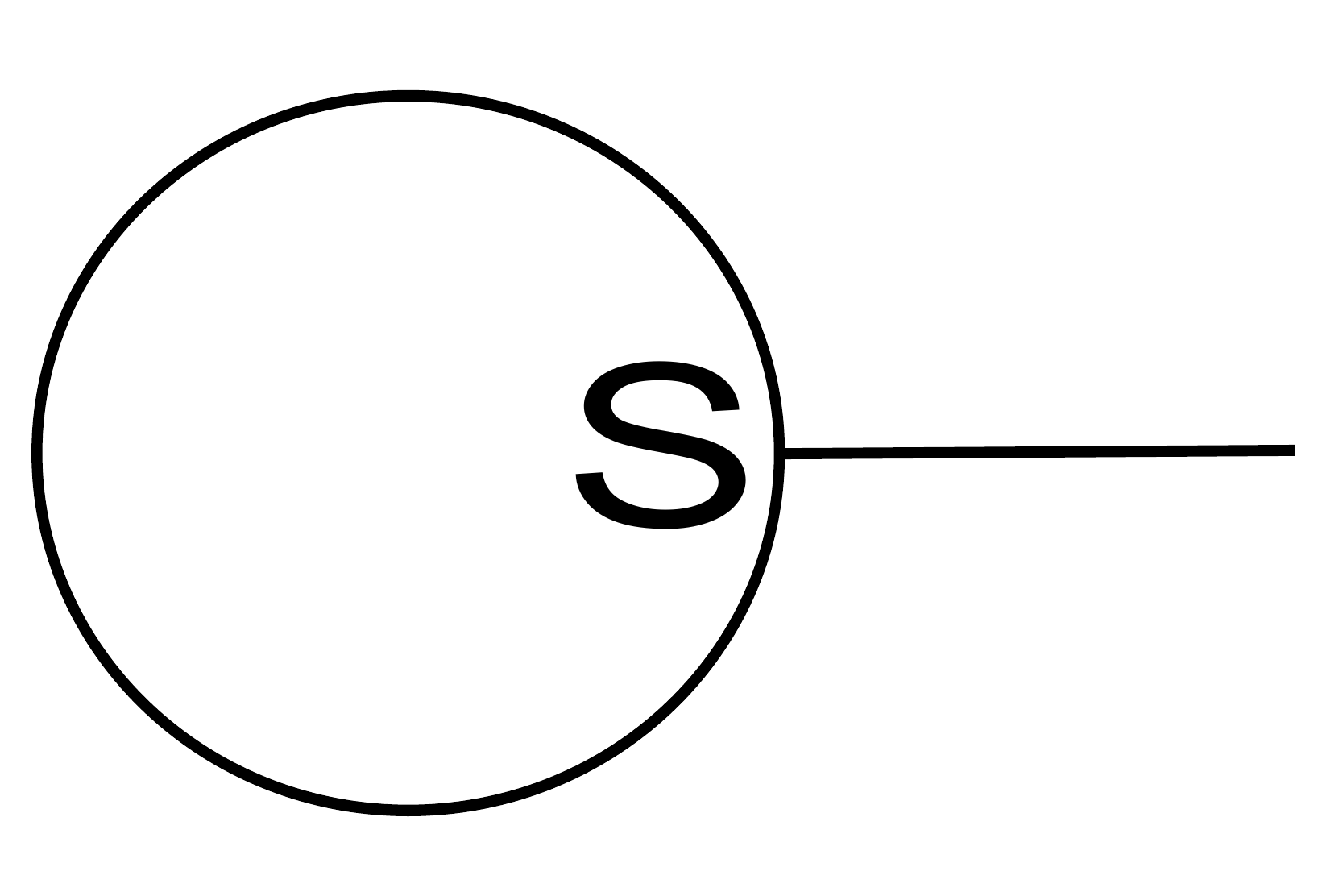}=\gra{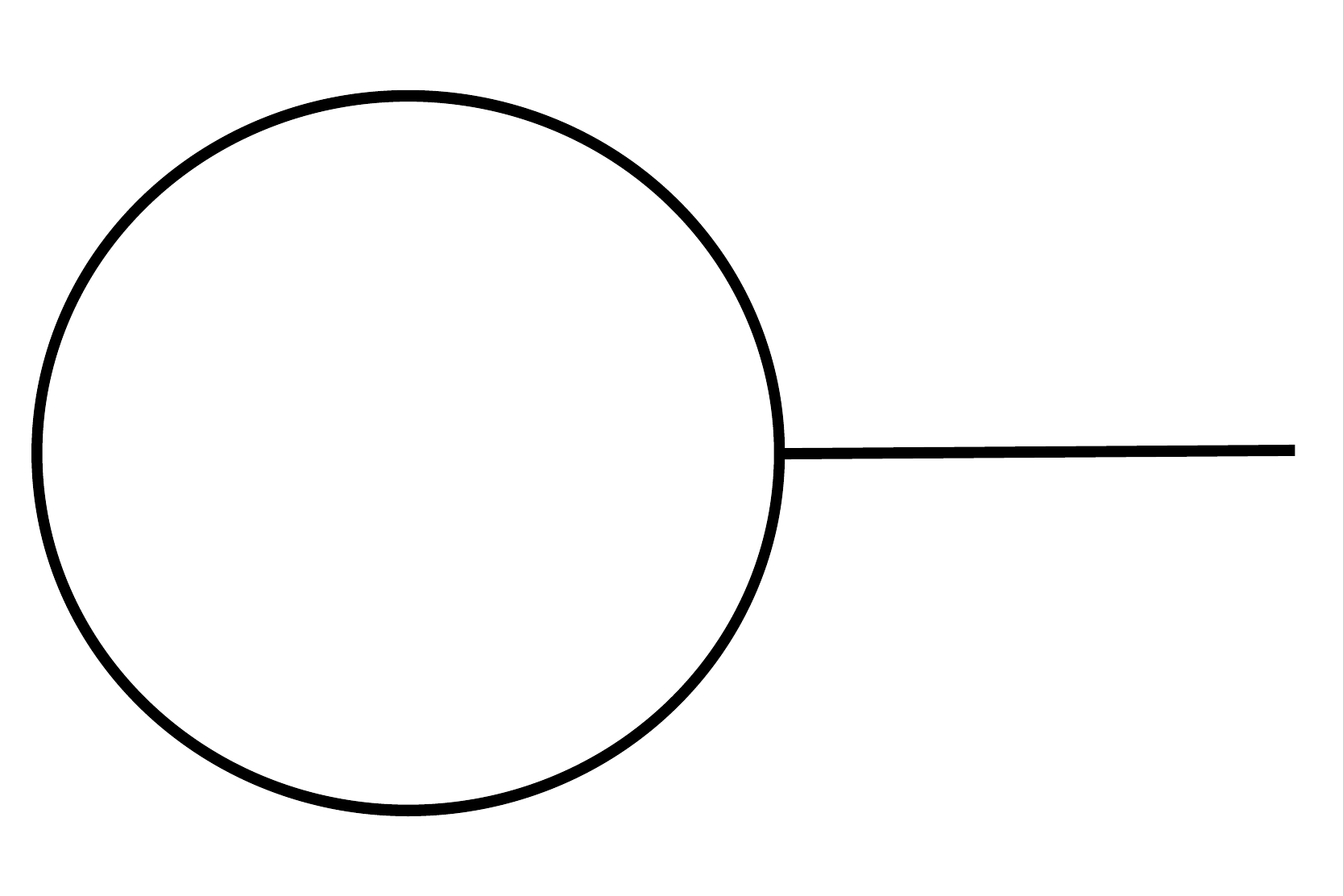}=0;$\\
\item[(iii)]$~\gra{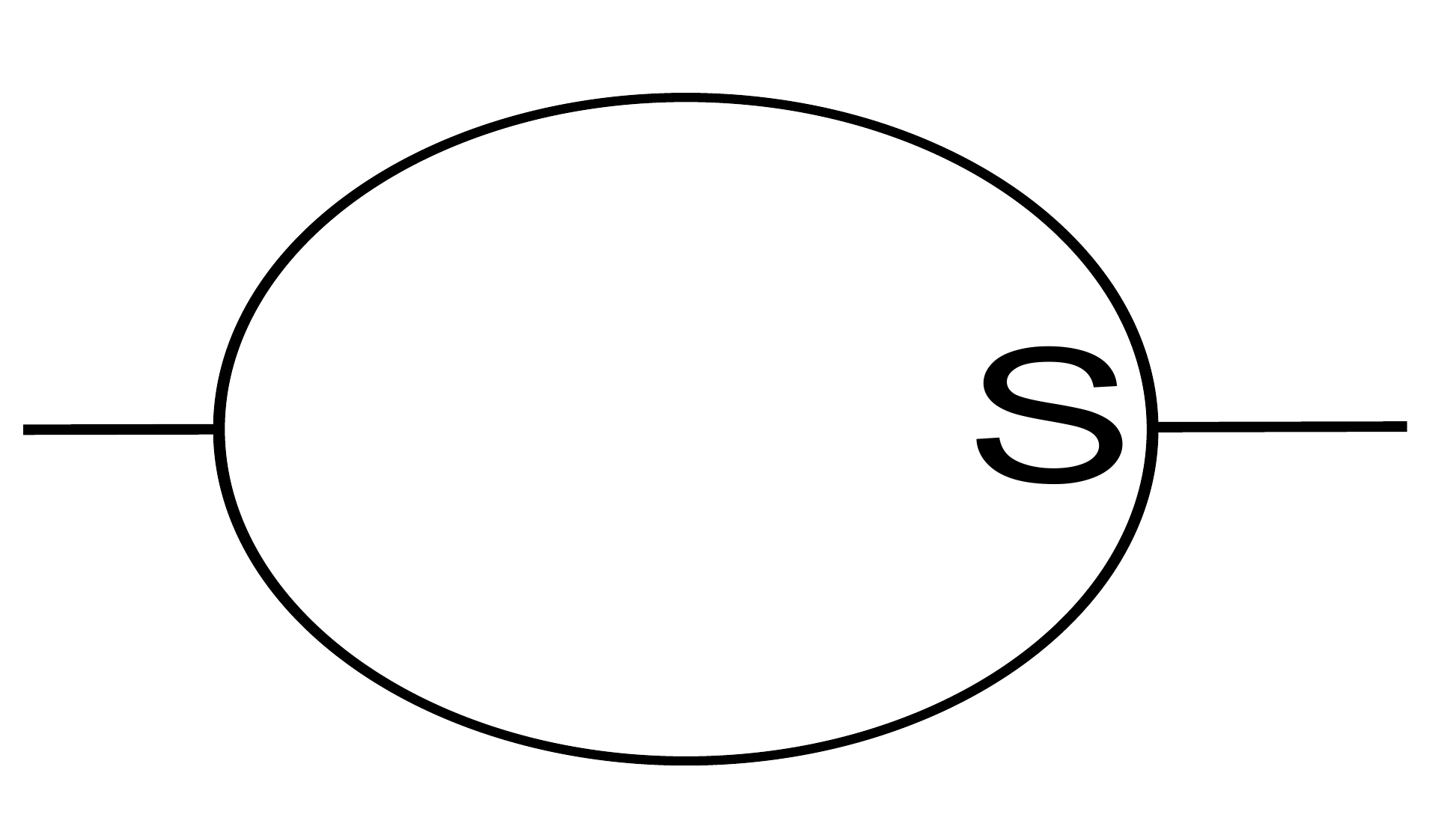}=0;$\\
\item[(iv)]$~~\gra{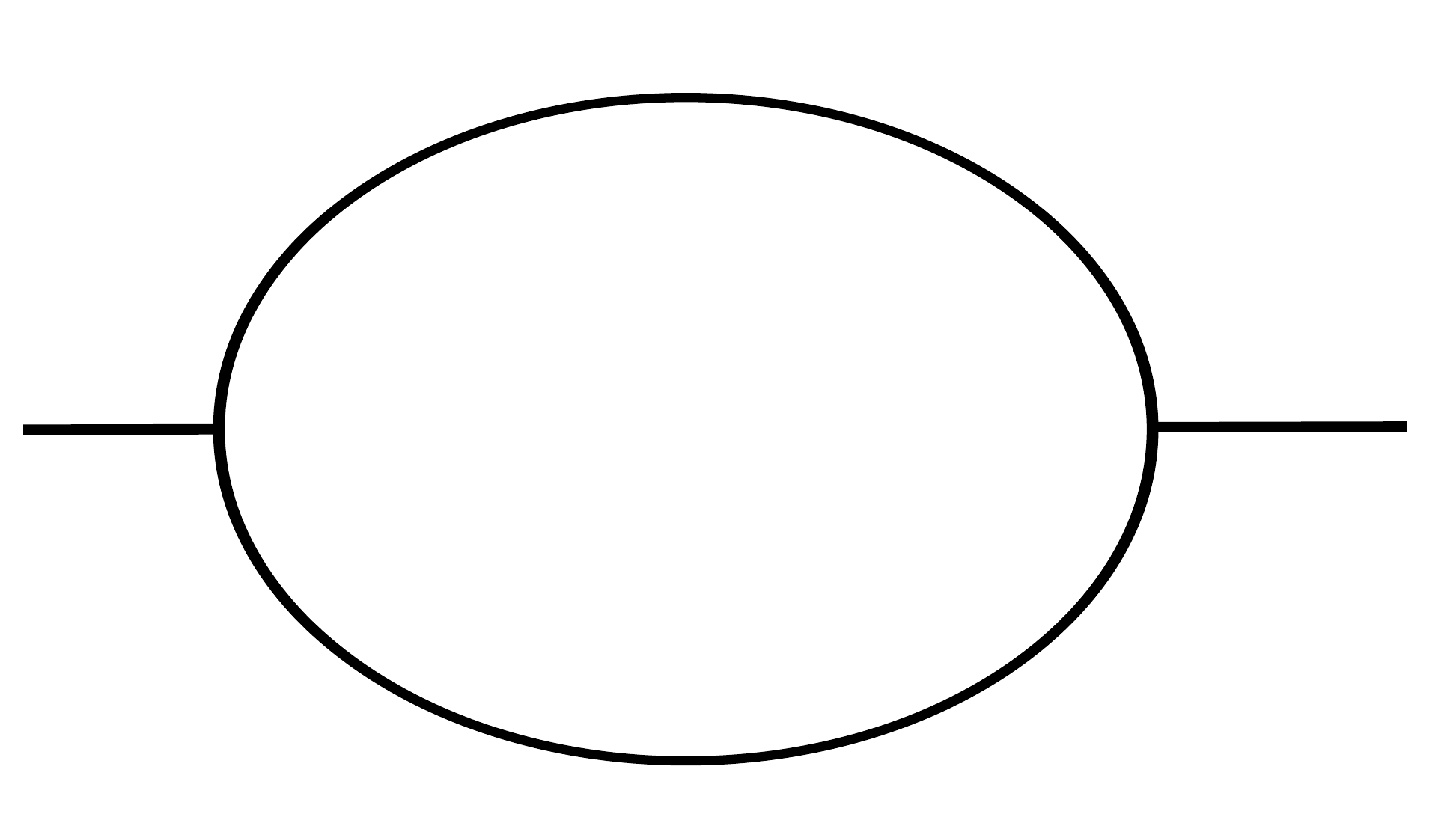}=\frac{\delta^2-2}{\delta}\gra{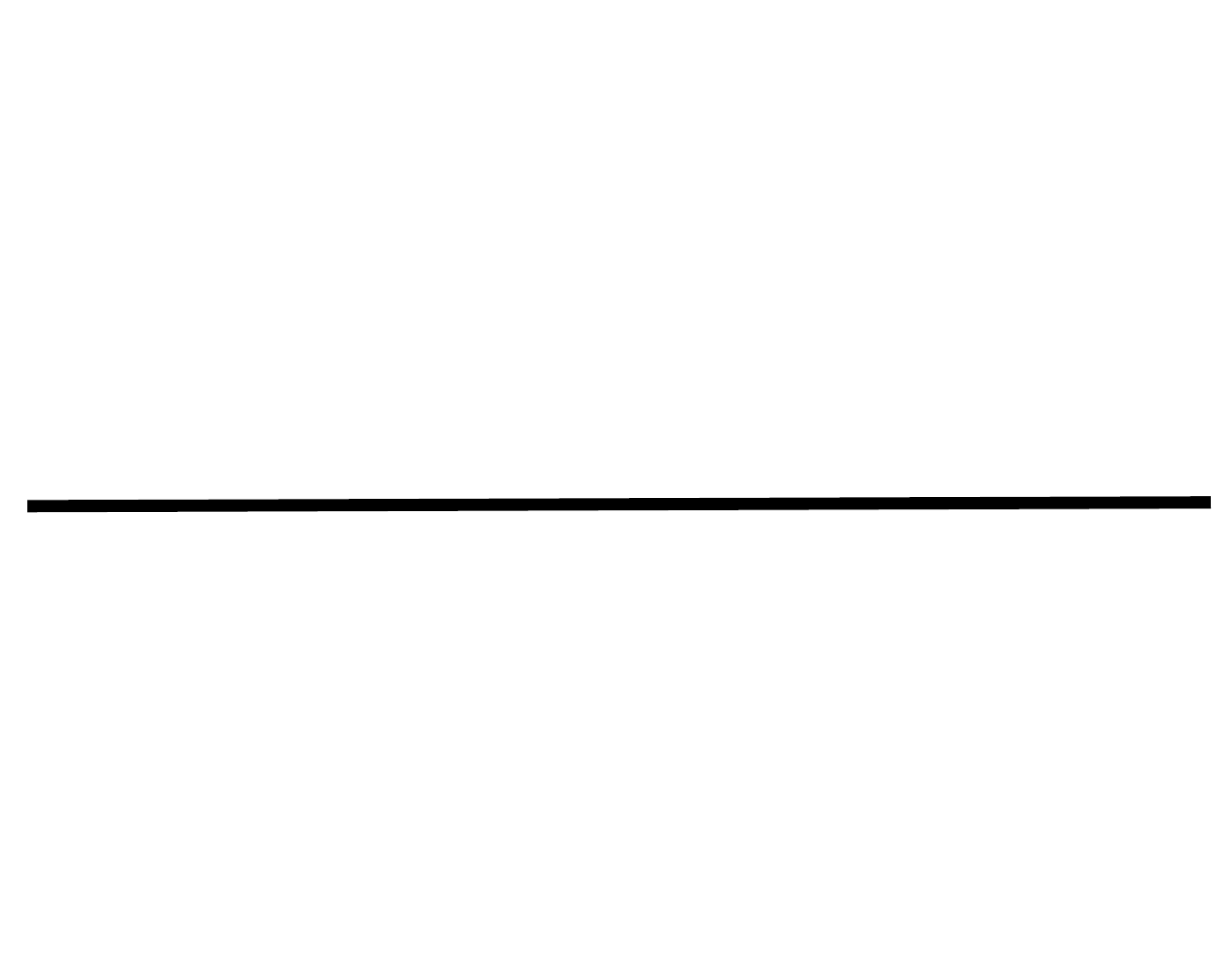};$\\
\item[(v)]$~~~\gra{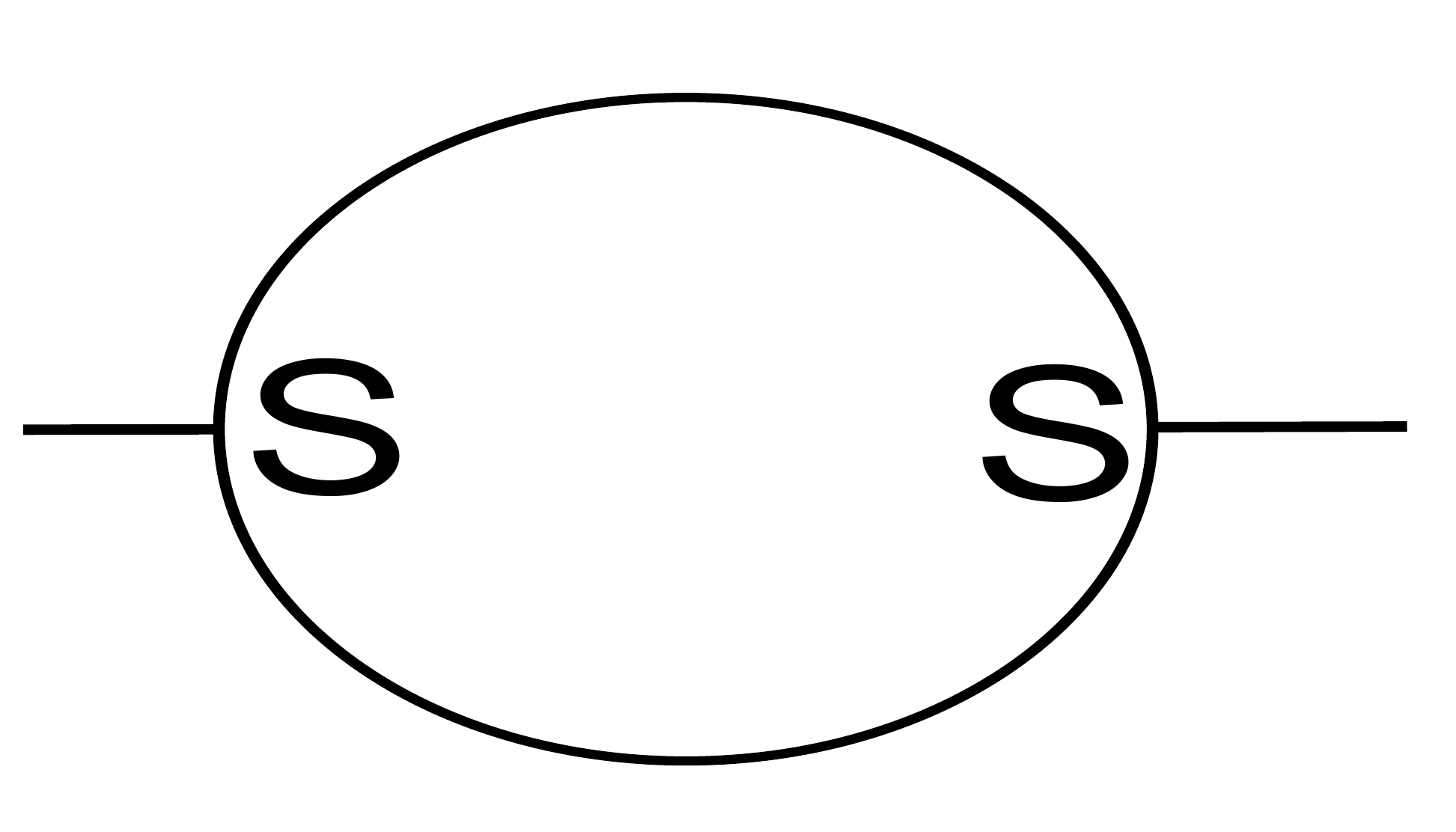}=\gamma\frac{\delta^3-2\delta}{\delta^2-1}\gra{VI};$\\
\item[(vi)]$~~\graa{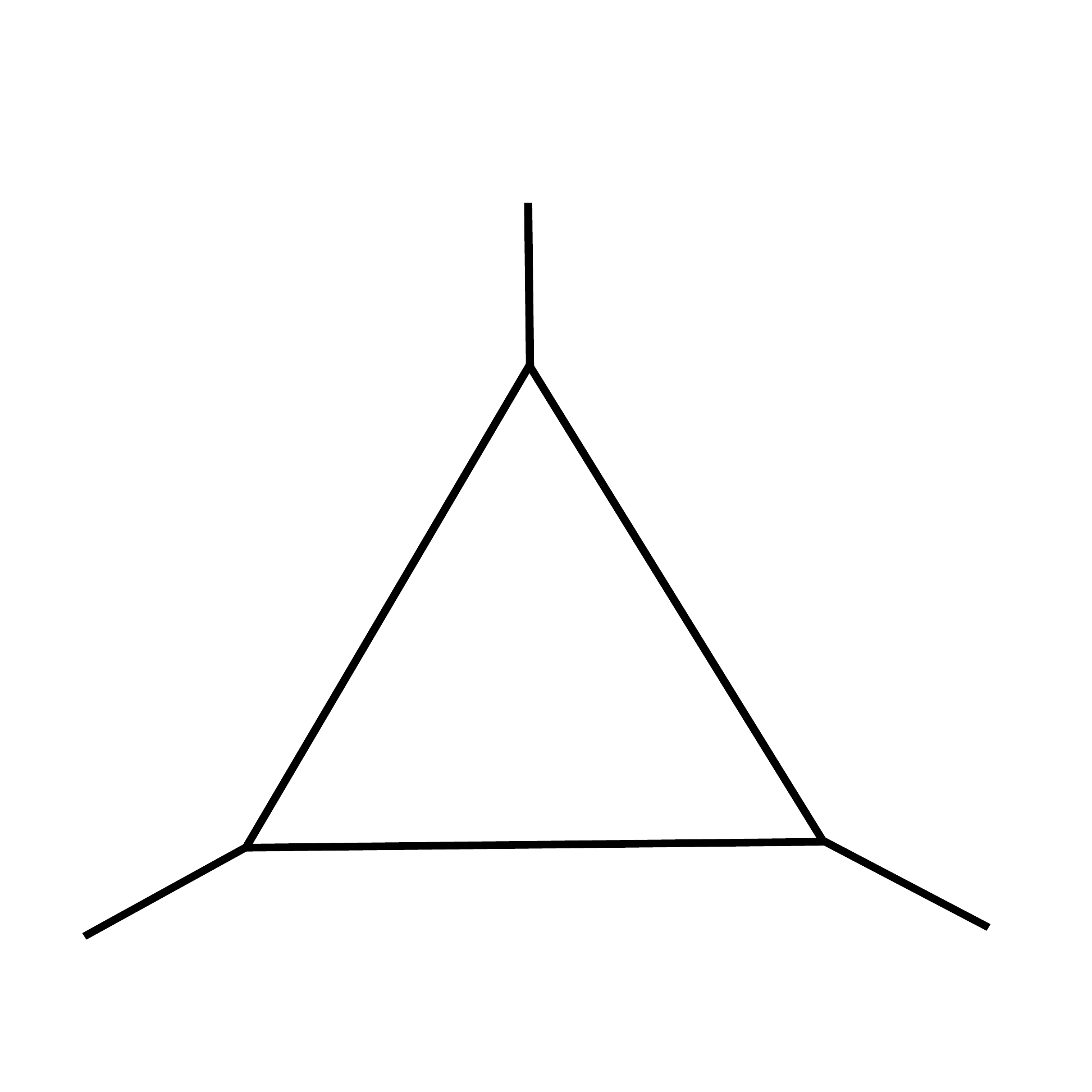}=\frac{\delta^2-3}{\delta}\graa{TLTRIV};$\\
\item[(vii)]$~\graa{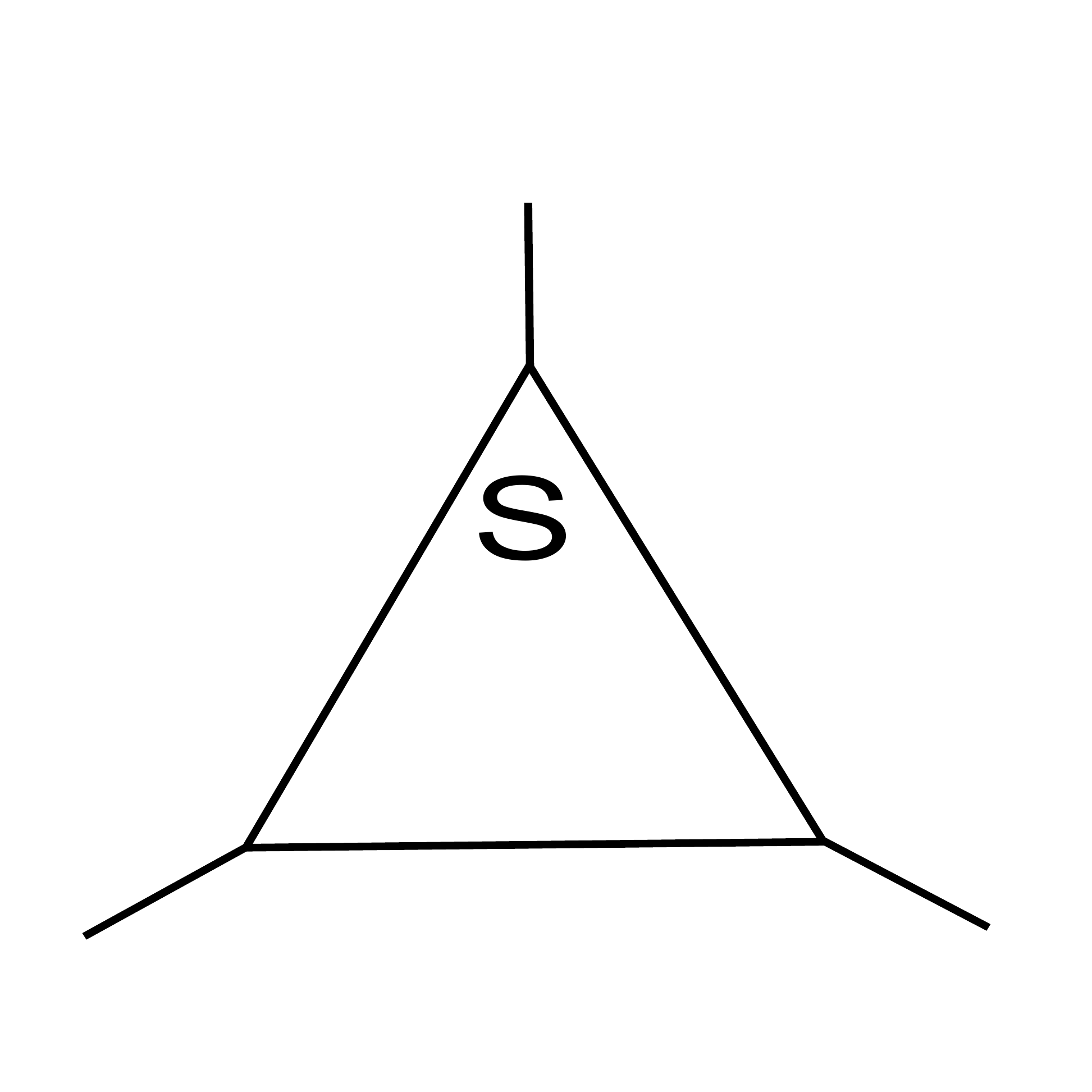}=-\frac{1}{\delta}\graa{SV};$\\
\item[(viii)]$\graa{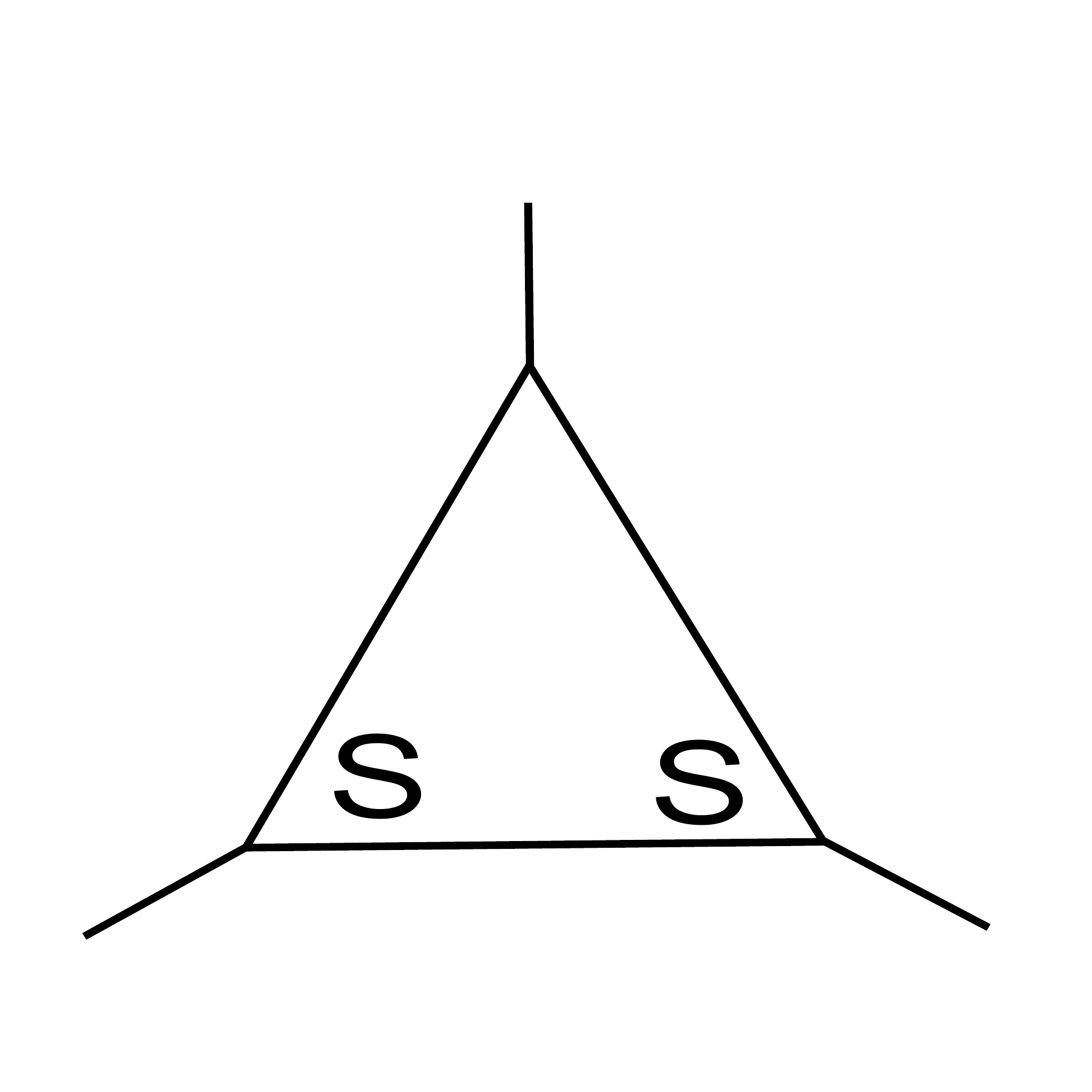}=(\gamma-1)\graa{SV}-\gamma\frac{\delta}{\delta^2-1}\graa{TLTRIV};$\\
\item[(ix)]$~~\graa{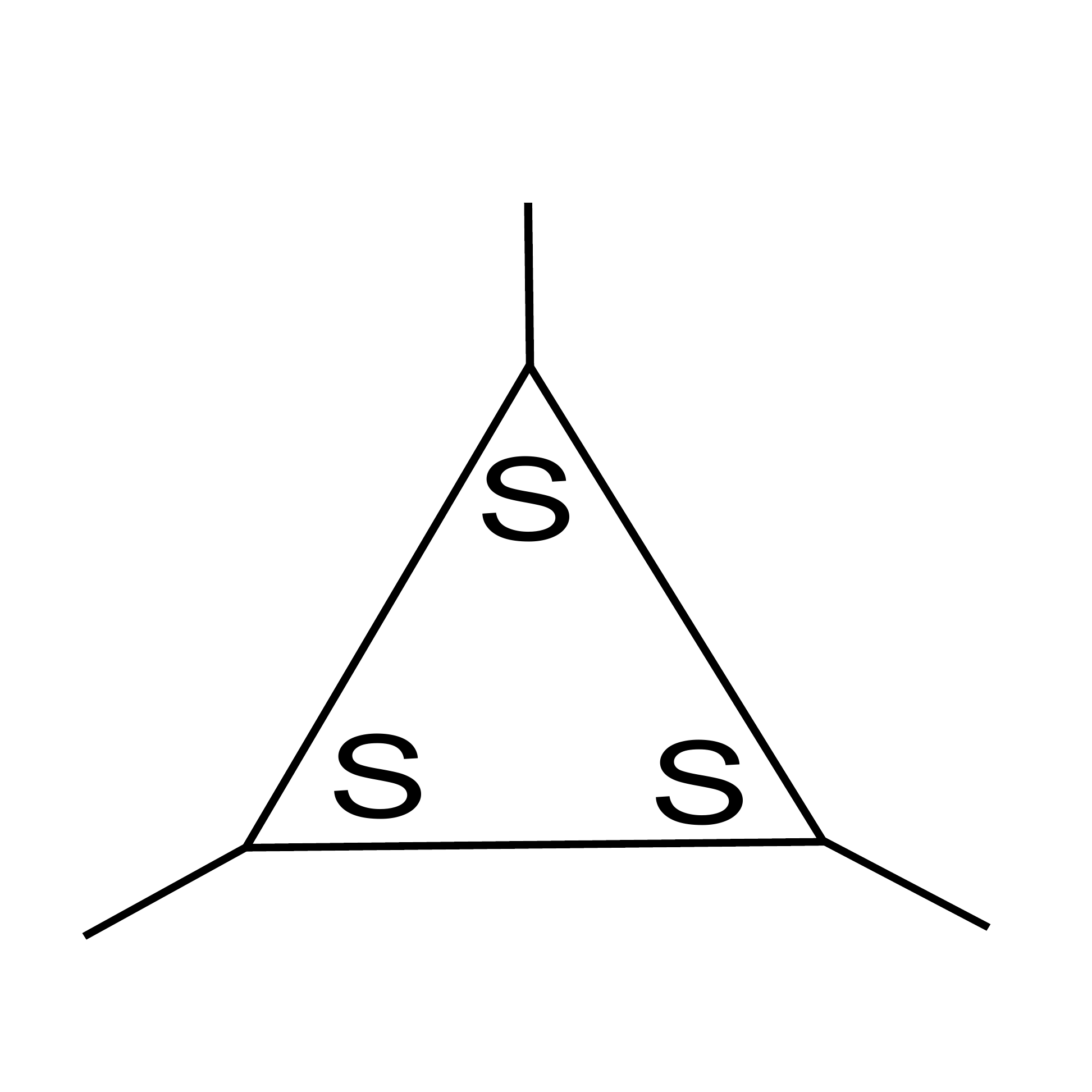}=\frac{\delta^4 \gamma+\delta^2 \left(-2 \gamma^2+\gamma-2\right)+2 (\gamma-1)^2}{\delta \left(\delta^2-1\right)}\graa{SV}+\gamma(\gamma-1)\frac{\delta^2}{\delta^2-1}\graa{TLTRIV}$.
\end{itemize}

There exists a braiding element in $\mathcal{P}^H(\delta, \gamma)_4$ which we neglect to write down, as it is complicated, and not important for this paper.
\end{dfn}

Combining several theorems gives the following isomorphism of planar algebras

\begin{equation}\label{eq:PH}
 \text{PA}( \cat{sl}{r+1}{k} , (\Lambda_1+\Lambda_r) ) \cong \overline{\mathcal{P}^H(\delta, \gamma)},
 \qquad \text{when $r \geq 2$ and $k \geq 3$}
\end{equation}
where
\[ \delta = \frac{e^\frac{2 i \pi(r+1)}{2(r + 1 + k)} - e^\frac{-2 i \pi(r+1)}{2(r + 1 + k)}}{ e^\frac{2 i \pi}{2(r + 1 + k)} - e^\frac{-2 i \pi}{2(r + 1 + k)}   } \text{ and }\gamma = \frac{e^\frac{4 i \pi}{2(r + 1 + k)}- e^\frac{(12+4(r+1)) i \pi}{2(r + 1 + k)}+ e^\frac{(8+8(r+1)) i \pi}{2(r + 1 + k)} -e^\frac{4 i \pi(r+1)}{2(r + 1 + k)}}{ e^\frac{8 i \pi}{2(r + 1 + k)}- e^\frac{(12+4(r+1)) i \pi}{2(r + 1 + k)}+ e^\frac{(4+8(r+1)) i \pi}{2(r + 1 + k)} -e^\frac{4 i \pi(r+1)}{2(r + 1 + k)}}\]

Let us quickly explain how we obtain this isomorphism. From \cite{MR1237835} we know that the Hecke algebras maps on to the endomorphism algebra of the representation $\Lambda_1$ of $U_q(\mathfrak{sl}_{r+1})$ with $q = e^\frac{2 i \pi(r+1)}{2(r + 1 + k)}$. The results of \cite{MR1710999} compute the idempotents of the semisimplified Hecke algebra. From these idempotents, we see the above map descends to an isomorphism between the semisimplified Hecke algebras, and the endomorphism algebras of $\Lambda_1 \in \overline{\Rep(U_q(\mathfrak{sl_{r+1}  }))}$. Using the braiding and duality maps of $\overline{\Rep(U_q(\mathfrak{sl_{r+1}  }))}$, we can then get an isomorphism between the semisimplified Hecke algebras, and the morphism spaces
\[   \Hom_{\overline{\Rep(U_q(\mathfrak{sl_{r+1}  }))}}(  \Lambda_1 \otimes \Lambda_r)^{\otimes n} \to \mathbf{1}).\]

The object $ \Lambda_1 \otimes \Lambda_r$ contains the sub-object $\Lambda_1 + \Lambda_r$. Hence we can cut down the semi-simplified Hecke algebra by the idempotent projecting onto $\Lambda_1 + \Lambda_r$ to obtain an isomorphism to  

\[   \Hom_{\overline{\Rep(U_q(\mathfrak{sl_{r+1}  }))}}(  (\Lambda_1 + \Lambda_r)^{\otimes n} \to \mathbf{1}).\]

The idempotent cut-down of the semisimplified Hecke algebras is, by the definition in \cite[Lemma 3.2]{MR4002229}, the box spaces of the planar algebra $\overline{\mathcal{P}^H(\delta, \gamma)}$ with $\delta$ and $\gamma$ as above. Hence, after translating to planar algebra language, we obtain a planar algebra isomorphism 
\[  \text{PA}( \overline{\Rep(U_q(\mathfrak{sl_{r+1}  }))} , (\Lambda_1+\Lambda_r) ) \cong \overline{\mathcal{P}^H(\delta, \gamma)}. \]

Finally, applying the equivalences of braided categories from \cite{MR1384612} gives the desired result.
 
\subsubsection*{The planar algebra $\text{BMW}(q,r)$}\hspace{1em}

Here we define the planar algebra $\text{BMW}(q,r)$, and explain its connection to certain endomorphism algebras in $\cat{so}{2r+1}{k}$ and $\cat{sp}{2r}{k}$.

\begin{dfn}\label{def:bmw}
Let $q \neq \pm 1$ and $r$ non-zero complex numbers. The braided planar algebra $\text{BMW}(q,r)$ has a single generator $\raisebox{-.5\height}{ \includegraphics[scale = .33]{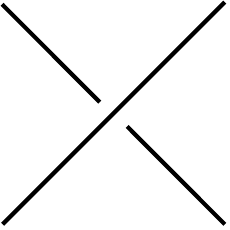}} \in \text{BMW}(q,r)_4$, satisfying Reidemiester moves 2 and 3, along with the additional relations:
\begin{itemize}
\item[ (i)] $\raisebox{-.5\height}{ \includegraphics[scale = .5]{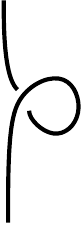}} = r \raisebox{-.5\height}{ \includegraphics[scale = .5]{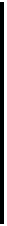}} $\\[1em]
\item[ (ii)] $\raisebox{-.5\height}{ \includegraphics[scale = .5]{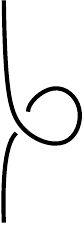}} = r^{-1} \raisebox{-.5\height}{ \includegraphics[scale = .5]{G2iso}}$\\[1em]
\item[ (iii)]$ \raisebox{-.5\height}{ \includegraphics[scale = .5]{BMWpositivecrossing}}  - \raisebox{-.5\height}{ \includegraphics[scale = .5]{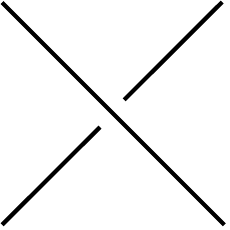}} = (q-q^{-1})\left( \raisebox{-.5\height}{ \includegraphics[scale = .5,angle = 90]{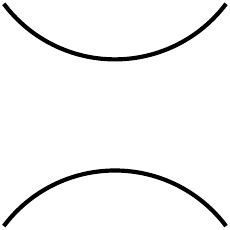}} - \raisebox{-.5\height}{ \includegraphics[scale = .5]{cupcap}}   \right) $
  \end{itemize}
The braiding in this planar algebra is simply the element $\raisebox{-.5\height}{ \includegraphics[scale = .5]{BMWpositivecrossing}}$.
\end{dfn}

An amalgamation of several theorems (\cite[Section 7.7,b]{MR2132671} and \cite{MR1384612}) gives the following isomorphisms of braided planar algebras (after some translating of languages): 
\begin{equation}\label{eq:so}
  \text{PA}( \cat{so}{2r+1}{k}\boxtimes \Rep(\Z{2}) , \Lambda_1 \boxtimes \mathbb{C}_{\operatorname{sgn} }) \cong \overline{\text{BMW}( q^2 , q^{4r} )} \text{ with } q = e^\frac{2 \pi i}{4(2r-1 + k)}. 
 \end{equation}

\begin{equation}\label{eq:sp}  
\text{PA}( \cat{sp}{2r}{k} , \Lambda_1 ) \cong \overline{\text{BMW}( q , -q^{2r+1} )} \text{ with } q = e^\frac{2 \pi i}{4(r + k + 1)}.
\end{equation}
In the latter case, the object $\Lambda_1 \in  \cat{sp}{2r}{k}$ tensor generates, and thus the idempotent completion of the $\overline{\text{BMW}( q , -q^{2r+1} )}$ braided planar algebra recovers the modular category $\cat{sp}{2r}{k}$. We have to be more careful in the former case, as the object $ \Lambda_1 \boxtimes \mathbb{C}_{\operatorname{sgn} }$ only tensor generates the subcategory
\begin{equation}\label{eq:so2} \ad(\cat{so}{2r+1}{k}  )\boxtimes \operatorname{Rep}(\Z{2}),   \end{equation}
 where
\[\ad(\cat{so}{2r+1}{k}  ) :=   \langle X\otimes X^* : X\in\cat{so}{2r+1}{k} \rangle  .\]
Hence the idempotent completion of the braided planar algebra $\overline{\text{BMW}( q , -q^{2r+1} )}$ gives the spherical braided tensor category from Equation~\ref{eq:so2}.

\subsubsection*{The planar algebra $G_2(q)$}\hspace{1em}

Here we define the planar algebra $G_2(q)$, and explain its connection to certain endomorphism algebras in $\cat{g}{2}{k}$.

\begin{dfn}
Let $q$ a non-zero complex number. The braided planar algebra $G_2(q)$ has a single generator $\raisebox{-.5\height}{ \includegraphics[scale = .33]{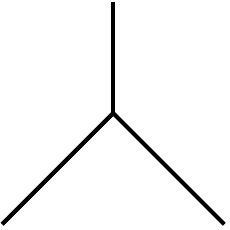}} \in G_2(q)_3$, satisfying the relations:

\begin{itemize}
\item[ (i)] $\raisebox{-.5\height}{ \includegraphics[width = .04\textwidth]{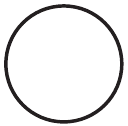}} = q^{10} + q^{8} + q^2 + 1 + q^{-2} + q^{-8} + q^{-10} $\\[1em]
\item[ (ii)]$\raisebox{-.5\height}{ \includegraphics[width = .04\textwidth]{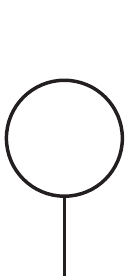}}= 0$ \\[1em]
\item[ (iii)] $\raisebox{-.5\height}{ \includegraphics[width = .04\textwidth]{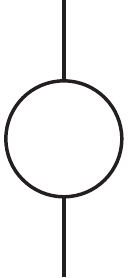}} = -\left(q^6 + q^4 + q^2 + q^{-2} + q^{-4} + q^{-6}\right) \raisebox{-.5\height}{ \includegraphics[width = .0025\textwidth]{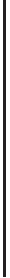}}$\\[1em]
\item[ (iv)] $\raisebox{-.5\height}{ \includegraphics[width = .06\textwidth]{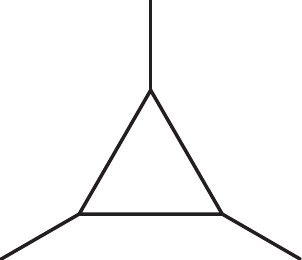}}=  \left(q^4 + 1 +q^{-4}\right)\raisebox{-.5\height}{ \includegraphics[width = .06\textwidth]{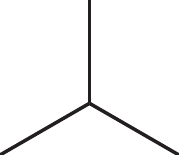}}$\\[1em]
\item[ (v)]$\raisebox{-.5\height}{ \includegraphics[width = .06\textwidth]{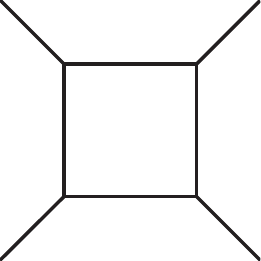}} =  -\left(q^2+q^{-2}\right) \left(\raisebox{-.5\height}{ \includegraphics[width = .04\textwidth]{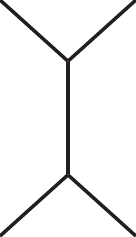}} + \raisebox{-.5\height}{ \includegraphics[width = .08\textwidth]{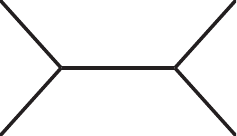}} \right)+ \left(q^2+1+q^{-2}\right) \left(\raisebox{-.5\height}{ \includegraphics[scale = .5,angle = 0]{cupcap}} +\raisebox{-.5\height}{ \includegraphics[scale = .5,angle = 90]{cupcap}} \right)$ \\[1em]
\item[ (vi)]$ \raisebox{-.5\height}{ \includegraphics[width = .06\textwidth]{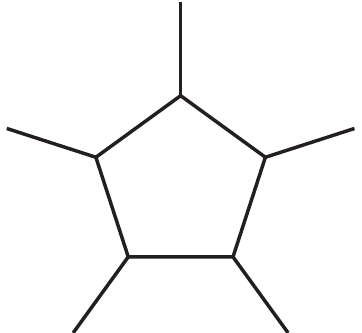}} =  \raisebox{-.5\height}{ \includegraphics[width = .06\textwidth]{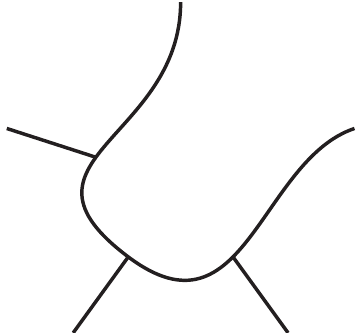}} + \raisebox{-.5\height}{ \includegraphics[width = .06\textwidth]{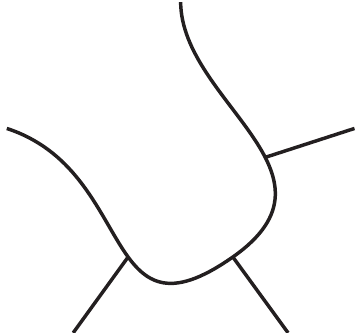}} + \raisebox{-.5\height}{ \includegraphics[width = .06\textwidth]{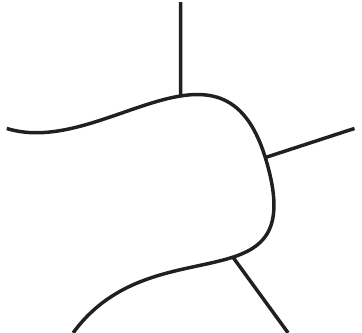}} + \raisebox{-.5\height}{ \includegraphics[width = .06\textwidth]{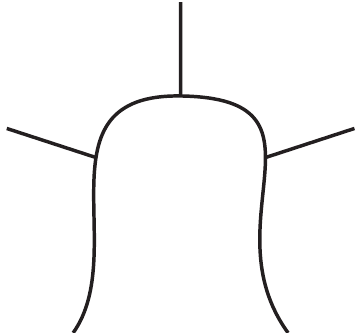}} + \raisebox{-.5\height}{ \includegraphics[width = .06\textwidth]{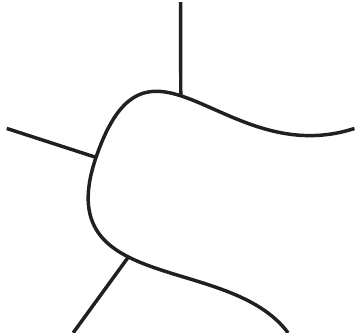}}$ \\[1em] 
\qquad$\qquad - \raisebox{-.5\height}{ \includegraphics[width = .06\textwidth]{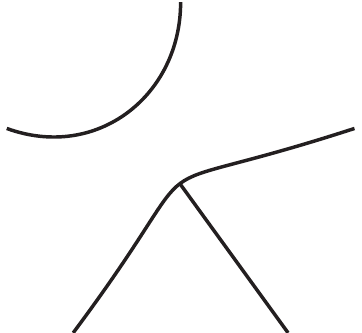}}-\raisebox{-.5\height}{ \includegraphics[width = .06\textwidth]{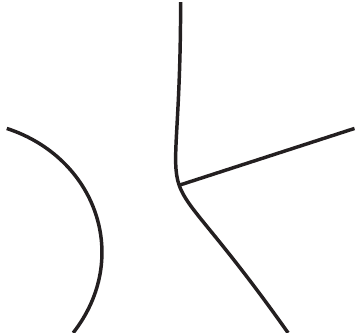}}-\raisebox{-.5\height}{ \includegraphics[width = .06\textwidth]{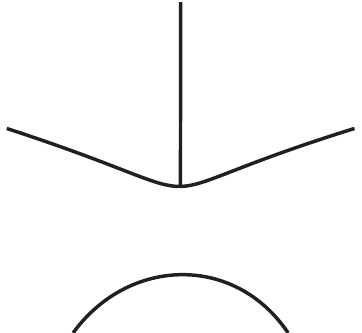}}-\raisebox{-.5\height}{ \includegraphics[width = .06\textwidth]{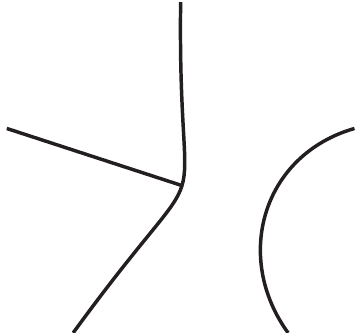}}-\raisebox{-.5\height}{ \includegraphics[width = .06\textwidth]{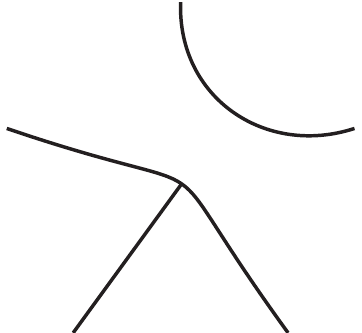}}$.
\end{itemize}
The braiding in this planar algebra is given by the element
\[ \frac{1}{1+q^{2}} \raisebox{-.5\height}{ \includegraphics[width = .04\textwidth]{I}} + \frac{1}{1+q^{-2}}\raisebox{-.5\height}{ \includegraphics[width = .08\textwidth]{H}} + \frac{1}{q^{2}+q^{4}} \raisebox{-.5\height}{ \includegraphics[scale = .5,angle = 0]{cupcap}}+ \frac{1}{q^{-2}+q^{-4}}\raisebox{-.5\height}{ \includegraphics[scale = .5,angle = 90]{cupcap}}. \] 
\end{dfn}

 In forthcoming work of Ostrik and Snyder, the following isomorphism of planar algebras is proven:

\begin{equation}\label{eq:g2}
  \text{PA}( \cat{g}{2}{k} , \Lambda_1 ) = \overline{G_2(q)}  \text{ with } q =  e^\frac{2 \pi i}{6 (4 + k)}. 
\end{equation}
In this case, the object $\Lambda_1 \in \cat{g}{2}{k}$  $\otimes$-generates the entire $\cat{g}{2}{k}$, and thus the idempotent completion of the planar algebra $ \overline{G_2(q) }$ recovers the category $\cat{g}{2}{k}$.
Additionally we have the following planar algebra isomorphisms, which will prove useful later in this paper.
\begin{lemma}\label{lem:g2iso}
There exist isomorphisms of planar algebras
\[   \overline{G_2(q) } \cong  \overline{G_2(-q) } \cong\overline{G_2(q^{-1}) } \cong \overline{G_2(-q^{-1}) }.\]
Furthermore, the isomorphism
\[   \overline{G_2(q) } \cong  \overline{G_2(-q) }\]
is an isomorphism of braided planar algebras. 
\end{lemma}
\begin{proof}
The isomorphism in each case simply sends the generator $\raisebox{-.5\height}{ \includegraphics[scale = .33]{trivalent}} \mapsto \raisebox{-.5\height}{ \includegraphics[scale = .33]{trivalent}}$. It is routine to check that all the relations are preserved by these isomorphisms.
\end{proof}

\vspace{3em}

\subsection*{Tambara-Yamagami categories} \hspace{1em}

A key tool for the computations in this paper are the Tambara-Yamagami categories \cite{MR1659954}. We will see later on that we can use auto-equivalences of the Tambara-Yamagami categories to construct auto-equivalences of the categories $\cat{so}{2r+1}{2}$. The Tambara-Yamagami categories are particularly nice to work with as it is possible to explicitly write down associators for these categories.

\begin{dfn}
Let $G$ a finite abelian group, $\chi$ a symmetric bicharacter $G\times G \to \mathbb{C}^\times$, and $\tau \in \{ \pm\}$. The Tambara-Yamagami category $\mathcal{TY}(G,\chi, \tau)$ is the fusion category with simple objects are $G \cup \{m\}$, and with fusion rules given by
\[     i \otimes j := i+j  \quad  \quad i \otimes m := m =: m\otimes i \quad \text{and} \quad m\otimes m := \oplus_G g.\]  
The associator morphisms are trivial, except for 
\begin{align*}
\alpha_{i,m,j} &:=  \chi(i,j) \id_m \\
(\alpha_{m,i,m})_{j,j} &:=  \chi(i,j)  \id_j  \\
(\alpha_{m,m,m})_{i,j} &:= \frac{\tau}{\sqrt{|G|}} \chi(i,j)^{-1}.
\end{align*}
\end{dfn} 

With such an explicit presentation of the Tambara-Yamagami categories we are able to compute the group of monoidal auto-equivalences. This result is certainly well-known to the experts, however we were unable to find a proof in the literature. Hence we give a quick proof.

\begin{lemma}\label{lem:TYautos}
Let $\Aut(G,\chi)$ be the group of automorphisms of $G$ preserving $\chi$. Then we have 
\[   \TenAut(   \mathcal{TY}(G,\chi, \tau) ) \cong \Aut(G,\chi).\]
\end{lemma}

\begin{proof}
We begin by classifying the gauge auto-equivalences of $\mathcal{TY}(G,\chi, \tau)$. Let $(\Id , \tau)$ be such a gauge auto-equivalence. Naively we have that $\tau$ is determined by $|G|^2 + 2|G| + |G|$ complex scalers. However solving for the hexagon commuting diagram shows that $\tau$ is completely determined by a function $\mu: G \to \mathbb{C}^\times$, with 
\begin{align*}
\tau_{i,j} &:= \frac{ \mu(i)\mu(j) }{\mu(ij) }\id_{i+j} \\
\tau_{i,m} &:= \mu(i)\id_{m} :=  \tau_{m,i} \\
(\tau_{m,m})_i &=\mu(i)^{-1} \id_i.
\end{align*}
Furthermore, each $\mu \in \Hom( G \to \mathbb{C}^\times)$ gives a gauge auto-equivalence $(\Id, \tau_\mu)$ of $\mathcal{TY}(G,\chi, \tau)$.

Let $\mu_1 , \mu_2 \in \Hom( G \to \mathbb{C}^\times)$. We define a natural isomorphism $\eta :(\Id, \tau_{\mu_1}) \to  (\Id, \tau_{\mu_2})$ by 
\[   \eta_i :=   \frac{ \mu_1(i) }{\mu_2(i)}\id_{i} \quad \text{ and } \quad \eta_m := \id_m.\]
A direct computation shows that $\eta$ is monoidal. Thus the group of gauge auto-equivalences of  $\TenAut(   \mathcal{TY}(G,\chi, \tau) )$ is trivial.

As the fusion ring automorphisms of $\mathcal{TY}(G,\chi, \tau)$ correspond to automorphisms of $G$, we have that $\TenAut(   \mathcal{TY}(G,\chi, \tau) ) \subseteq \Aut(G)$. Explicitly for $\psi \in \Aut(G)$, the corresponding fusion ring automorphism fixes $m$, and sends $i\mapsto \psi(i)$. A direct computation shows that there exists a tensor structure for this fusion ring automorphism if and only if $\psi$ preserves the symmetric bicharacter $\chi$.
\end{proof}
\begin{rmk}\label{rmk:tyautos}
An explicit isomorphism  
\[ \Aut(G,\chi) \to \TenAut(   \mathcal{TY}(G,\chi, \tau) ) \]
is given by
\[  \psi \mapsto \left( \begin{aligned}
                                   &i \mapsto \psi(i) \\
                                   &m \mapsto m
  \end{aligned}   , \quad \begin{aligned}
                                   \psi_{i,j} = \id_{i+j} &\quad {\psi_{m,m}}_i = \id_{i} \\
                                   \psi_{i,m} = \id_m      &\quad \psi_{m,i} = \id_m
  \end{aligned}   \right).\]
\end{rmk}

\vspace{3em}

\subsection*{Functorality of equivariantization} \hspace{1em}

A key tool for constructing auto-equivalences in this paper is through the functorality of equivariantization. In \cite[Theorem 4.4]{MR2609644} a 2-equivalence between the 2-category of $G$-crossed braided fusion categories, and the 2-category of braided fusion categories over $\Rep(G)$ is given. The author was unable to understand the full details of this 2-equivalence. In this subsection we present our interpretation of the functorality of equivariantization, for both the braided and monoidal cases.

We begin with the braided situation. Let $G$ be a finite group, then we can form two categories. On one hand we have the category of $G$-crossed braided fusion categories, whose objects are triples $( \cC , \rho , \sigma)$, where $\cC$ is $G$-graded fusion category, $\rho$ is a monoidal functor $ \underline{G} \to \underline{ \TenAut(\cC) }$, and $\omega$ is a family of natural isomorphisms
\[\omega_{X_g,Y}: X_g \otimes Y \to \rho_g(Y) \otimes X_g,\]
satisfying various conditions. A morphism  $( \cC^1 , \rho^1 , \sigma^1) \to ( \cC^2 , \rho^2 , \sigma^2)$ is a pair $(\mathcal{F} , \eta)$, where $\mathcal{F}$ is a monoidal functor $\cC^1 \to \cC^2$, and $\eta$ is a family of monoidal natural isomorphisms 
\[   \eta^g :   \rho^2(g)\circ \mathcal{F}\to \mathcal{F}\circ \rho^1(g),\]
such that the following diagram commutes:
\begin{equation}\label{eq:com}
\begin{tikzcd}[row sep=3em ,column sep = .2cm]
\cF(X_g)\otimes \cF(Y) \arrow[rr,"\tau_{X_g, Y}"] \arrow[dr,"\sigma^2_{\cF(X_g), \cF(Y) }"]  && \cF( X_g \otimes Y ) \arrow[rr,"\cF(\sigma^1_{X_g,Y})"]  && \cF( \rho^2_g(Y) \otimes X_g) \\
& (\rho^2\circ \mathcal{F})(Y) \otimes \cF(X_g) \arrow[rr,"\eta^g_Y \otimes \operatorname{id}_{\cF(X_g)}"] && (\cF \circ \rho^1_g)(Y) \otimes \cF(X_g) \arrow[ur,"\tau_{\rho^1_g(Y), X_g}"]
\end{tikzcd} 
\end{equation}

On the other hand we have the category of braided fusion categories over $\Rep(G)$, whose objects are pairs $(\mathcal{B} , \mathcal{H})$, where $\mathcal{B}$ is a braided fusion category, and $\mathcal{H}$ is a fully faithful braided functor $\Rep(G) \to \mathcal{B}$. A morphism $(\mathcal{B}_1,\mathcal{H_1}) \to (\mathcal{B}_2,\mathcal{H_2})$ is just a braided functor $\mathcal{B}_1 \to \mathcal{B}_2$.

We now describe a functor from the category of $G$-crossed braided fusion categories, to the category of braided fusion categories over $\Rep(G)$.
\begin{rmk}
We wish to point out that this category of braided fusion categories over $\Rep(G)$ is different from the category of braided fusion categories over $\Rep(G)$ in \cite{MR2609644}. This difference is the reason why we are only able to give a functor between category of $G$-crossed braided fusion categories, to the category of braided fusion categories over $\Rep(G)$, and not an equivalence as was done in \cite{MR2609644}.
\end{rmk}

Given $(\cC,\rho,\omega)$, then $G$-equivariantization produces a braided fusion category $\cC^G$ along with a faithful braided functor $H: \Rep(G) \to \cC^G$ defined by:
\[  (V,\pi) \mapsto ( \mathbf{1}\otimes V , \{\pi(g) \} ).\]

Given a morphism $(\mathcal{F}, \eta):  (\cC^1,\rho^1,\omega^1)\to  (\cC^2,\rho^2,\omega^2)$, we define a braided functor $\cF^G : (\cC^1)^G \to  (\cC^2)^G $ by:
\[  \cF^G ( (X , \{\mu_g : X \to \rho^1_g(X) \}   ) ) :=   (\cF(X) , \{ (\eta^g)^{-1}  \cF(\mu_g) : \cF(X) \to \rho^2_g(\cF(X)) \}   )  .       \]
The tensor structure morphisms are exactly the tensor structure morphisms of $\cF$, which are $G$-equivariant morphisms from the monoidality of $\eta$. The functor $\cF^G$ is braided due to the commutativity of \eqref{eq:com}.

In the monoidal situation we can also construct a functor. Let $G$ a finite group, then we again define two categories. On one hand we the category of fusion categories with $G$-action in a similar way to the definition of the category of $G$-crossed braided fusion categories, but without any $G$-crossed braiding structure or conditions.

On the other hand we have the category of fusion categories over $\Rep(G)$, whose objects are pairs $(\cC , \mathcal{H})$, where $\cC$ is a fusion category, and $\mathcal{H}$ is a fully faithful braided functor $\Rep(G) \to \mathcal{Z}(\cC)$. A morphism $(\cC_1 , \mathcal{H}_1)\to (\cC_2 , \mathcal{H}_2)$ is just a monoidal functor $\cC_1 \to \cC_2$.

A similar construction as in the braided case gives a functor from category of fusion categories with $G$-action, to the category of fusion categories over $\Rep(G)$.

\section{Auto-equivalences for Lie type $A$}
In this section we compute the monoidal, and braided auto-equivalences of the categories $\cat{sl}{r+1}{k}$. This case is in some sense easier than the others, as there are no exceptional auto-equivalences to construct. However, the categories $\cat{sl}{r+1}{k}$ have many invertible objects, which give rise to a vast number of simple current auto-equivalences. Also, the existence of the $\Z{2}$-symmetry of the $A_r$ Dynkin diagram induces an order $2$ braided auto-equivalence. Hence the categories $\cat{sl}{r+1}{k}$ still have a large amount of auto-equivalences that we have to construct, making this case non-trivial, even without the existence of exceptional auto-equivalences.

The outline of our computations is as follows. 

We begin by showing that the category $\cat{sl}{r+1}{k}$ has no non-trivial gauge auto-equivalences. To do this, we first study the gauge auto-equivalences of the category $\ad(\cat{sl}{r+1}{k})$. We know that $(\Lambda_1 +\Lambda_r)$ is a generating object of $\ad(\cat{sl}{r+1}{k})$, and further, the planar algebra generated by $(\Lambda_1 +\Lambda_r)$ is isomorphic to $\overline{\mathcal{P}^H(\delta, \gamma)}$ for a specific choice of $\delta$ and $\gamma$. Thus, any gauge auto-equivalence of $\ad(\cat{sl}{r+1}{k})$ will appear as a planar algebra automorphism of $\overline{\mathcal{P}^H(\delta, \gamma)}$.

Using the generators and relations presentation of $\overline{\mathcal{P}^H(\delta, \gamma)}$ we compute the automorphism group, finding several non-trivial automorphisms. In theory each of these non-trivial automorphisms could lift to a gauge auto-equivalences of the category $\ad(\cat{sl}{r+1}{k})$. However, we are able to show that each of these non-trivial automorphisms lifts to a non-gauge auto-equivalence. The proper way to do this would be to compute the idempotents of the planar algebra $\overline{\mathcal{P}^H(\delta, \gamma)}$, and to show that each non-trivial automorphism must exchange non-isomorphic idempotents. However, computing the idempotents of $\overline{\mathcal{P}^H(\delta, \gamma)}$ proved to be too computationally intensive to complete. Instead we identify known pivotal auto-equivalences of the category $\ad(\cat{sl}{r+1}{k})$ coming from charge-conjugation and level-rank duality, that fix the object $(\Lambda_1 + \Lambda_r)$. These auto-equivalences must appear as automorphisms of the planar algebra $\overline{\mathcal{P}^H(\delta, \gamma)}$, hence we are left with no room in the automorphism group of $\overline{\mathcal{P}^H(\delta, \gamma)}$ for a gauge auto-equivalence.

With the fact that the gauge auto-equivalence group of $\ad(\cat{sl}{r+1}{k})$ is trivial, we can then leverage the fact that $\cat{sl}{r+1}{k}$ is a $\Z{r+1}$-graded extension of $\ad(\cat{sl}{r+1}{k})$ to show that the gauge auto-equivalence group of $\cat{sl}{r+1}{k}$ is also trivial.  

As the category $\cat{sl}{r+1}{k}$ has no non-trivial gauge auto-equivalences, the group of monoidal auto-equivalences of $\cat{sl}{r+1}{k}$ are a subgroup of the group of fusion ring automorphisms. The results of \cite{MR1887583} compute the order of the group of fusion ring automorphisms of $\cat{sl}{r+1}{k}$, giving upper bounds on the size of the monoidal auto-equivalence group of $\cat{sl}{r+1}{k}$. We find out that these bounds are sharp, by constructing enough monoidal auto-equivalences of $\cat{sl}{r+1}{k}$ to realise this upper bound. We construct these auto-equivalences using the charge-conjugation symmetry of the $A_r$ Dynkin diagram, and simple current auto-equivalences which we get from the large number of invertible objects in $\cat{sl}{r+1}{k}$.

With the high level argument explained, let us begin with the finer details of the computation.

\begin{lemma}
Suppose $\delta \neq \pm \sqrt{2}$. The planar algebra $\overline{\mathcal{P}^H(\delta, \gamma)}$ has planar algebra automorphism group (up to planar algebra natural isomorphism) as follows
\[     \Aut(\overline{\mathcal{P}^H(\delta, \gamma)}) = \begin{cases}
\Z{2}, \text{ if $\gamma \neq 1$},\\
\Z{2}\times \Z{2}  \text{ if $\gamma = 1$}.
\end{cases}\]

\end{lemma}
\begin{proof}
As the planar algebra $\overline{\mathcal{P}^H(\delta, \gamma)}$ is generated by the two elements $\gra{TLTRIV} ,\gra{SV} \in \overline{\mathcal{P}^H(\delta, \gamma)}_3$, any planar algebra automorphism of $\overline{\mathcal{P}^H(\delta, \gamma)}$ is completely determined by its value on these two elements. Explicitly this means that $\phi \in   \Aut(\overline{\mathcal{P}^H(\delta, \gamma)})$ is completely determined by four scalers $c_1,c_2,c_3,c_4 \in \mathbb{C}$, with 
\[   \phi\left( \gra{TLTRIV} \right) = c_1 \gra{TLTRIV} + c_2  \gra{SV}, \quad \text{ and } \quad  \phi\left( \gra{SV} \right) = c_3 \gra{TLTRIV} + c_4  \gra{SV}.\]

Applying $\phi$ to relations (iii) through (vii) of $\overline{\mathcal{P}^H(\delta, \gamma)}$ yields the following system of equations
\begin{align*}
 \delta-\frac{2}{\delta} \quad &= \quad c_1^2 \left(\delta-\frac{2}{\delta}\right)+\frac{c_2^2 \delta \left(\delta^2-2\right) \gamma}{\delta^2-1} \\
0  \quad &= \quad c_1 c_3 \left(\delta-\frac{2}{\delta}\right)+\frac{c_2 c_4 \delta \left(\delta^2-2\right) \gamma}{\delta^2-1}  \\
 \frac{\delta \left(\delta^2-2\right) \gamma}{\delta^2-1}  \quad &= \quad c_3^2 \left(\delta-\frac{2}{\delta}\right)+\frac{c_4^2 \delta \left(\delta^2-2\right) \gamma}{\delta^2-1}  \\
\frac{c_1 \left(\delta^2-3\right)}{\delta}  \quad &= \quad \frac{c_1^3 \left(\delta^2-3\right)}{\delta}-\frac{3 c_1 c_2^2 \delta \gamma}{\delta^2-1}+\frac{c_2^3 \delta^2 (\gamma-1) \gamma}{\delta^2-1}  \\
\frac{c_2 \left(\delta^2-3\right)}{\delta}   \quad &= \quad -\frac{3 c_1^2 c_2}{\delta}+3 c_1 c_2^2 (\gamma-1)+\frac{c_2^3 \left(\delta^4 \gamma+\delta^2 \left(-2 \gamma^2+\gamma-2\right)+2 (\gamma-1)^2\right)}{\delta \left(\delta^2-1\right)}  \\
-\frac{c_3}{\delta}  \quad &= \quad \frac{c_1^2 c_3 \left(\delta^2-3\right)}{\delta}-\frac{\delta \gamma \left(2 c_1 c_2 c_4+c_2^2 c_3\right)}{\delta^2-1}+\frac{c_2^2 c_4 \delta^2 (\gamma-1) \gamma}{\delta^2-1} \\
-\frac{c_4}{\delta}  \quad &= \quad -\frac{c_1^2 c_4+2 c_1 c_2 c_3}{\delta}+(\gamma-1) \left(2 c_1 c_2 c_4+c_2^2 c_3\right)+\frac{c_2^2 c_4 \left(\delta^4 \gamma+\delta^2 \left(-2 \gamma^2+\gamma-2\right)+2 (\gamma-1)^2\right)}{\delta \left(\delta^2-1\right)}
\end{align*}
With an additional assumption that $\delta \neq \sqrt{2}$, a computer algebra program\footnote{Reduce from the Mathematica library} can solve these equations to find the following solution sets, where in each case $\varepsilon_1,\varepsilon_2 \in \{+,-\}$:

\begin{itemize}

\item The solutions
\[   \phi^1_{\varepsilon_1, \varepsilon_2} : \{ c_1 =\varepsilon_1  ,
c_2 = 0,
c_3 =0 ,
c_4 = \varepsilon_2\}  ,\]
which exist for all parameter values $\delta$ and $\gamma$.

\item The solutions
\[   \phi^2_{\varepsilon_1, \varepsilon_2} : \{ c_1 =\varepsilon_1   \frac{\sqrt{\delta^2-1} (\gamma-1)}{\sqrt{\delta^4 \gamma+\delta^2 \gamma^2-2 \delta^2 \gamma+\delta^2-\gamma^2+2 \gamma-1}},
c_2 = \frac{c_1 \delta }{\gamma-1},
c_3 = \varepsilon_2 \frac{1-c_1^2}{c_2} ,
c_4 = \frac{c_1 c_2 c_3}{c_1^2-1}\}  ,\]
which exist for all $\delta$ when $\gamma \neq 1$.

\item The solutions
 \[   \phi^3_{\varepsilon_1, \varepsilon_2} : \{ c_1 =0,
c_2 =\varepsilon_1 \frac{\sqrt{\delta^2-1}}{\delta},
c_3 =\varepsilon_2  \frac{1}{c_2} ,
c_4 = 0 \}  ,\]
which exist for all $\delta$ when $\gamma = 1$.

\end{itemize}

We can apply these solutions to the remaining relations of $\overline{\mathcal{P}^H(\delta, \gamma)}$ to see that 
\[  \phi^1_{+, +}, \quad \phi^1_{-.-}, \quad  \phi^2_{+, +},\quad  \text{ and } \quad  \phi^2_{-.-}\]
 are planar algebra automorphisms of $\overline{\mathcal{P}^H(\delta, \gamma)}$ when $\gamma \neq 1$, and 
 \[ \phi^1_{+, +}, \quad \phi^1_{+, -},\quad \phi^1_{-, +}, \quad  \phi^1_{-.-}, \quad  \phi^3_{+, +}, \quad \phi^3_{+, -}, \quad \phi^3_{-, +},\quad \text{ and} \quad   \phi^3_{-.-}\]
  are planar algebra automorphisms of $\overline{\mathcal{P}^H(\delta, \gamma)}$ when $\gamma = 1$.

Finally, there exist planar algebra natural isomorphisms 
\[     \phi^\circ_{\varepsilon_1, \varepsilon_2} \to \phi^\circ_{-\varepsilon_1, -\varepsilon_2}.\]
Hence we have two planar algebra automorphisms, up to planar algebra natural isomorphism, of $\overline{\mathcal{P}^H(\delta, \gamma)}$ when $\gamma \neq 1$, and four when $\gamma = 1$. A direct computation shows that these planar algebra automorphisms compose as in the statement of the Lemma. 
\end{proof}

As a corollary, we are able to use this planar algebra computation to determine the gauge auto-equivalences of $\ad(  \cat{sl}{r+1}{k})$.

\begin{cor}\label{cor:Agauge}
We have 
\[   
 \operatorname{Gauge}( \ad(  \cat{sl}{r+1}{k}))    =      \begin{cases}
\Z{2}, \text{ if $r = 2$ and $k = 3$},\\
\{e\} \text{ otherwise}.
\end{cases}
\]
\end{cor}
\begin{proof}
We prove this statement in several cases.

\begin{trivlist}\leftskip=2em
\item \textbf{Case $k=1$:}

In the $k=1$ case we have that $\ad(  \cat{sl}{r+1}{2}) \simeq \operatorname{Vec}$ which clearly has no non-trivial gauge auto-equivalences.

\vspace{1em}

\item \textbf{Case $k=2$:}

When $k= 2$ we have that $\ad(  \cat{sl}{r+1}{2})$ is monoidaly equivalent to $\ad(  \cat{sl}{2}{r+1})$ via \cite[Theorem 5.1]{MR3162483}. The category $\ad(  \cat{sl}{2}{r+1})$ is shown to have trivial gauge auto-equivalence group in \cite[Lemma 4.2]{MR3808052}.

\vspace{1em}

\item \textbf{Case $k \geq 3$ and $(r,k) \neq (2,3)$:}

Recall that the object $(\Lambda_1 + \Lambda_r)$ is a symmetrically self-dual $\otimes$-generator of $\ad(  \cat{sl}{r+1}{k})$, and that the planar algebra generated by $(\Lambda_1 + \Lambda_r)$ is isomorphic to $\overline{\mathcal{P}^H(\delta, \gamma)}$ with 
\[ \delta = \frac{e^\frac{2 i \pi(r+1)}{2(r + 1 + k)} - e^\frac{-2 i \pi(r+1)}{2(r + 1 + k)}}{ e^\frac{2 i \pi}{2(r + 1 + k)} - e^\frac{-2 i \pi}{2(r + 1 + k)}   } \text{ and }\gamma = \frac{e^\frac{4 i \pi}{2(r + 1 + k)}- e^\frac{(12+4(r+1)) i \pi}{2(r + 1 + k)}+ e^\frac{(8+8(r+1)) i \pi}{2(r + 1 + k)} -e^\frac{4 i \pi(r+1)}{2(r + 1 + k)}}{ e^\frac{8 i \pi}{2(r + 1 + k)}- e^\frac{(12+4(r+1)) i \pi}{2(r + 1 + k)}+ e^\frac{(4+8(r+1)) i \pi}{2(r + 1 + k)} -e^\frac{4 i \pi(r+1)}{2(r + 1 + k)}}.\]

From \cite[Theorem A]{1607.06041} we know that the group of planar algebra automorphisms of $\overline{\mathcal{P}^H(\delta, \gamma)}$ corresponds to the group of pivotal auto-equivalences of $\ad(  \cat{sl}{r+1}{k})$ that fix the object $(\Lambda_1 + \Lambda_r)$. As the category $\ad(  \cat{sl}{r+1}{k})$ has a unique pivotal structure, we have that every monoidal auto-equivalence of $\ad(  \cat{sl}{r+1}{k})$ is pivotal. Hence the group of planar algebra automorphisms of $\overline{\mathcal{P}^H(\delta, \gamma)}$ also corresponds to the group of monoidal auto-equivalences of $\ad(  \cat{sl}{r+1}{k})$ that fix the object $(\Lambda_1 + \Lambda_r)$. Recalling the previous Lemma, and translating the parameters, we find the categories $\ad(  \cat{sl}{r+1}{k})$ for $k \neq r+1$ have exactly two monoidal auto-equivalences that fix the object $(\Lambda_1 + \Lambda_r)$, and the categories $\ad(  \cat{sl}{r+1}{r+1})$ have exactly four monoidal auto-equivalences that fix the object $(\Lambda_1 + \Lambda_r)$. We now attempt to identify these auto-equivalences.

From Corollary~\ref{cor:cc} we know that the category $ \cat{sl}{r+1}{k}$ has the order 2 charge-conjugation auto-equivalence, which maps 
\[      \sum_{i=1}^r  \lambda_i \Lambda_i  \mapsto \sum_{i=1}^r  \lambda_{r - i + 1} \Lambda_i.\]
In particular, this charge-conjugation auto-equivalence restricts to a monoidal auto-equivalence of $\ad(  \cat{sl}{r+1}{k})$ fixes the object $(\Lambda_1 + \Lambda_r)$. Further, the charge-conjugation auto-equivalence moves the object $(\Lambda_2 + 2\Lambda_r) \in   \ad(  \cat{sl}{r+1}{k})$. Therefore the charge-conjugation auto-equivalence restricts to a non-gauge auto-equivalence of $\ad(  \cat{sl}{r+1}{k})$ that fixes $(\Lambda_1 + \Lambda_r)$.

From \cite[Theorem 5.1]{MR3162483} there exists a level-rank duality auto-equivalence of $\ad(  \cat{sl}{r+1}{r+1})$ which has order 2. The details of this auto-equivalence are complicated, and can be found in the cited paper. For us it suffices to know that this auto-equivalence fixes the object $(\Lambda_1 + \Lambda_r)$, and the auto-equivalence is anti-braided, so it will only fix an object if its twist is equal to $\pm 1$. Consider the object $(\Lambda_2 + \Lambda_{r-1}) \in \ad(  \cat{sl}{r+1}{r+1})$. This object has twist equal to $e^{2\pi i \frac{4r}{4(r+1)} }$ when $r \geq 3$. Hence the twist of this object is not equal to  $\pm 1$ when $r \geq 3$, and so this object must be moved by the level-rank duality auto-equivalence. In particular, this auto-equivalence is not gauge if $r \geq 3$.

We are now in place to compute the gauge auto-equivalence group of $ \ad(  \cat{sl}{r+1}{k})$. When $k \neq r+1$ there are exactly two monoidal auto-equivalences of $ \ad(  \cat{sl}{r+1}{k})$ that fix the object $(\Lambda_1 + \Lambda_r)$. One of these is the identity, and the other is charge-conjugation auto-equivalence. Thus there is no room for a non-trivial gauge auto-equivalence of $ \ad(  \cat{sl}{r+1}{k})$ when $k \neq r+1$.

When $k = r+1$ there are exactly four monoidal auto-equivalences of $ \ad(  \cat{sl}{r+1}{r+1})$ that fix the object $(\Lambda_1 + \Lambda_r)$. One of these is the identity, another is the charge-conjugation auto-equivalence, and a third is the level-rank duality auto-equivalence. When $r \geq 3$ the level-rank auto-equivalence moves the object $(\Lambda_2 + \Lambda_{r-1})$, yet the charge-conjugation auto-equivalence fixes this object. Thus the charge-conjugation and level-rank duality auto-equivalences are distinct. Further these two auto-equivalences have order 2, thus their composition gives a fourth distinct auto-equivalence that moves the object $(\Lambda_2 + \Lambda_{r-1})$, and fixes the object $(\Lambda_1 + \Lambda_r)$. Thus we have constructed four non-gauge auto-equivalences of $ \ad(  \cat{sl}{r+1}{r+1})$ that fix the object $(\Lambda_1 + \Lambda_r)$, hence there is no room for a non-trivial gauge auto-equivalence of $ \ad(  \cat{sl}{r+1}{r+1})$ when $r \geq 3$.

\vspace{1em}

\item \textbf{Case $(r,k)= (2,3)$:}

A direct computation shows that there are only two fusion ring automorphisms of $ \ad(  \cat{sl}{3}{3})$. The single non-trivial fusion ring automorphism fixes the objects $\mathbf{1}$ and $(\Lambda_1 +\Lambda_2)$, and exchanges the objects $(3\Lambda_1)$ and $(3\Lambda_2)$. We can realise this non-trivial fusion ring automorphism as the charge-conjugation auto-equivalence. However there are four monoidal auto-equivalences of $ \ad(  \cat{sl}{3}{3})$ that fix $(\Lambda_1 +\Lambda_2)$. Therefore two of these four monoidal auto-equivalences must be gauge.
\end{trivlist}
\end{proof}

With the knowledge that $\ad(  \cat{sl}{r+1}{k})$ has no gauge auto-equivalences except for a single exception, we can leverage this to show that $\cat{sl}{r+1}{k}$ has no non-trivial gauge auto-equivalences, with no exceptions.

\begin{theorem}
The category $\cat{sl}{r+1}{k}$ has trivial gauge auto-equivalence group.
\end{theorem}
\begin{proof}
We prove this statement in several cases

\begin{trivlist}\leftskip=2em
\item \textbf{Case $r=1$:}

If $r = 1$, then the statement of this Theorem is exactly \cite[Lemma 4.2]{MR3808052} paired with Lemma~\ref{lem:gaugeauto}.

\vspace{1em}

\item \textbf{Case $(r,k) = (2,3)$:}

If $r = 2$ and $k = 3$, then the object $\Lambda_1$ generates an orientated $A_2$ planar algebra with $q = e^\frac{2 i \pi}{12}$ \cite{MR1403861}. This planar algebra has two generators, and the automorphisms are parametrised by a single non-zero complex number. Further, every element of this 1-parameter family of planar algebra automorphisms is naturally isomorphic to the identity. Hence, the orientated version of \cite[Theorem A]{1607.06041} (which has not appeared in print, but is known as folklore) implies that the group of pivotal auto-equivalences of $\cat{sl}{3}{3}$ that fix the object $\Lambda_1$ is trivial. As every gauge auto-equivalence of $\cat{sl}{3}{3}$ is pivotal by Lemma~\ref{lem:gagpiv}, we have the result.

\vspace{1em}

\item \textbf{Case $r \geq 2$ and $(r,k) \neq (2,3)$:}

For this general case we will use \cite[Theorem 9.1]{1711.00645}. This Theorem allows us to leverage the fact that $\cat{sl}{r+1}{k}$ is a $\Z{r+1}$-graded extension of $\ad(\cat{sl}{r+1}{k})$ in order to compute the gauge auto-equivalences of $\cat{sl}{r+1}{k}$. To apply this Theorem we must show three requirements are satisfied. They are
\begin{enumerate}
\item the category $\ad( \cat{sl}{r+1}{k})$ must have no non-trivial gauge auto-equivalences,
\item if $z$ is an invertible object of $ \ad( \cat{sl}{r+1}{k} )$ such that there exists a non-trivial graded component $\mathcal{C}_g$ of $ \cat{sl}{r+1}{k}$ such that for all $X_g \in \cC_g$ we have $z\otimes X_g \cong X_g$, then $z\cong \mathbf{1}$, and 
\item the categories $\ad(\ad( \cat{sl}{r+1}{k} )    ) $ and $ \ad( \cat{sl}{r+1}{k})$ must be monoidally equivalent.
\end{enumerate}
Let us verify each of these requirements in turn.

\vspace{1em}

1) The first requirement follows from Corollary~\ref{cor:Agauge}.

\vspace{1em}

2) The category $ \cat{sl}{r+1}{k}$ has the $r$ non-trivial invertible objects $\{  (k \Lambda_i) : 1 \leq i \leq r\}$ which may, or may not live in  $\ad( \cat{sl}{r+1}{k})$ depending on $k$. Let $\cC_g$ be the $g$-graded component of $ \cat{sl}{r+1}{k}$, where $1\leq g \leq r+1$. If the invertible object $(k \Lambda_i)$ was a fixed point for all objects in $\cC_g$, then in particular it would fix the object $(\Lambda_g)$, so we would have
\[  (k \Lambda_i) \otimes (\Lambda_g) \cong  (\Lambda_g) .\]
Frobenius reciprocity implies $(k \Lambda_i)$ is a sub-object of $(\Lambda_g) \otimes  (\Lambda_g)^*$. However we can decompose this tensor product into simples to see
\[   (k \Lambda_i) \subset (\Lambda_g) \otimes  (\Lambda_g)^* \cong \bigoplus_{j = 0}^{\min(g, r+1 - g)}     (\Lambda_j + \Lambda_{r+1 - j}).\]
Hence the only possibility for $(k\Lambda_i)$ to fix $\Lambda_g$ is when $k = 2$ and $i = g =  \frac{r+1}{2}$. Thus if $k\neq 2$ we are done.

If $k=2$ then we need to find an object in $\cC_\frac{r+1}{2}$ that is not fixed by $\left(2\Lambda_{\frac{r+1}{2}}\right)$. Consider the object $\left(\Lambda_1 + \Lambda_{\frac{r-1}{2}}\right)$. We can directly compute that $\left(2\Lambda_{\frac{r+1}{2}}\right)\otimes \left(\Lambda_1 + \Lambda_{\frac{r-1}{2}}\right) \cong  \left(\Lambda_{\frac{r+3}{2}} + \Lambda_r\right)$. Hence  $\left(\Lambda_1 + \Lambda_{\frac{r-1}{2}}\right)$ is not fixed by $\left(2\Lambda_{\frac{r+1}{2}}\right)$ and we are done. Thus the second requirement is satisfied.

\vspace{1em}

3) For the third requirement, it suffices to note that $(\Lambda_1 +\Lambda_r)$ $\otimes$-generates $\ad( \cat{sl}{r+1}{k})$, and is contained in $\ad(\ad( \cat{sl}{r+1}{k} )    ) $ as $(\Lambda_1 +\Lambda_r)  \subset (\Lambda_1 +\Lambda_r) \otimes (\Lambda_1 +\Lambda_r)^*  $.

\vspace{1em}

As our three requirements are satisfied, we can apply \cite[Theorem 9.1]{1711.00645} to see 
\[      \text{Gauge}(\cat{sl}{r+1}{k}) = H^2(\Z{r+1} , \mathbb{C}^\times) = \{e\}.\]
\end{trivlist}
\end{proof}

Let us move on to realising the fusion ring automorphisms of $\cat{sl}{r+1}{k}$ as monoidal auto-equivalences. From the results of \cite{MR1887583} we know that 
\begin{align*}
| \operatorname{FusEq}(\cat{sl}{2}{2}) |  & = 1 \\
| \operatorname{FusEq}(\cat{sl}{2}{k}) | &= |\{0\leq a \leq 1  : 1+a \text{ is coprime to } 2\}| \text{ for $k \neq 2$} \\
  | \operatorname{FusEq}(\cat{sl}{r+1}{1}) | &= |\{0\leq a \leq r  : 1+a \text{ is coprime to } r+1\}| \\
    | \operatorname{FusEq}(\cat{sl}{r+1}{2}) | &= |\{0\leq a \leq r  : 1+2a \text{ is coprime to } r+1\}| \\ 
        | \operatorname{FusEq}(\cat{sl}{r+1}{k}) | &= 2|\{0\leq a \leq r  : 1+ka \text{ is coprime to } r+1\}|  \text{ for $k \geq 3$ and $r \geq 2$}.
\end{align*}
As we have shown that $\cat{sl}{r+1}{k}$ has no non-trivial gauge auto-equivalences, we know that $| \operatorname{FusEq}(\cat{sl}{r+1}{k}) |$ gives an upper bound on the order of $\TenAut(\cat{sl}{r+1}{k})$. We will soon see that this upper bound is sharp. Let us move on to constructing these $| \operatorname{FusEq}(\cat{sl}{r+1}{k}) |$ auto-equivalences.

\begin{lemma}
For every $r$ and $k$ we have that
   \[ \TenAut(\cat{sl}{r}{k})  \]
   is a group of order 
   \[  | \operatorname{FusEq}(\cat{sl}{r+1}{k}) |.\]        

\end{lemma}
\begin{proof}
\hspace{1em}

\begin{trivlist}\leftskip=2em
\item \textbf{Case $r=1$:}

The case of $r = 1$ is exactly \cite[Lemma 4.9]{MR3808052}.

\vspace{1em}

\item \textbf{Case $r\geq 2$:}

For $r \geq 1$ we first note that the charge-conjugation auto-equivalence of $\cat{sl}{r+1}{k}$ constructed in Corollary~\ref{cor:cc} gives a non-trivial auto-equivalence of $\cat{sl}{r+1}{k}$ for all $r \geq 2$. This charge-conjugation auto-equivalence is characterised by the fact that it maps
\[  \Lambda_1 \mapsto \Lambda_r.\]

Let us now fix an $0\leq a \leq r$ such that $1+ka$ is coprime to $r+1$, and use this to construct an auto-equivalence of $\cat{sl}{r+1}{k}$. We pick out the invertible object $k\Lambda_1 \in \cat{sl}{r+1}{k}$. This invertible object has order 
\[ M = r+1,\]
and self-braiding eigenvalue 
\[ e^{\pi i \frac{rk}{2(r+1)}}.\]

As $1 + ka$ is coprime to $r+1$, we apply the simple current construction from Lemma~\ref{lem:simplecurrent} to get monoidal auto-equivalences $\cF_{k\Lambda_1, a}$ of $\cat{sl}{r+1}{k}$. We have that
\[  \sigma_{\Lambda_1,k\Lambda_1}\circ \sigma_{k\Lambda_1,\Lambda_1} = e^{2\pi i \frac{r}{r+1}}.\]
Thus
\[   \cF_{k\Lambda_1,a}(\Lambda_1) = (k\Lambda_1)^{-ar} \otimes \Lambda_1 =  (k\Lambda_{a}) \otimes \Lambda_1  = \begin{cases}
(k-1) \Lambda_{a} + \Lambda_{a+1}, \text{ if } a \neq r\\
(k-1) \Lambda_{a} , \text{ if } a = r.
\end{cases}\]
Hence the monoidal auto-equivalences $\cF_{g,a}$ are pairwise distinct for all $a$. Thus these simple current auto-equivalences give 
\[   | \{ a: 1 + ka \text{ is coprime to } r+1\} | \]
distinct auto-equivalences of $\cat{sl}{r+1}{k}$. 

There is the possibility that these simple current auto-equivalences overlap with the charge conjugation auto-equivalences. To investigate when this occurs, we compute
\[\cF_{g,a}(\Lambda_1) = \Lambda_r\]
if and only if $k=2$ and $a = r$, or if $k=1$ and $a = r-1$. Thus, when paired with the charge-conjugation auto-equivalence of $\cat{sl}{r+1}{k}$ we have constructed
\[   | \{ a: 1 + ka \text{ is coprime to } r+1\} | \]
distinct auto-equivalences of $\cat{sl}{r+1}{k}$ when $k \in  \{1,2\}$, and 
\[      2| \{ a: 1 + ka \text{ is coprime to } r+1\} | \]
distinct auto-equivalences of $\cat{sl}{r+1}{k}$ when $k \geq 3$. Thus we have realised the upper bound on $\TenAut(\cat{sl}{r+1}{k})$ given by $ | \operatorname{FusEq}(\cat{sl}{r}{k}) | $.

\end{trivlist}
\end{proof}

Let us explicitly describe these monoidal auto-equivalences of $\cat{sl}{r+1}{k}$, by describing how they act on the object $\Lambda_1$. It follows from \cite{MR1237835} that this completely describes the auto-equivalence.

\begin{rmk}\label{lab:conc}

For an $a \in \{0\leq a \leq r  : 1+ka \text{ is coprime to } r+1\}$ and choice of sign $\varepsilon \in \{+,-\}$, we get a monoidal auto-equivalence of $\cat{sl}{r+1}{k}$ that sends
\[  \Lambda_1 \mapsto \begin{cases}
(k-1) \Lambda_{a} + \Lambda_{a+1} &\text{ if } a \neq r \text{ and } \varepsilon = +,\\
(k-1) \Lambda_{a}&  \text{ if } a = r, \text{ and } \varepsilon = +, \\
(k-1) \Lambda_{r+1  -  a} + \Lambda_{r- a}& \text{ if } a \neq r \text{ and } \varepsilon = -,\\
(k-1) \Lambda_{r+1 - a} &  \text{ if } a = r, \text{ and } \varepsilon = -.
\end{cases}\]
\end{rmk}

The composition of the monoidal auto-equivalences of $\cat{sl}{r+1}{k}$ is determined in Appendix~\ref{app:terry}. Here it is found the group structure is as given in Theorem~\ref{thm:main}.

Finally, we investigate which of the monoidal auto-equivalences of $\cat{sl}{r+1}{k}$ are braided.

\begin{lemma}
For all $r$ and $k$ we have 
\[    \BrAut( \cat{sl}{r+1}{k})    =  \Z{2}^{c+p+t}\]
where
\[ c =  \begin{cases}
-1, &\text{ if } r = 1 \text{ and } k=2\\
0, &\text{ if } k \leq 2 \\
1 ,& \text{ if } k \geq 3,
\end{cases}\]
$p$ is the number of distinct odd primes dividing $r+1$, and 
\[ t =  \begin{cases}
0, \text{ if either } r \text{ is even, or } r \text{ is odd and } k \equiv 0 \pmod 4, \text{ or } k \text{ is odd and } r \equiv 1 \pmod 4, \\
1 , \text{ otherwise}.
\end{cases}\]
\end{lemma}
\begin{proof}
The simple current auto-equivalence $\cF_{g,a}$ is braided if and only if 
\[   a^2 - \frac{rka}{2} \equiv 0 \pmod {r+1}.\]
From the analysis of \cite{MR1887583}, we have that there are $2^{p+t}$ such choices for $a$, where $p$ and $t$ are in the statement of the Lemma. It is shown in \cite{MR1887583} that each of these braided $\cF_{g,a}$ has order two. 

Further, the charge-conjugation auto-equivalence is braided, giving us an additional order two braided auto-equivalence when $r \geq 2$, giving the $c$ term in the Lemma. The special case of $c=-1$ when $r = 1$ and $k = 2$ is to uncount the simple current auto-equivalence of $\cat{sl}{2}{2}$ coming from $a = 1$ that is naturally isomorphic to the identity.

\end{proof}

\section{Auto-equivalences for Lie type $B$}

In this section we compute the monoidal, and braided auto-equivalences of the categories $\cat{so}{2r+1}{k}$. The outline of our computations is as follows. 

We begin with the gauge auto-equivalences of $\cat{so}{2r+1}{k}$. We know that the planar algebra generated by the object $\Lambda_1 \boxtimes \mathbb{C}_{\operatorname{sgn}} \in \cat{so}{2r+1}{k}\boxtimes \Rep(\Z{2})$ is isomorphic to $\overline{\text{BMW}( q^2 , q^{4r} )}$ with $q = e^\frac{2 \pi i}{4(2r-1 + k)}$. However, the object $\Lambda_1 \boxtimes \mathbb{C}_{\operatorname{sgn}}$ only generates the subcategory \[\ad(\cat{so}{2r+1}{k})\boxtimes \Rep(\Z{2}) \subset  \cat{so}{2r+1}{k}\boxtimes \Rep(\Z{2}).\] Thus the planar algebra automorphism group of $\overline{\text{BMW}( q^2 , q^{4r} )}$ will only give us information about the gauge auto-equivalences of $\ad(\cat{so}{2r+1}{k})\boxtimes \Rep(\Z{2})$. 

A straightforward computation shows that the planar algebra $\overline{\text{BMW}( q^2 , q^{4r} )}$ has trivial automorphism group, except in the special case where $k = 2r+1$. A further computation shows that this non-trivial planar algebra automorphism corresponds to an auto-equivalence of
\[\ad(\cat{so}{2r+1}{k})\boxtimes \Rep(\Z{2})\]
 which exchanges simple objects. Thus, the category $\ad(\cat{so}{2r+1}{k})\boxtimes \Rep(\Z{2})$ has no non-trivial gauge auto-equivalences. From this fact, it immediately follows that $\ad(\cat{so}{2r+1}{k})$ also has no non-trivial gauge auto-equivalences. As $\cat{so}{2r+1}{k}$ is a $\Z{2}$-graded extension of $\ad(\cat{so}{2r+1}{k})$, we can now apply \cite[Theorem 9.1]{1711.00645} to see that $\cat{so}{2r+1}{k}$ also has no non-trivial gauge auto-equivalences.

The results of \cite{MR1887583} compute the fusion ring automorphisms of $\cat{so}{2r+1}{k}$ as: 
\begin{align*}
\{e\} & \text{   if } k=1, \\
\Z{2} \times (\Z{2r+1}^\times / \{\pm 1\})& \text{   if } k=2, \text{ and } \\
\Z{2} & \text{   if } k> 2. 
\end{align*}

Thus to complete the computation of the auto-equivalences of $\cat{so}{2r+1}{k}$ we have to determine which of the above fusion ring automorphisms are realised as either monoidal, or braided auto-equivalences. The $k=1$ case is trivial, as there are no non-trivial auto-equivalences. The $k>2$ case is also rather easy, as the single non-trivial fusion ring automorphism is always realised as a simple current automorphism. The difficult case is $k=2$. Here we proceed by using the modular data of $\cat{so}{2r+1}{2}$ to bound the size of the braided auto-equivalence group of $\cat{so}{2r+1}{2}$, which turns out to be significantly smaller than $\Z{2} \times (\Z{2r+1}^\times / \{\pm 1\})$. We then perform a trick to leverage this bound to also determine a bound for the size of the monoidal auto-equivalence group of $\cat{so}{2r+1}{2}$. To realise these upper bounds we observe that the de-equivariantization of $\cat{so}{2r+1}{2}$ is a Tambara-Yamagami category, which we can explicitly identify. Using Lemma~\ref{lem:TYautos} we can compute the auto-equivalence group of this Tambara-Yamagami category. We can then apply the functorality of equivariantization to extend these auto-equivalences of the Tambara-Yamagami category, up to auto-equivalences of $\cat{so}{2r+1}{2}$. These auto-equivalences realise the upper bound we computed earlier, hence the computation is complete.

With the high level argument in mind, let us proceed with the details. We begin by computing the planar algebra automorphisms of $\overline{\text{BMW}(q,r)}$. This computation will also be useful later on, for our computations in the Lie type $C$ section.

 \begin{lemma}\label{lem:paBMW}
The planar algebra $\overline{\text{BMW}(q,r)}$ has no non-trivial automorphisms unless $r = \pm i$. If $r = \pm i$, then there exists a single non-trivial automorphism sending
\[      \raisebox{-.5\height}{ \includegraphics[scale = .5]{BMWpositivecrossing}} \quad \mapsto \quad -  \raisebox{-.5\height}{ \includegraphics[scale = .5]{BMWnegitivecrossing}} .\]
This automorphism is not braided, and exchanges non-isomorphic simple objects in the idempotent completion.
\end{lemma}
\begin{proof}
As the planar algebra $\overline{\text{BMW}(q,r)}$ is generated by the single element 
\[   \raisebox{-.5\height}{ \includegraphics[scale = .5]{BMWpositivecrossing}} \in \overline{\text{BMW}(q,r)}_4,\]
any planar algebra automorphism of $\overline{\text{BMW}(q,r)}$ is completely determined by its value on this element. The box space $\overline{\text{BMW}(q,r)}_4$ is 3-dimensional, so $\phi \in  \Aut(\overline{\text{BMW}(q,r)})$ is determined by three scalers $\alpha, \beta, \gamma \in \mathbb{C}$, with
\[  \phi\left( \raisebox{-.5\height}{ \includegraphics[scale = .5]{BMWpositivecrossing}} \right ) \quad =\quad  \alpha \raisebox{-.5\height}{ \includegraphics[scale = .5,angle = 90]{cupcap}} \quad  + \quad \beta \raisebox{-.5\height}{ \includegraphics[scale = .5]{cupcap}} \quad  + \quad \gamma \raisebox{-.5\height}{ \includegraphics[scale = .5]{BMWpositivecrossing}} .\]

Solving for the relation
\[\raisebox{-.5\height}{ \includegraphics[scale = .5]{BMWpositivecrossing}}  - \raisebox{-.5\height}{ \includegraphics[scale = .5]{BMWnegitivecrossing}} = (q-q^{-1})\left( \raisebox{-.5\height}{ \includegraphics[scale = .5,angle = 90]{cupcap}} - \raisebox{-.5\height}{ \includegraphics[scale = .5]{cupcap}}   \right)\]
yields the equation
\[   \beta - \alpha = (\gamma-1)(q- q^{-1}).\]

Solving for the Reidemeister 2 relation yields the additional three equations
\begin{align*}
\alpha\beta  + \gamma^2 - \alpha\gamma(q - q^{-1}) &= 1, \\
\alpha^2 + \beta^2 + \alpha\beta\left(\frac{r - r^{-1}}{q - q^{-1}}+1\right) + \frac{\alpha\gamma}{r} + \gamma\beta r  + \gamma\alpha(q - q^{-1})&= 0, \\
\gamma(\alpha + \beta)&= 0. 
\end{align*}

Solving for all four of these above equations yields two solutions:
\[    (\alpha, \beta, \gamma) \quad =\quad (0,0,1) \quad \text{or} \quad (q-q^{-1}, -(q - q^{-1}), -1).\]
The first of these solutions is just the identity, and the second we can write as
\[     \phi\left( \raisebox{-.5\height}{ \includegraphics[scale = .5]{BMWpositivecrossing}}\right) \quad = \quad -  \raisebox{-.5\height}{ \includegraphics[scale = .5]{BMWnegitivecrossing}}.\]

Solving for the relation
\[\raisebox{-.5\height}{ \includegraphics[scale = .5]{negTwist}} = r \raisebox{-.5\height}{ \includegraphics[scale = .5]{G2iso}}\]
reveals that the second is an automorphism only if $ r = -r^{-1}$, i.e. only if $r = \pm i$.

Clearly the possible non-trivial $\phi$ satisfies Reidiemester 3, thus $\phi$ is a planar algebra automorphism of $\text{BMW}(q,\pm i)$, and hence of $\overline{\text{BMW}(q,\pm i)}$ by Proposition~\ref{prop:neg}.

\vspace{1em}

As the braiding in $\overline{\text{BMW}(q,\pm i)}$ is the crossing generator, the non-trivial automorphism $\phi$ is clearly not braided.

\vspace{1em}

The idempotent completion of $\overline{\text{BMW}(q,\pm i)}$ contains the two distinct simple objects 
\[ \pm \frac{iq}{(q \pm i)^2}\raisebox{-.5\height}{ \includegraphics[scale = .5,angle = 90]{cupcap}} \mp \frac{iq}{(q \mp i)(q \pm i)^2}\raisebox{-.5\height}{ \includegraphics[scale = .5]{BMWpositivecrossing}} \mp \frac{q^2}{(q \mp i)(q\pm i)^2}\raisebox{-.5\height}{ \includegraphics[scale = .5]{BMWnegitivecrossing}}    \]
and 
\[\frac{\pm iq}{(q \pm i)^2}\raisebox{-.5\height}{ \includegraphics[scale = .5,angle = 90]{cupcap}} \pm \frac{q^2}{(q\mp i)(q \pm i)^2} \raisebox{-.5\height}{ \includegraphics[scale = .5]{BMWpositivecrossing}} \pm \frac{iq}{(q \mp i)(q \pm i)^2}\raisebox{-.5\height}{ \includegraphics[scale = .5]{BMWnegitivecrossing}} .   \]
The non-trivial automorphism $\phi$ exchanges these two idempotents, and thus exchanges the corresponding non-isomorphic simple objects in the idempotent completion.
\end{proof}

Specialising this Lemma to the $\cat{so}{2r+1}{k}$ case gives the following Corollary.

\begin{cor}\label{cor:gaugeB}
The category $\cat{so}{2r+1}{k}$ has no non-trivial gauge auto-equivalences.
\end{cor}
\begin{proof}
Let us begin with the special case of $k=1$, which requires special treatment. The categories $\cat{so}{2r+1}{1}$ have Ising fusion rules, hence these categories have no non-trivial gauge auto-equivalences by \cite{Cain-normal}.

In the general case of $k\geq 2$, we will use \cite[Theorem 9.1]{1711.00645}. This Theorem allows us to leverage the fact that $\cat{so}{2r+1}{k}$ is a $\Z{2}$-graded extension of $\ad(\cat{so}{2r+1}{k})$, in order to compute the gauge auto-equivalences of $\cat{so}{2r+1}{k}$. To apply this Theorem we must show three requirements are satisfied. They are
\begin{enumerate}
\item The category $\ad( \cat{so}{2r+1}{k})$ must have no non-trivial gauge auto-equivalences,

\item if $z$ is an invertible object of $ \ad( \cat{so}{2r+1}{k} )$ such that for every object $X$ in the non-trivially graded component of $ \cat{so}{2r+1}{k}$ we have $z\otimes X_g \cong X_g$, then $z\cong \mathbf{1}$, and

\item the categories $\ad(\ad( \cat{so}{2r+1}{k} )    ) $ and $ \ad( \cat{so}{2r+1}{k})$ must be monoidally equivalent.
\end{enumerate}

Let us verify each of these requirements in turn.

\vspace{1em}

1) From Lemma~\ref{lem:gaugeauto}, every gauge auto-equivalence of $\ad(\cat{so}{2r+1}{k})\boxtimes \Rep(\Z{2})$ will appear as a planar algebra automorphism of the planar algebra generated by $X$, where $X$ is any $\otimes$-generator of $\ad(\cat{so}{2r+1}{k})\boxtimes \Rep(\Z{2})$. The object $\Lambda_1\boxtimes \mathbb{C}_{\operatorname{sgn}}$ is a particularly nice $\otimes$-generator of $\ad(\cat{so}{2r+1}{k})\boxtimes \Rep(\Z{2})$, as we have the following isomorphism of planar algebras from Equation~\eqref{eq:so}
\[ \text{PA}(\ad( \cat{so}{2r+1}{k})\boxtimes \Rep(\Z{2}) , \Lambda_1\boxtimes \mathbb{C}_{\operatorname{sgn}} ) \cong \overline{\text{BMW}( q^2 , q^{4r} )} \text{ with } q = e^\frac{2 \pi i}{4(2r-1 + k)}.\]

From Lemma~\ref{lem:paBMW}, the planar algebra $\overline{\text{BMW}( q^2 , q^{4r} )}$ only has a single non-trivial automorphism when $k = 2r+1$, which corresponds to an auto-equivalence of $\ad( \cat{so}{2r+1}{k}\boxtimes \Rep(\Z{2})$  that exchanges distinct simple objects. Thus, the category $\ad( \cat{so}{2r+1}{k})\boxtimes \Rep(\Z{2}) $ has no non-trivial gauge auto-equivalences. As a non-trivial gauge auto-equivalence of $\ad( \cat{so}{2r+1}{k})$ would canonically extend to a non-trivial gauge auto-equivalence of $\ad( \cat{so}{2r+1}{k})\boxtimes \Rep(\Z{2}) $, we see that $\ad( \cat{so}{2r+1}{k})$ has no non-trivial gauge auto-equivalences.

\vspace{1em}

2) The only non-trivial invertible object of $\ad( \cat{so}{2r+1}{k})$ is the object $k \Lambda_1$. The non-trivial graded component of $\cat{so}{2r+1}{k}$ contains the object $\Lambda_r$, and we have the fusion
\[     ( k \Lambda_1) \otimes (\Lambda_r) \cong (  (k-1)\Lambda_1 + \Lambda_r).\]
Thus, when $k>1$ (a case we have already covered), the invertible object $k \Lambda_1$ is not a fixed point for the non-trivially graded component of $\cat{so}{2r+1}{k}$.

\vspace{1em}

3) To show that $\ad(\ad( \cat{so}{2r+1}{k} )    )  \simeq \ad( \cat{so}{2r+1}{k})$, we have to break up into two cases depending on the parity of $k$. 

If $k$ is even, then consider the object $k\Lambda_r \in  \ad( \cat{so}{2r+1}{k})$. We have the fusion
\[ (k\Lambda_r) \otimes (k\Lambda_r)^* \supset \Lambda_1, \]
therefore $\Lambda_1 \in \ad(\ad( \cat{so}{2r+1}{k} )    )$. As $\Lambda_1$ $\otimes$-generates $ \ad( \cat{so}{2r+1}{k})$, we thus have 
\[\ad(\ad( \cat{so}{2r+1}{k} )    )  \simeq \ad( \cat{so}{2r+1}{k}).\]

If $k$ is odd and not equal to $1$, then consider the object $(k-1)\Lambda_r  \in  \ad( \cat{so}{2r+1}{k})$. We have the fusion
\[ ((k-1)\Lambda_r) \otimes ((k-1)\Lambda_r)^* \supset \Lambda_1, \]
therefore $\Lambda_1 \in \ad(\ad( \cat{so}{2r+1}{k} )    )$. As $\Lambda_1$ $\otimes$-generates $ \ad( \cat{so}{2r+1}{k})$, we thus have 
\[\ad(\ad( \cat{so}{2r+1}{k} )    )  \simeq \ad( \cat{so}{2r+1}{k}).\]

\vspace{1em}

As all three conditions are satisfied, we can apply \cite[Theorem 9.1]{1711.00645} to see that the gauge auto-equivalence group of $\cat{so}{2r+1}{k}$ is isomorphic to $H^2(\Z{2}, \mathbb{C}^\times)= \{e\}$.
\end{proof}

Now that we have shown the categories $\cat{so}{2r+1}{k}$ have no non-trivial gauge auto-equivalences, we next have to determine which fusion ring automorphisms of $\cat{so}{2r+1}{k}$ are realised as either monoidal or braided auto-equivalences. There are three cases to work through, $k=1$, $k=2$, and $k\geq 3$.

\subsection*{Case: $k=1$} \hspace{1em}

When $k= 1$, there are no non-trivial fusion ring automorphisms of $\cat{so}{2r+1}{k}$, thus 
\[ \BrAut(\cat{so}{2r+1}{1} ) =  \TenAut(\cat{so}{2r+1}{1} )  =   \text{Gauge}(\cat{so}{2r+1}{1}) = \{e\},\]
as in the statement of Theorem~\ref{thm:main}.

\subsection*{Case: $k \geq 3$}  \hspace{1em}

When $k \geq 3$, there is a single non-trivial fusion ring automorphism of $\cat{so}{2r+1}{k}$, characterised by the fact it maps
\[    (\Lambda_r) \leftrightarrow ((k-1)\Lambda_1 + \Lambda_r).\]

We will show this automorphism is always realised by a monoidal auto-equivalence of $\cat{so}{2r+1}{k}$, which we construct as a simple current auto-equivalence.

Consider the order two invertible object $(k\Lambda_1) \in \cat{so}{2r+1}{k}$. This invertible object has self braiding eigenvalue $1$ is $k$ is even, and $-1$ if $k$ is odd. Thus we can use Lemma~\ref{lem:simplecurrent} to get a monoidal auto-equivalence of $\cat{so}{2r+1}{k}$ when $k$ is even, and a braided auto-equivalence when $k$ is odd. We compute that this auto-equivalence maps 
\[    (\Lambda_r) \mapsto (\Lambda_r)\otimes(k\Lambda_1) \cong ((k-1)\Lambda_1 + \Lambda_r) .\]

Thus we have, for $k \geq 3$, 
\[   \TenAut( \cat{so}{2r+1}{k}  ) = \Z{2} \quad \text{ and } \quad  \BrAut( \cat{so}{2r+1}{k}  )  = \begin{cases}
\{e\}, \text{ if $k$ is even,}\\
\Z{2}, \text{   if $k$ is odd.}\end{cases}\]
As in the statement of Theorem~\ref{thm:main}.

\subsection*{Case: $k=2$} \hspace{1em}

By far and away the most difficult case are the categories $\cat{so}{2r+1}{2}$, due to the existence of many exotic fusion ring automorphisms. To simplify notation we give these categories explicit presentations. We label the simple objects of $\cC(\mathfrak{so}_{2r+1} , 2)$ as:
\[  1 , Z ,X_1, X_2, Y_i \]
with $1\leq i \leq r$. The objects $1$ and $Z$ have dimension 1, the objects $X_1$, and $X_2$ have dimension $\sqrt{2r+1}$, and the objects $Y_i$ all have dimension $2$. The commutative fusion rules are given by
\begin{align*}
Z\otimes Z &\cong 1 \\ 
Z\otimes X_1 &\cong X_2 \\ 
Z\otimes X_2 &\cong X_1 \\
Z\otimes Y_i &\cong Y_i \\
X_1 \otimes X_1 \cong X_2\otimes X_2 &\cong 1 \oplus \bigoplus_i Y_i \\
X_1\otimes X_2 &\cong Z \oplus \bigoplus_i Y_i \\
Y_i \otimes Y_j &\cong  Y_{\min(i+j , 2r+1 - i - j)} \oplus Y_{|i-j|}    \quad    i\neq j \\
Y_i \otimes Y_i&\cong 1 + Z + Y_{\min(2i, 4r+2 - 2i)}.
\end{align*} The twists of the simple objects $1$ and $Z$ are 1, the twists of the objects $X_1$ and $X_2$ have 8 fold symmetry with respect to $r$, which we present in Table~\ref{tab:twistB}. The twists of the simple objects $Y_j$ are given by the formula
\[   t_{Y_j} = e^{2\mathbf{i}\pi  \frac{ j^2r}{2r+1}}. \]

\begin{table}[h!]
   
\centering 
    
    \begin{tabular}{c | c c}
    	\toprule
			$r \mod 8$ & $t_{X_1}$ 					  & $t_{X_2}$  \\
			\midrule
			$0$  	   & 1       					  &   -1  	\\
			$1$  	   & $e^{2\mathbf{i}\pi \frac{1}{8}}$       &   $e^{2\mathbf{i}\pi \frac{5}{8}}$  	\\
			$2$  	   & $e^{2\mathbf{i}\pi \frac{1}{4}}$       &   $e^{2\mathbf{i}\pi \frac{3}{4}}$  	\\
			$3$  	   & $e^{2\mathbf{i}\pi \frac{3}{8}}$       &   $e^{2\mathbf{i}\pi \frac{7}{8}}$  	\\
			$4$  	   & -1      					  &   1 	\\
		    $5$  	   & $e^{2\mathbf{i}\pi \frac{5}{8}}$       &   $e^{2\mathbf{i}\pi \frac{1}{8}}$  	\\
			$6$  	   & $e^{2\mathbf{i}\pi \frac{3}{4}}$       &   $e^{2\mathbf{i}\pi \frac{1}{4}}$ 	\\
			$7$  	   & $e^{2\mathbf{i}\pi \frac{7}{8}}$       &   $e^{2\mathbf{i}\pi \frac{3}{8}}$  	\\
			          	                 
    	\bottomrule

	    \end{tabular}
	    	\caption{\label{tab:twistB}Twists of the simple objects $X_1$ and $X_2$ in $\cC(\mathfrak{so}_{2r+1},2)$ }
\end{table}

The group of automorphisms of the $\cC(\mathfrak{so}_{2r+1},2)$ fusion ring is isomorphic to $\Z{2} \times \Z{2r+1}^\times / (\pm)$. The automorphism corresponding to the $\Z{2}$ factor simply sends
\[    X_1 \leftrightarrow X_2, \]
and fixes all other objects. The automorphism corresponding to an element $m \in \Z{2r+1}^\times / (\pm)$ sends
\[  Y_i \mapsto Y_{\min( mi \pmod {2r+1}) ,  -mi \pmod {2r+1})    }, \]
and fixes all other objects.

As there are no non-trivial gauge auto-equivalences of $\cC(\mathfrak{so}_{2r+1} , 2)$, we have that
\[  \TenAut(\cC(\mathfrak{so}_{2r+1} , 2))\subseteq  \Z{2} \times \Z{2r+1}^\times / (\pm).\] 
It turns out that this upper bound for $\TenAut(\cC(\mathfrak{so}_{2r+1} , 2))$ is not sharp, as there are fusion ring automorphisms of $\cC(\mathfrak{so}_{2r+1} , 2)$ that are not realisable as monoidal auto-equivalences. We begin our computations by giving improved bounds. While not \textit{a-priori}, these new bounds will be sharp.

\begin{lemma}\label{lem:bupper}
We have
\[ \BrAut( \cC(\mathfrak{so}_{2r+1} , 2)  ) \subseteq \{n\in  \Z{2r+1}^\times : n^2 = 1 \} / (\pm).\]
\end{lemma}
\begin{proof}
To determine an upper bound for the group of braided auto-equivalences of $\cC(\mathfrak{so}_{2r+1},2)$ we calculate the group of fusion ring automorphisms that preserve the twists of the simple objects.

 As $X_1$ and $X_2$ always have differing twists, the fusion ring automorphism $X_1 \leftrightarrow X_2$ does not preserve twists. 
 
 Let $n \in \Z{2r+1}^\times / (\pm)$, and consider the fusion ring automorphism 
 \[  Y_i \mapsto Y_{\min( ni \pmod {2r+1}) ,  -ni \pmod {2r+1})    }. \]
A calculation shows that this fusion ring automorphism preserves the twists of all the simple objects if and only if $n^2 \equiv 1 \pmod {2r+1}$.
\end{proof}

Using similar techniques, we can prove the following Lemma, giving a necessary condition for there to be a braided equivalence between $\cC(\mathfrak{so}_{2r+1} , 2)$ and its braided opposite. While this Lemma may seem irrelevant, we will soon find an important use for it when trying to improve the bound on $\TenAut(\cC(\mathfrak{so}_{2r+1} , 2))$.
\begin{lemma}\label{lem:bopeq}
There exists a braided equivalence $\cC(\mathfrak{so}_{2r+1} , 2) \to \cC(\mathfrak{so}_{2r+1} , 2) ^\text{rev}$ only if $\Z{2r+1}^\times$ contains an element $n$ such that $n^2 = -1$.
\end{lemma}
\begin{proof}
Let $\mathcal{F} : \cC(\mathfrak{so}_{2r+1} , 2) \to \cC(\mathfrak{so}_{2r+1} , 2) ^\text{rev}$ a braided equivalence. As the underlying monoidal categories of $\cC(\mathfrak{so}_{2r+1} , 2)$ and  $\cC(\mathfrak{so}_{2r+1} , 2) ^\text{rev}$ are the same, we have that 
\[  \mathcal{F}(Y_i)  = Y^\text{rev}_{\min( ni \pmod {2r+1}) ,  -ni \pmod {2r+1})    }\]
for some $n \in \Z{2r+1}^\times$. In particular we have that $Y_1 =  Y^\text{rev}_{\min( n \pmod {2r+1}) ,  -n \pmod {2r+1})}$. As $\mathcal{F}$ preserves twists we get the equation
\[  e^{2\mathbf{i}\pi  \frac{r}{2r+1}} = e^{2\mathbf{i}\pi  \frac{-n^2r}{2r+1}}, \]
which implies that $n^2 = -1 \pmod{2r+1}$.
\end{proof}

We now return to improving the upper bound on $\TenAut(\cat{so}{2r+1}{2})$. We achieve this bound by a combinatorial argument, counting the size of the group $\BrAut(\mathcal{Z}(\cat{so}{2r+1}{2}))$ in two ways. On one hand we give an upper bound for $\BrAut(\mathcal{Z}(\cat{so}{2r+1}{2}))$ by considering fusion rules and twists of the Drinfeld center. On the other hand we count the size of $\BrAut(\mathcal{Z}(\cat{so}{2r+1}{2}))$ using the following equation
\[   | \BrAut(\mathcal{Z}(\cat{so}{2r+1}{2})) | = |  \TenAut( \cat{so}{2r+1}{2} )| \cdot      |  \BrAut(\cat{so}{2r+1}{2})|. \]
This equation follows from the fact that the data for an invertible bimodule over $\cat{so}{2r+1}{2}$ consists of an invertible $\cat{so}{2r+1}{2}$ module, and an outer auto-equivalence of $\cat{so}{2r+1}{2}$ \cite{MR2909758}. As the category $\cat{so}{2r+1}{2}$ is braided and has no non-trivial gauge auto-equivalences, we have that $\operatorname{Out}( \cat{so}{2r+1}{2} ) =\TenAut( \cat{so}{2r+1}{2} ) $. Hence we can combine the two ways to count the size of $\BrAut(\mathcal{Z}(\cat{so}{2r+1}{2}))$ to give an upper bound for $\TenAut( \cat{so}{2r+1}{2} ) $. Working through this approach gives the following Lemma.

\begin{lemma}\label{lem:tenbound}
We have 
\[
 |\TenAut( \cat{so}{2r+1}{2})|  \leq
  \begin{cases}
                                   4|\BrAut( \cat{so}{2r+1}{2}) |& \text{if $\Z{2r+1}^\times$ contains an element $n$ such that $n^2 = -1$} \\
 								   2|\BrAut( \cat{so}{2r+1}{2})| & \text{otherwise.}
  \end{cases}
\]
\end{lemma}
\begin{proof}
Let us begin by showing that the category $\mathcal{Z}(\cat{so}{2r+1}{2})$ has no non-trivial braided gauge auto-equivalences. This fact will allow us to apply combinatorial techniques for the remainder of the proof.

Aiming for a contradiction, suppose $\mathcal{Z}(\cat{so}{2r+1}{2})$ had a non-trivial braided gauge auto-equivalence $\mathcal{F}$. Then \cite[Theorem 1.1]{MR2677836} would imply the existence of a non-trivial invertible bimodule $\mathcal{M}$ over $\cat{so}{2r+1}{2}$. As $\cF$ is gauge, \cite[Proposition 3.1]{1902.06165} tells us that the rank of $\mathcal{M}$ is equal to the rank of $ \cat{so}{2r+1}{2}$. The modular category $ \cat{so}{2r+1}{2}$ has no non-trivial gauge auto-equivalences via Corollary~\ref{cor:gaugeB}, and so in particular it has no non-trivial braided gauge auto-equivalences. Thus \cite[Theorem 5.2]{MR2677836} combined with \cite{1807.06131} show that the only invertible module category over $ \cat{so}{2r+1}{2}$ with maximal rank is the trivial module. Hence $\mathcal{M}$ is equivalent to $ \cat{so}{2r+1}{2}$ as a left $ \cat{so}{2r+1}{2}$-module, which then implies 
\[\mathcal{M} \simeq \cat{so}{2r+1}{2}_{\mathcal{G}}\]
as an $ \cat{so}{2r+1}{2}$-bimodule, with $\mathcal{G}$ an outer auto-equivalence of  $ \cat{so}{2r+1}{2}$. As $\mathcal{M}$ is non-trivial, the auto-equivalence $\mathcal{G}$ must be non-trivial, and in particular must be non-gauge by Corollary~\ref{cor:gaugeB}. However, we can now use \cite[Section 3]{1603.04318} to see that $\mathcal{G}\boxtimes \mathcal{G}^\text{op}$ and $\mathcal{F}$ are naturally isomorphic, up to an inner auto-equivalence of $\mathcal{Z}(\cat{so}{2r+1}{2})$. Due to the braiding on $\mathcal{Z}(\cat{so}{2r+1}{2})$, any inner auto-equivalence must be gauge, and so the functors $\mathcal{G}\boxtimes \mathcal{G}^\text{op}$ and $\mathcal{F}$ must act the same on objects. As $\mathcal{G}$ is non-gauge, we thus have $\mathcal{F}$ is also non-gauge, giving the contradiction. 

As the category $\mathcal{Z}(\cat{so}{2r+1}{2})$ has no non-trivial braided gauge auto-equivalences, we can bound the size of $\BrAut(\mathcal{Z}(\cat{so}{2r+1}{2}))$ by considering fusion rules and twists, as the group $\BrAut(\mathcal{Z}(\cat{so}{2r+1}{2}))$ is a subgroup of the group of fusion ring automorphisms of $\mathcal{Z}(\cat{so}{2r+1}{2})$ that preserve twists. We will refer to such fusion ring automorphisms as \textit{braided fusion ring automorphisms} for the remainder of the proof.

Consider the objects $X_1\boxtimes \mathbf{1}^\text{rev}$ and $\mathbf{1}\boxtimes X_1^\text{rev}$ in 
\[  \mathcal{Z}(\cat{so}{2r+1}{2}) =\cat{so}{2r+1}{2} \boxtimes \cat{so}{2r+1}{2}^\text{rev}.\] 
These objects $\otimes$-generate the factors $\cat{so}{2r+1}{2}$ and $\cat{so}{2r+1}{2}^\text{rev}$ respectively. Thus any braided fusion ring automorphism of $\mathcal{Z}(\cat{so}{2r+1}{2})$ that fixes these objects must be of the form $\mathcal{F} \boxtimes \mathcal{G}$ for $\mathcal{F}$ a braided fusion ring automorphism of $\cat{so}{2r+1}{2}$, and $\mathcal{G}$ a braided fusion ring automorphism of $\cat{so}{2r+1}{2}^\text{rev}$.

We now consider $\mathcal{F}$ a $\mathcal{F}$ a braided fusion ring automorphism of $\mathcal{Z}(\cat{so}{2r+1}{2})$ that moves either of the objects $X_1\boxtimes \mathbf{1}^\text{rev}$ and $\mathbf{1}\boxtimes X_1^\text{rev}$. Simply considering dimensions shows that there are eight possible objects that each of these two objects can be sent to, they are
\[   X_1\boxtimes \mathbf{1}^\text{rev}, X_1\boxtimes Z^\text{rev},X_2\boxtimes \mathbf{1}^\text{rev}, X_2\boxtimes Z^\text{rev} ,\mathbf{1}\boxtimes X_1^\text{rev}, Z\boxtimes X_1^\text{rev},\mathbf{1}\boxtimes X_2^\text{rev}, \text{ and } Z\boxtimes X_2^\text{rev}.  \]
By considering dimensions and twists, we can determine that $\mathcal{F}(  X_1 \boxtimes X_1^\text{rev} )$ must be either $X_1 \boxtimes X_1^\text{rev}$ or $X_2 \boxtimes X_2^\text{rev}$. With this information, we are able to compute Table~\ref{tab:posmaps}, showing possibilities for $\mathcal{F}( \mathbf{1} \boxtimes X_1^\text{rev})$, given $\mathcal{F}( X_1 \boxtimes \mathbf{1}^\text{rev})$. In particular the twists of the objects for the two possibilities of $\mathcal{F}( \mathbf{1}\boxtimes X_1^\text{rev} )$ always differ by $-1$. This shows that $\mathcal{F}$ is completely determined by where it sends the object $X_1\boxtimes \mathbf{1}^\text{rev}$. This gives us a naive bound
\[  |\BrAut(\mathcal{Z}( \cat{so}{2r+1}{2}))|  \leq 8 |\BrAut( \cat{so}{2r+1}{2}) |^2.    \]  
To improve this bound we now split the proof up in to cases.

\begin{table}[h!]
   
\centering 
    
    \begin{tabular}{c  c c}
    	\toprule
			$\mathcal{F}( X_1 \boxtimes \mathbf{1}^\text{rev})$ &    			      & $\mathcal{F}( \mathbf{1}\boxtimes X_1^\text{rev} )$ \\
			\midrule
			$X_1\boxtimes \mathbf{1}^\text{rev}$  	   & $\implies$ &   $\mathbf{1}\boxtimes X_1^\text{rev}$  \text{ or } $Z\boxtimes X_2^\text{rev}$	\\
			$ X_1\boxtimes Z^\text{rev}$  	  		   & $\implies$ &   $\mathbf{1}\boxtimes X_2^\text{rev}$  \text{ or } $Z\boxtimes X_1^\text{rev}$ 	\\
			$ X_2\boxtimes \mathbf{1}^\text{rev}$  	   & $\implies$ &   $\mathbf{1}\boxtimes X_2^\text{rev}$  \text{ or } $Z\boxtimes X_1^\text{rev}$ 	\\
			$X_2\boxtimes Z^\text{rev}$  	   		   & $\implies$ &   $\mathbf{1}\boxtimes X_1^\text{rev}$  \text{ or } $Z\boxtimes X_2^\text{rev}$	\\
			$\mathbf{1}\boxtimes X_1^\text{rev}$  	   & $\implies$ &   $ X_1\boxtimes \mathbf{1}^\text{rev}$ \text{ or } $ X_2\boxtimes Z^\text{rev}$	\\	
			$Z\boxtimes X_1^\text{rev}$  	   		   & $\implies$ &   $ X_2\boxtimes \mathbf{1}^\text{rev}$ \text{ or } $ X_1\boxtimes Z^\text{rev}$	\\		
			$\mathbf{1} \boxtimes X_2^\text{rev}$  	   & $\implies$ &   $ X_2\boxtimes \mathbf{1}^\text{rev}$ \text{ or } $ X_1\boxtimes Z^\text{rev}$	\\							
			$Z\boxtimes X_2^\text{rev}$  	   		   & $\implies$ &   $ X_1\boxtimes \mathbf{1}^\text{rev}$ \text{ or } $ X_2\boxtimes Z^\text{rev}$	\\          	                 
    	\bottomrule
    
	    \end{tabular}
	    	\caption{\label{tab:posmaps} Possibilities for $\mathcal{F}( \mathbf{1}\boxtimes X_1^\text{rev} )$.}

\end{table}

\hfill\begin{minipage}{\dimexpr\textwidth-1cm}

\textbf{Case: $r \equiv 0 \pmod 4$ and $\Z{2r+1}^\times$ contains an element $n$ such that $n^2 = -1$}

By considering dimensions and twists there are four possible objects in the image of $X_1\boxtimes \mathbf{1}^\text{rev}$ under $\mathcal{F}$. They are
\[ X_1\boxtimes \mathbf{1}^\text{rev}, X_1\boxtimes Z^\text{rev},\mathbf{1}\boxtimes X_1^\text{rev}, Z\boxtimes X_1^\text{rev}.
\]
\end{minipage}
\vspace{.5cm}

\hfill\begin{minipage}{\dimexpr\textwidth-1cm}
\textbf{Case: $r \equiv 2 \pmod 4$ and $\Z{2r+1}^\times$ contains an element $n$ such that $n^2 = -1$ }

By considering dimensions and twists there are four possible objects in the image of $X_1\boxtimes \mathbf{1}^\text{rev}$ under $\mathcal{F}$. They are
\[ X_1\boxtimes \mathbf{1}^\text{rev}, X_1\boxtimes Z^\text{rev},\mathbf{1}\boxtimes X_2^\text{rev}, Z\boxtimes X_2^\text{rev}.
\]
\end{minipage}
\vspace{.5cm}

\hfill\begin{minipage}{\dimexpr\textwidth-1cm}
\textbf{Case: $r \equiv 1 \pmod 2$ and $\Z{2r+1}^\times$ contains an element $n$ such that $n^2 = -1$ }

The first supplement to quadratic reciprocity shows that this is a vacuous case.
\end{minipage}
\vspace{.5cm}

\hfill\begin{minipage}{\dimexpr\textwidth-1cm}
\textbf{Case: $r \equiv 1 \pmod 2$ and $\Z{2r+1}^\times$ does not contain an element $n$ such that $n^2 = -1$ }

By considering dimensions and twists there are two possible objects in the image of $X_1\boxtimes \mathbf{1}^\text{rev}$ under $\mathcal{F}$. They are
\[ X_1\boxtimes \mathbf{1}^\text{rev} \text{ and } X_1\boxtimes Z^\text{rev}.
\]
\end{minipage}
\vspace{.5cm}

\hfill\begin{minipage}{\dimexpr\textwidth-1cm}
\textbf{Case: $r \equiv 0 \pmod 4$ and $\Z{2r+1}^\times$ does not contain an element $n$ such that $n^2 = -1$ }

By considering dimensions and twists there are four possible objects in the image of $X_1\boxtimes \mathbf{1}^\text{rev}$ under $\mathcal{F}$. They are
\[ X_1\boxtimes \mathbf{1}^\text{rev}, X_1\boxtimes Z^\text{rev},\mathbf{1}\boxtimes X_1^\text{rev}, Z\boxtimes X_1^\text{rev}.
\]
Aiming towards a contradiction, suppose that $\mathcal{F}(X_1 \boxtimes \mathbf{1}^\text{rev}) = \mathbf{1}\boxtimes X_1^\text{rev}$. Then $\mathcal{F}$ induces a braided equivalence $\cat{so}{2r+1}{2} \to \cat{so}{2r+1}{2}^\text{rev}$. However this is a contradiction by Lemma~\ref{lem:bopeq}, thus 
\[\mathcal{F}(X_1 \boxtimes \mathbf{1}^\text{rev}) \ncong \mathbf{1}\boxtimes X_1^\text{rev}.\]

Aiming towards another contradiction, suppose that there exist braided fusion ring automorphisms $\mathcal{F}_1 , \mathcal{F}_2$ such that $\mathcal{F}_1(X_1\boxtimes \mathbf{1}^\text{rev}) = X_1\boxtimes Z^\text{rev}$ and $\mathcal{F}_2(X_1\boxtimes \mathbf{1}^\text{rev}) = Z\boxtimes X_1^\text{rev}$. From Table~\ref{tab:posmaps} we see that $\mathcal{F}_1(\mathbf{1} \boxtimes X_1^\text{rev}) = Z \boxtimes X_1^\text{rev}$. Therefore  $ \mathcal{F}_1^{-1}\circ \mathcal{F}_2 $ is a braided fusion ring automorphism sending $X_1\boxtimes \mathbf{1}^\text{rev}$ to $\mathbf{1}\boxtimes X_1^\text{rev}$, which induces the same contradiction as before. 

Together, we see there are at most two objects in the image of $X_1\boxtimes \mathbf{1}^\text{rev}$ under $\mathcal{F}$ a braided fusion ring automorphism of $\cat{so}{2r+1}{2}$.
\end{minipage}
\vspace{.5cm}

\hfill\begin{minipage}{\dimexpr\textwidth-1cm}
\textbf{Case: $r \equiv 2 \pmod 4$ and $\Z{2r+1}^\times$ does not contain an element $n$ such that $n^2 = -1$ }

By considering dimensions and twists there are four possible objects that a braided fusion ring automorphism can send $X_1\boxtimes \mathbf{1}^\text{rev}$ to. They are
\[ X_1\boxtimes \mathbf{1}^\text{rev}, X_1\boxtimes Z^\text{rev},\mathbf{1}\boxtimes X_2^\text{rev}, Z\boxtimes X_2^\text{rev}.
\]
The same arguments as in the previous case show that at most two of these objects can be the image of $X_1\boxtimes \mathbf{1}^\text{rev}$ under $\mathcal{F}$.
\end{minipage}
\vspace{.5cm}

This case by case analysis gives us the following bound 
\[
 |\BrAut(\mathcal{Z}( \cat{so}{2r+1}{2}))|  \leq
  \begin{cases}
                                   4|\BrAut( \cat{so}{2r+1}{2}) |^2& \text{if $\Z{2r+1}^\times$ contains an element $n$ such that $n^2 = -1$} \\
 								   2|\BrAut( \cat{so}{2r+1}{2})|^2 & \text{otherwise.}
  \end{cases}.
\]
Combining this with the equality 
\[   | \BrAut(\mathcal{Z}(\cat{so}{2r+1}{2})) | = |  \TenAut( \cat{so}{2r+1}{2} )| \cdot      |  \BrAut(\cat{so}{2r+1}{2})|, \]
gives the statement of the Lemma.
\end{proof}

We move on to showing that the bounds on $\TenAut( \cat{so}{2r+1}{2}) $ from Lemma~\ref{lem:tenbound}, and on $\BrAut( \cat{so}{2r+1}{2}) $ from Lemma~\ref{lem:bupper} are sharp. Our main tool to construct monoidal and braided auto-equivalences of the category $\cat{so}{2r+1}{2}$ is to identify the de-equivariantization of $\cat{so}{2r+1}{2}$ by the $\operatorname{Rep}(\Z{2})$ category $\langle \mathbf{1},Z \rangle$. It turns out that computing the auto-equivalences of this de-equivariantization is much easier. We then use the functorality of equivariantization to obtain auto-equivalences of the category $\cat{so}{2r+1}{2}$.

There are two structures in play here, depending on if we think of $\cat{so}{2r+1}{2}$ as a braided category, or just a plain monoidal category. We begin by considering the braided case, which is slightly easier.

It is clear from the fusion rules and T-matrix of $\cat{so}{2r+1}{2}$ that there is unique fully faithful braided functor $\Rep(\Z{2}) \to \cat{so}{2r+1}{2}$, defined by 
\[   \text{triv} \mapsto \mathbf{1} \text{ and }  \text{sgn} \mapsto Z.\] 
By de-equivariantizating $\cat{so}{2r+1}{2}$ we get a triple of data $(\mathcal{D} , \rho , \sigma)$, where $\mathcal{D}$ is a $\Z{2}$-graded fusion category, $\rho$ a monoidal functor $\rho:\underline{\Z{2}} \to \underline{\TenAut}(\mathcal{D})$, and $\sigma$ is a family of $\Z{2}$-crossed braidings 
\[    X \otimes Y_1 \to \rho(1)[  Y_1 ] \otimes X, \]
such that the category $\mathcal{D}^{\Z{2}}$ is braided equivalent to $\cat{so}{2r+1}{2}$. Our goal is to determine as much information about this triple $(\mathcal{D} , \rho , \sigma)$ as possible.

From the results of \cite{MR3541678} we know that the category $\mathcal{D}$ is a Tambara-Yamagami category $\mathcal{TY}(\Z{2r+1},\chi,\tau)$, where $\chi$ is a symmetric bicharacter $\Z{2r+1} \times \Z{2r+1} \to \mathbb{C}^\times$, and $\tau \in \pm$. Furthermore they show that at the level of objects $\rho(1)$ fixes $m$, and sends $i \leftrightarrow -i$. Using up a degree of freedom, we can thus arrange $\Z{2r+1}$ so that the forgetful functor $\cat{so}{2r+1}{2} \to \cat{so}{2r+1}{2} // \Rep(\Z{2}) \simeq  \mathcal{TY}(\Z{2r+1},\chi,\tau)$ sends
\[   Y_i \mapsto i \oplus -i.\]

The $\Z{2}$-crossed braided category $\mathcal{TY}(\Z{2r+1},\chi,\tau)$ has a braided subcategory with $\Z{2r+1}$ fusion rules. By \cite{MR1734419} such categories are completely classified by $q$ a $2r+1$-th root of unity. The associator on $(\Z{2r+1},q )$ is trivial, and the braiding of two objects $i$ and $j$ is given by $q^{ij}$. To determine $q$ we note that the adjoint subcategory of $\cat{so}{2r+1}{2}$ is braided equivalent to $(\Z{2r+1},q )^{\Z{2}}$. Forgetting gives a braided functor that sends
\[   Y_i \mapsto  i \oplus -i. \]
In particular this implies that the twist of $1 \in (\Z{2r+1},q )$ is equal to the twist of $Y_1$. Thus we have that $q = e^{2\mathbf{i}\pi  \frac{r}{2r+1}}$.

As $\mathcal{TY}(\Z{2r+1}, \chi, \tau)$ is a $\Z{2}$-crossed braided extension of $(\Z{2r+1},q )$, we have that the single non-invertible simple object of $\mathcal{TY}(\Z{2r+1}, \chi, \tau)$ forms a module category over the braided category $(\Z{2r+1},q )$. The unique rank one module over $(\Z{2r+1},q )$, is $\operatorname{Vec}$ with trivial structure morphisms. Using the theory of $G$-crossed braided extensions \cite{MR2677836}, we can see that the associator morphism
\[   ( i \otimes m ) \otimes j \to   i \otimes( m  \otimes j) \]
in $\mathcal{TY}(\Z{2r+1}, \chi, \tau)$ is given by the braiding of $i$ and $j$ in $(\Z{2r+1},q )$. This completely determines the bicharacter $\chi$, giving us
\[ \chi(i,j) = q^{ij} = e^{2\mathbf{i}\pi  \frac{ijr}{2r+1}}. \]

We neglect to determine $\tau$.

\begin{lemma}\label{lem:tyaut}
Let $n \in \{ n \in \Z{2r+1}^\times : n^2 = 1\}$, then there exists a monoidal auto-equivalence $\mathcal{F}_n$ of $\mathcal{TY}(\Z{2r+1}, \chi, \tau)$ defined by
\[   \mathcal{F}_n(m) = m  \text{ and } \mathcal{F}_n(i) = ni \pmod{2r+1}. \]
The structure constants for $\mathcal{F}_n(m)$ are trivial. The map $n \mapsto \mathcal{F}_n(m)$ gives an isomorphism
\[ \{ n \in \Z{2r+1}^\times : n^2 = 1\} \to \TenAut(\mathcal{TY}(\Z{2r+1},\chi,\tau)). \]
\end{lemma}
\begin{proof}
This is a direct application of Lemma~\ref{lem:TYautos} and Remark~\ref{rmk:tyautos}. It is routine to check that $n \in \Z{2r+1}^\times$ preserves $\chi$ if and only if $n^2 = 1$.
\end{proof}

Now we determine $\rho: \underline{\Z{2}} \to\underline{\TenAut}(\mathcal{TY}(\Z{2r+1},\chi,\tau)  )$. Trivially we have that $\rho(0)$ is the identity on $\mathcal{TY}(\Z{2r+1},\chi,\tau)$. As mentioned before, $\rho(1)$ restricts to the $(\Z{2r+1},q)$ to give the auto-equivalence $i \mapsto -i$. From Lemma~\ref{lem:tyaut} there is a unique monoidal auto-equivalence of $\mathcal{TY}(\Z{2r+1}, \chi, \tau)$ that restricts to $i \mapsto -i$ on the $\Z{2r+1}$ subcategory. Thus $\rho(1) \cong \mathcal{F}_{2r}$. We neglect to determine the tensor structure of $\rho$.

Finally we determine the $\Z{2}$-crossed braiding $\sigma$ on $\mathcal{TY}(\Z{2r+1}, \chi, \tau)$. Solving \cite[Equations (82) and (83)]{MR2609644} using a slight alteration of the techniques used in \cite{0011037} gives that there are four possibilities for $\sigma$, depending on a choice of square root of both $\chi(1,1)$ and $\tau$.
\begin{align*}
\sigma_{i,j} &:= \chi(i,j)\operatorname{id}_{i+j} \\
\sigma_{i,m} &:=  \chi(1,1)^{\frac{i^2}{2}}\operatorname{id}_m \\
(\sigma_{m,m})_i &:= \sigma_{i,m}\chi(i,i) \sqrt{ \frac{\tau}{\sqrt{2r+1}}\sum_{j = 0}^{2r} \sigma_{j,m}} \id_i.
\end{align*}
Solving for equation  \cite[Equation (81)]{MR2609644} shows that $\chi(1,1)^{\frac{1}{2}} = \chi(1,1)^{2r^2}$, uniquely determining the choice of square root of $\chi(1,1)$. This allows us to write
\[ \sigma_{i,m} =  \chi(1,1)^{2i^2r^2}\operatorname{id}_m.\]
Without $\tau$ we are unable to determine its choice of square root. Thus we have two possible $\Z{2}$ crossed braidings. Thankfully the following key Lemma is independent from both the choice of $\tau$, and the choice of square root of $\tau$.  

\begin{lemma}\label{lem:prev}
Let $n \in \{ n \in \Z{2r+1}^\times : n^2 = 1\}$, then there exist monoidal natural isomorphisms
\[ \eta^0 : \rho(0) \circ \mathcal{F}_{n}  \to  \mathcal{F}_{n} \circ \rho(0)\]
and 
\[ \eta^1 : \rho(1) \circ \mathcal{F}_{n}  \to  \mathcal{F}_{n} \circ \rho(1)\]
defined by $\eta^0_X = \id_{ \mathcal{F}_{n}(X) }$ and $\eta^1_X = \id_{ \mathcal{F}_{2rn}(X) }$.

Furthermore, the pair $(\mathcal{F}_n, \eta )$ is a automorphism of $( \mathcal{TY}(\Z{2r+1},\chi,\pm) , \rho, \sigma)$ in the category of $\Z{2}$-crossed braided fusion categories.
\end{lemma}
\begin{proof}
As the tensor structure morphisms for the monoidal functor $\mathcal{F}_{n}$ are trivial, the natural isomorphisms $\eta^0$ and $\eta^1$ are easily checked to be monoidal.

To verify that $(\mathcal{F}_n, \eta )$ is a automorphism of $( \mathcal{TY}(\Z{2r+1},\chi,\pm) , \rho, \sigma)$ in the category of $\Z{2}$-crossed braided fusion categories, we have to check the commutativity of the diagram~\eqref{eq:com}. Working through all possibilities for all objects in $ \mathcal{TY}(\Z{2r+1},\chi,\pm)$ shows that commutativity is equivalent to the following equations
\begin{align*}
  \mathcal{F}_n( \sigma_{i,j} )  &= \sigma_{ni,nj} \text{ for all } i,j \in \Z{2r+1}, \\
  \mathcal{F}_n(\sigma_{i,m} ) &=  \sigma_{ni,m} \text{ for all } i \in \Z{2r+1}, \\
  \mathcal{F}_n(\sigma_{m,i} ) &=  \sigma_{m,ni} \text{ for all } i \in \Z{2r+1}, \\ 
  \mathcal{F}_n(\sigma_{m,m} ) &=  \sigma_{m,m} .
\end{align*}
Expanding these gives the equations:
\begin{align*}
  \chi(i,j)  &= \chi(ni,nj) , \\
  \chi(1,1)^{2i^2r^2} &=  \chi(1,1)^{2n^2i^2r^2}, \\
  \chi(1,1)^{2i^2r^2}\chi(i,i) \sqrt{ \frac{\tau}{\sqrt{2r+1}}\sum_{j = 0}^{2r} \chi(1,1)^{2j^2r^2}} &=  \chi(1,1)^{2n^2i^2r^2} \chi(ni,ni) \sqrt{ \frac{\tau}{\sqrt{2r+1}}\sum_{j = 0}^{2r} \chi(1,1)^{2j^2r^2}} , \\
    \text{ for all } i,j &\in \Z{2r+1}.
\end{align*}
Each of these follows from the fact that $\chi(1,1)$ is a $2r+1$-th root of unity, and 
\[n^2 \equiv 1 \pmod {2r+1}.\]
\end{proof}

With the above Lemma in hand we can now apply the functorality of equivariantization to get a homomorphism
\[   n \mapsto \mathcal{F}_n^{\Z{2}} : \{ n \in \Z{2r+1}^\times : n^2 = 1\} \to \BrAut(\cat{so}{2r+1}{2}) \]
We explicitly compute that $\mathcal{F}_n^{\Z{2}}$ fixes the objects $\mathbf{1},Z,X_1,X_2$, and sends the object $Y_i \mapsto  Y_{\min( ni \pmod {2r+1}) ,  -ni \pmod {2r+1})}$. This homomorphism is surjective (although not \textit{a-priori}), but not injective.

\begin{lemma}
The kernel of the homomorphism $n \mapsto \mathcal{F}_n^{\Z{2}}$ is $\{1 , -1 \}$.
\end{lemma}
\begin{proof}
Let $n \in \Z{2r+1}^\times$ such that $n^2 = 1$ such that $\mathcal{F}_n^{\Z{2}} \cong \Id_{\cat{so}{2r+1}{2}}$. Then $Y_1 = \mathcal{F}_n^{\Z{2}}(Y_1) = Y_{\min( n \pmod {2r+1}) ,  -n \pmod {2r+1})}$, giving us that $n = \pm 1$.

On the other hand the auto-equivalence $\mathcal{F}_{2r}^{\Z{2}}$ is a gauge auto-equivalence of $\cat{so}{2r+1}{2}$, as it fixes all the simple objects. Thus by Corollary~\ref{cor:gaugeB} $\mathcal{F}_{2r}^{\Z{2}}$ is trivial.
\end{proof}
\begin{cor}
There is an injection 
\[   \{ n \in \Z{2r+1}^\times : n^2 = 1\} / (\pm) \to \BrAut(\cat{so}{2r+1}{2}).\]
\end{cor}

This injection realises the upper bound on $\BrAut(\cat{so}{2r+1}{2})$ from Lemma~\ref{lem:bupper}, hence we have
\[   \BrAut(\cat{so}{2r+1}{2})   \cong  \{ n \in \Z{2r+1}^\times : n^2 = 1\} / (\pm). \] 
Let $\omega(2r+1)$ be the number of distinct prime divisors  of $2r+1$. We have that 
\[  \Z{2}^{\omega(2r+1)} \cong \{ n \in \Z{2r+1}^\times : n^2 = 1\},\]
with an explicit isomorphism sending 
\[   \{p_1^\pm , p_2^\pm\cdots , p_k^\pm \} \mapsto   \sum_{i = 1}^k \pm \frac{2r+1}{p_i} \pmod{2r+1}. \]
This allows us to write 
\[   \BrAut(\cat{so}{2r+1}{2})   \cong  \Z{2}^{\omega(2r+1) - 1} \]
as in the statement of Theorem~\ref{thm:main}.

\vspace{1cm}

 We now consider $\cat{so}{2r+1}{2}$ as just a monoidal category, forgetting the braided structure. The monoidal category $\cat{so}{2r+1}{2}$ has two structures as a category over $\operatorname{Rep}(\Z{2})$, given by the two central functors $\langle Z \rangle \to \mathcal{Z}( \cat{so}{2r+1}{2}   )$ defined by:
 \[ F^+: Z \mapsto Z\boxtimes \mathbf{1}^\text{rev}\quad \text{ and } \quad   F^-: Z \mapsto \mathbf{1} \boxtimes Z^\text{rev}.\]
 
Thus de-equivariantizing $\cat{so}{2r+1}{2}$ gives us two fusion categories with $\Z{2}$ action, corresponding to the two lifts $F^\plus$ and $F^\minus$. We call these two fusion categories with $\Z{2}$ action 
\[   ( \mathcal{D}_\plus , \rho^\plus) \qquad \text{ and } \qquad( \mathcal{D}_\minus , \rho^\minus)\]

 respectively, where $\mathcal{D}_\pm$ are fusion categories, and $\rho^\pm : \underline{\Z{2}} \to \underline{\TenAut}(\mathcal{D}_\pm)$ are monoidal functors, such that $\mathcal{D}_\pm^{\Z{2}} =\cat{so}{2r+1}{2}$.

It follows from the braided case that $\mathcal{D}_\plus = \mathcal{TY}(\Z{2r+1},\chi,\tau)$, with 
\[  \chi(i,j) =  e^{2\mathbf{i}\pi  \frac{ijr}{2r+1}}.\]
As $\mathcal{D}_\pm^{\Z{2}} = \cat{so}{2r+1}{2}$ we can count the simple objects of both categories to see that $\rho(1)$ can only fix the simple objects $0$ and $m$ in $\mathcal{TY}(\Z{2r+1},\chi,\tau)$. From Lemma~\ref{lem:tyaut} there is a unique monoidal auto-equivalence of $\mathcal{TY}(\Z{2r+1},\chi,\tau)$ with this property. Hence we have that $\rho^+(1) =\mathcal{F}_{2r}$.

As $F^\minus$ factors through $\cat{so}{2r+1}{2}^\text{rev}$, we have that $\mathcal{D}^\minus = (\mathcal{D}^\plus)^\text{op} = \mathcal{TY}(\Z{2r+1},\chi^{-1},\tau)$. The same argument as in the $\rho^\plus$ case shows that $\rho^\minus = \mathcal{F}_{2r}$.

\begin{lemma}\label{lem:Tyeq}
Let $n \in \Z{2r+1}^\times$ such that $n^2 = -1$, then there exists a monoidal equivalence $\mathcal{F}_n: \mathcal{TY}(\Z{2r+1}, \chi, \pm) \to \mathcal{TY}(\Z{2r+1}, \chi^{-1}, \pm)$ that fixes $m$, and sends $i \mapsto ni$. The structure constants for these monoidal equivalences are trivial. Furthermore there exist monoidal natural isomorphisms
\[ \eta^0 : \rho^\minus(0) \circ \mathcal{F}_{n}  \to  \mathcal{F}_{n} \circ \rho^\plus(0)\]
and 
\[ \eta^1 : \rho^\minus(1) \circ \mathcal{F}_{n}  \to  \mathcal{F}_{n} \circ \rho^\plus(1)\]
defined by $\eta^0_X = \id_{ \mathcal{F}_{n}(X) }$ and $\eta^1_X = \id_{ \mathcal{F}_{2rn}(X) }$.

Further, the pair $(\mathcal{F}_n, \eta)$ gives a morphism
\[ ( \cat{so}{2r+1}{2},F^\plus) \to  ( \cat{so}{2r+1}{2},F^\minus         )\]
in the category of fusion categories with $\Z{2}$-action.
\end{lemma}
\begin{proof}
This proof is a slight altercation of the proofs of Lemmas~\ref{lem:tyaut} and \ref{lem:prev}.
\end{proof}
  
Using the functorality of equivariantization, we get for each $n \in \Z{2r+1}^\times$ such that $n^2 = -1$, a morphism in the category of fusion categories over $\Rep(\Z{2})$: 
\[  \mathcal{F}_n^{\Z{2}}:  ( \cat{so}{2r+1}{2},F^\plus) \to  ( \cat{so}{2r+1}{2},F^\minus ).\]
Forgetting the central structures gives monoidal auto-equivalences of $\cat{so}{2r+1}{2}$ for each $n \in \Z{2r+1}^\times$ such that $n^2 = 1$. Combining these new monoidal auto-equivalences, with the braided auto-equivalences we have already constructed, we obtain a homomorphism
\[ n \mapsto \mathcal{F}_n^{\Z{2}} : \{ n \in \Z{2r+1}^\times : n^2 = \pm 1\} \to \TenAut(\cat{so}{2r+1}{2}).\]
Explicitly we compute that $\mathcal{F}_n^{\Z{2}}$ fixes the objects $\mathbf{1}, Z, X_1,X_2$, and sends $Y_i \mapsto  Y_{\min( ni \pmod {2r+1}) ,  -ni \pmod {2r+1})}$.

The same as in the braided case, we can show that the kernel of the homomorphism $n \mapsto \mathcal{F}_n^{\Z{2}}$ is $\{1,-1\}$, and thus we get an injection
\[  \{ n \in \Z{2r+1}^\times : n^2 = \pm 1\} / \{\pm\} \to  \TenAut(\cat{so}{2r+1}{2}).\]
We can use the boson $Z$ to construct a simple current monoidal auto-equivalence of $\cat{so}{2r+1}{2}$ sending $X_1 \leftrightarrow X_2$, and fixing all other simple objects. This gives us an injection 
\[  \Z{2} \times  \{ n \in \Z{2r+1}^\times : n^2 = \pm 1\} / \{\pm\}\to  \TenAut(\cat{so}{2r+1}{2}),\]
realising the upper bound on $\TenAut(\cat{so}{2r+1}{2})$ from Lemma~\ref{lem:tenbound}.
 
Via the first supplement to quadratic reciprocity, we have that $\Z{2r+1}$ has a solution to $n^2 = -1$ if and only if every prime $p$ dividing $2r+1$ satisfies $p\equiv 1 \pmod 4$. This allows us to write
\[\TenAut( \cat{so}{2r+1}{2})  =
  \begin{cases}
                                   \Z{2}\times \Z{2}\times \BrAut( \cat{so}{2r+1}{2}) & \text{if every prime $p$ dividing $2r+1$ satisfies $p\equiv 1 \pmod 4$} \\
 								   \Z{2} \times \BrAut( \cat{so}{2r+1}{2}) & \text{otherwise.}
  \end{cases}
\]
As in the statement of Theorem~\ref{thm:main}.

\section{Auto-equivalences for Lie type $C$}

In this section we compute the monoidal, and braided auto-equivalences of the categories $\cat{sp}{2r}{k}$. The outline of our computations is as follows. 

We begin by showing that there are no non-trivial gauge auto-equivalences of the categories $\cat{sp}{2r}{k}$. Recall that the object $\Lambda_1$ $\otimes$-generates $\cat{sp}{2r}{k}$. Thus any non-trivial gauge auto-equivalence of $\cat{sp}{2r}{k}$ will appear as a non-trivial planar algebra automorphism of the planar algebra generated by $\Lambda_1$. We know that this planar algebra is isomorphic to 
\[ \overline{\text{BMW}( q , -q^{2r+1} )} \text{ with } q = e^\frac{2 \pi i}{4(r + k + 1)}.\]
We can reuse a computation from the Lie type $B$ section to show that this planar algebra has no non-trivial automorphisms that give rise to gauge auto-equivalences of $\cat{sp}{2r}{k}$. Hence the category $\cat{sp}{2r}{k}$ has no non-trivial gauge auto-equivalences.

The results of \cite{MR1887583} compute the fusion ring automorphisms of $\cat{sp}{2r}{k}$ as:
\begin{align*}
\{e\} & \text{ if $r = 2$ and $k= 1$ }, \\ 
\{e\} & \text{ if $r \neq k$ and $rk$ is odd }, \\ 
\Z{2} & \text{ if $r \neq k$ and $rk$ is even }, \\ 
\Z{2} & \text{ if $r = k$ and $r$ is odd }, \\ 
\Z{2}\times \Z{2} & \text{ if $r = k$ and $r$ is even }.
\end{align*}

We find that all of these fusion ring automorphisms are realised as monoidal auto-equivalences, though only some of them are braided. The first four cases above are easily dealt with by constructing the auto-equivalences as simple current auto-equivalences. The latter two cases are more interesting, as there exists an exotic fusion ring automorphism. We show this exotic automorphism is realised by lifting the interesting planar algebra automorphism of $ \overline{\text{BMW}( q , i )}$ seen in Lemma~\ref{lem:paBMW}.

Let us being our computations.

\begin{lemma}
The category $\cat{sp}{2r}{k}$ has no non-trivial gauge auto-equivalences
\end{lemma}
\begin{proof}
Consider the object $\Lambda_1 \in \cat{sp}{2r}{k}$. Recall from Equation~\ref{eq:sp} that the planar algebra generated by $\Lambda_1$ is isomorphic to 
\[ \overline{\text{BMW}( q , -q^{2r+1} )} \text{ with } q = e^\frac{2 \pi i}{4(r + k + 1)}.\]
As the object $\Lambda_1$ $\otimes$-generates $\cat{sp}{2r}{k}$, we can apply Lemma~\ref{lem:gaugeauto} to see that every non-trivial gauge auto-equivalence of $\cat{sp}{2r}{k}$ appears as a non-trivial planar algebra automorphism of $\overline{\text{BMW}( q , -q^{2r+1} )}$. From Lemma~\ref{lem:paBMW}, the planar algebra $\overline{\text{BMW}( q , -q^{2r+1} )}$ only has a single non-trivial automorphism when $k = r$, which corresponds to an auto-equivalence of $\cat{sp}{2r}{k}$ that exchanges distinct objects. Thus the category $\cat{sp}{2r}{k}$ has no non-trivial gauge auto-equivalences.
\end{proof}

Next up we have to show that all fusion ring automorphisms of $\cat{sp}{2r}{k}$ are realisable. For this we break up into cases.

\subsection*{Case: $r=2$ and $k=1$ or $r \neq k$ and $rk$ is odd}\hspace{1em}

In this case we have no non-trivial fusion ring automorphisms of $\cat{sp}{2r}{k}$, thus
\[ \BrAut(\cat{sp}{2r}{k} ) =  \TenAut(\cat{sp}{2r}{k} )  =   \text{Gauge}(\cat{sp}{2r}{k}) = \{e\},\]
as in the statement of Theorem~\ref{thm:main}.

\subsection*{Case: $r \neq k$ and $rk$ is even}\hspace{1em}

In this case the upper bound of the auto-equivalence group is $\Z{2}$. The modular category $\cat{sp}{2r}{k}$ has an order 2 invertible object $k\Lambda_r$ which we can use to construct a simple current auto-equivalence. This invertible object has self-braiding eigenvalue $1$ when $rk \equiv 0 \pmod 4$, and $-1$ when $rk \equiv 2 \pmod 4$. Thus we can use Lemma~\ref{lem:simplecurrent} to construct a (possibly trivial) monoidal auto-equivalence of $\cat{sp}{2r}{k}$. This auto-equivalence sends 
\[\Lambda_1\mapsto \Lambda_1 \otimes (k\Lambda_r) \cong (\Lambda_{r-1} + (k-1)\Lambda_r),\]
 hence this auto-equivalence is non-trivial apart from the case $r=2$ and $k=1$, which we have already dealt with. This auto-equivalence is braided if and only if $(k\Lambda_r)$ has self-braiding eigenvalue $-1$, i.e if and only if $rk \equiv 2 \pmod 4$. Thus we have
\[   \TenAut( \cat{sp}{2r}{k}  ) = \Z{2} \quad \text{ and } \quad  \BrAut( \cat{sp}{2r}{k}  )  = \begin{cases}
\{e\}, \text{ if $rk \equiv 0 \pmod 4$,}\\
\Z{2}, \text{   if $rk \equiv 2 \pmod 4$.}\end{cases}\]
as in the statement of Theorem~\ref{thm:main}.

\subsection*{Case: $r = k$}

When $r = k$, the category $\cat{sp}{2r}{r}$ has an exotic order 2 fusion ring automorphism defined as follows. For an object $\sum_{i= 1}^r \lambda_i \Lambda_i$ consider the Young diagram whose $j$-th row consists of $r - \sum_{i = k - j}^k \lambda_i$ boxes. The transpose of this Young diagram corresponds to another simple object of $\cat{sp}{2r}{r}$, under the same combinatorial interpretation. The image of the object $\sum_{i= 1}^r \lambda_i \Lambda_i$ is exactly the simple object corresponding to this transposed Young diagram. As a quick example, we work out where $\Lambda_2$ gets sent to under this exotic automorphism.

\[   (\Lambda_2) \sim  \raisebox{-.5\height}{ \includegraphics[scale = .5]{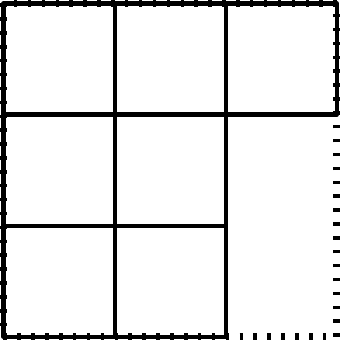}}\quad  \overset{Tr}{\longrightarrow}  \quad \raisebox{-.5\height}{ \includegraphics[scale = .5]{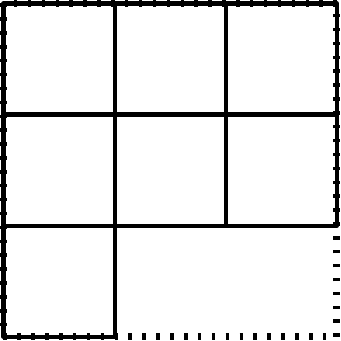}} \sim(2\Lambda_1).\]

This exotic automorphism of the $\cat{sp}{2r}{r}$ fusion ring is completely characterised by the fact that it fixes the object $\Lambda_1$, while not being gauge. The fact this automorphism fixes the object $\Lambda_1$ is particularly convenient for us, as if the automorphism is realisable, then it will appear as a non-trivial planar algebra automorphism of the planar algebra generated by $\Lambda_1$. From Equation~\ref{eq:sp} we know that the planar algebra generated by $\Lambda_1$ is isomorphic to $\overline{\text{BMW}( e^\frac{2 \pi i}{4(2r + 1)} , -i )}$.
Further, from Lemma~\ref{lem:paBMW}, there exists a planar algebra automorphism of $\overline{\text{BMW}( e^\frac{2 \pi i}{4(2r + 1)} , -i )}$ that corresponds to a non-braided monoidal auto-equivalence of $\cat{sp}{2r}{r}$ that exchanges distinct simple objects. By the unique characterisation of the above exotic fusion ring automorphism, we must have that this non-trivial planar algebra automorphism of $\overline{\text{BMW}( e^\frac{2 \pi i}{4(2r + 1)} , -i )}$ corresponds to a non-braided monoidal auto-equivalence of $\cat{sp}{2r}{r}$ that realises the exotic fusion ring automorphism. We denote this non-trivial auto-equivalence as the \textit{level-rank duality auto-equivalence} of $\cat{sp}{2r}{r}$.

When $r$ is odd, there is only a single non-trivial automorphism of the $\cat{sp}{2r}{r}$ fusion ring. The non-trivial fusion ring automorphism is realised by the level-rank duality auto-equivalence. Thus when $r$ is odd, we have
\[\TenAut( \cat{sp}{2r}{r}  ) = \Z{2} \quad \text{ and } \quad  \BrAut( \cat{sp}{2r}{k}  ) = \{e\} \]
as in the statement of Theorem~\ref{thm:main}.

When $r$ is even, there are three non-trivial automorphisms of the $\cat{sp}{2r}{r}$ fusion ring. The invertible object $k\Lambda_r \in \cat{sp}{2r}{r}$ is a has self-braiding eigenvalue $1$, as $r^2 \equiv 0\pmod 4$, so we can apply the same arguments as in the $rk \equiv 1 \pmod 2$ and $r\neq k$ case to construct an order 2, non-braided auto-equivalence that sends 
\[ \Lambda_1\mapsto  (\Lambda_{r-1} + (k-1)\Lambda_r).\] Furthermore the level-rank duality auto-equivalence gives us a non-braided order 2 auto-equivalence that fixes $\Lambda_1$. These two auto-equivalences are clearly not inverse to each other, thus their composition gives us a third non-trivial monoidal auto-equivalence. This composition is not braided, as it maps $\Lambda_1\mapsto  (\Lambda_{r-1} + (k-1)\Lambda_r)$, which have differing twists. Thus when $r$ is even, we have
\[\TenAut( \cat{sp}{2r}{r}  ) = \Z{2}\times \Z{2} \quad \text{ and } \quad  \BrAut( \cat{sp}{2r}{k}  ) = \{e\} \]
as in the statement of Theorem~\ref{thm:main}.
\section{Auto-equivalences for Lie type $G$}

In this section we compute the monoidal, and braided auto-equivalences of the categories $\cat{g}{2}{k}$. We remark that the results of this section are dependant on work that has yet to appear in print by Ostrik and Snyder. 

The outline of our computations is as follows. 

We begin with a standard planar algebra computation, showing that the $\overline{G_2(q)}$ planar algebra only has no non-trivial automorphisms, up to planar algebra natural isomorphism. Hence the category $\cat{g}{2}{k}$ has no non-trivial gauge auto-equivalences.

The results of \cite{MR1887583} compute the fusion ring automorphisms of $\cat{g}{2}{k}$ as 
\begin{align*}
\{e\} &\text{ if $k \neq \{3,4\}$},\\
\Z{3} &\text{ if $k  = 3$},\\
\Z{2} &\text{ if $k  = 4$}.
\end{align*}

When $k = 3$ or $k = 4$, we see the category $\cat{g}{2}{3}$ has exotic fusion ring automorphisms. For $k= 3$ we find that these exotic automorphisms do not lift to the category $\cat{g}{2}{k}$. Our argument here is to consider the planar algebra generated by the object $\Lambda_1$, and the planar algebra generated by the image of $\Lambda_1$ under the automorphism. We show that these two planar algebras are non-isomorphic, thus the automorphism can not lift to a monoidal auto-equivalence. For $k= 4$ we show the exotic fusion ring automorphism does lift to a braided auto-equivalence. Our argument is similar to the $k= 3$ case, where we consider the planar algebra generated by the object $\Lambda_1$, and the planar algebra generated by the image of $\Lambda_1$ under the automorphism. However in this case we show that planar algebras are isomorphic as braided planar algebras, and hence there is a braided auto-equivalence of $\cat{g}{2}{k}$ realising the fusion ring automorphism.

Let us begin our computations for this section.

\begin{lemma}\label{lem:G2}
The planar algebra $\overline{G_2(q)}$ has no non-trivial automorphisms, up to planar algebra natural isomorphism.
\end{lemma}
\begin{proof}
Let $\phi$ be a planar algebra automorphism of $\overline{G_2(q)}$, then $\phi$ is completely determined by how it acts on the generator \raisebox{-.5\height}{ \includegraphics[scale = .5]{trivalent}}. As the 3-box space of $\overline{G_2(q)}$ is 1-dimensional, we must have that
\[   \phi\left( \raisebox{-.5\height}{ \includegraphics[scale = .5]{trivalent}} \right ) \quad =\quad  \alpha \raisebox{-.5\height}{ \includegraphics[scale = .5]{trivalent}} \]
for some $\alpha \in \mathbb{C}$.

As $\phi$ must preserve the relations of $\overline{G_2(q)}$, we can see from relation (iii) of the $G_2(q)$ planar algebra that $\alpha = \pm 1$. It is routine to check that the automorphisms
\[   \phi\left( \raisebox{-.5\height}{ \includegraphics[scale = .5]{trivalent}} \right ) \quad =\quad  \pm \raisebox{-.5\height}{ \includegraphics[scale = .5]{trivalent}}\]
are consistent with the other relations, and the automorphisms preserve the negligible ideal by Proposition~\ref{prop:neg}. Thus we have two planar algebra automorphisms of $\overline{G_2(q)}$. However these two automorphisms are isomorphic to each other via the natural isomorphism
\[   \mu :=-\raisebox{-.5\height}{ \includegraphics[scale = .5]{G2iso}}. \]
\end{proof}

The object $\Lambda_1$ $\otimes$-generates $\cat{g}{2}{k}$, and we have the isomorphism of planar algebras
 \[ \text{PA}( \cat{g}{2}{k} , \Lambda_1 ) = \overline{G_2(q)}  \text{ with } q =  e^\frac{2 \pi i}{6 (4 + k)}, \]
 from Equation~\eqref{eq:g2}. Hence we can apply Lemma~\ref{lem:gaugeauto} to get the following Corollary.

\begin{cor}\label{lem:gaugeG2}
The categories $\cat{g}{2}{k}$ have no non-trivial gauge auto-equivalences.
\end{cor}

Now that we have shown that the categories $\cat{g}{2}{k}$ have no non-trivial gauge auto-equivalences, we move on to determining which of the fusion ring automorphisms of $\cat{g}{2}{k}$ are realisable. We break into cases.

\subsection*{Case: $k \neq \{3,4\}$}\hspace{1em}

In this case there are no non-trivial fusion ring automorphisms. Thus
\[ \BrAut(\cat{g}{2}{k} ) =  \TenAut(\cat{g}{2}{k} )  =   \text{Gauge}(\cat{g}{2}{k}) = \{e\},\]
as in the statement of Theorem~\ref{thm:main}.

\subsection*{Case: $k =3$} \hspace{1em}

When $k= 3$, there are two non-trivial fusion ring automorphisms, which are 
\[  \Lambda_1  \mapsto \Lambda_2 \mapsto 3\Lambda_1 \mapsto \Lambda_1,     \]
and its inverse.

\begin{lemma}\label{lem:G2L3}
The fusion ring automorphism $\Lambda_1  \mapsto \Lambda_2 \mapsto 3\Lambda_1 \mapsto \Lambda_1$ does not lift to a monoidal auto-equivalence of the category $\cat{g}{2}{3}$.
\end{lemma}
\begin{proof}
Using Frobenius-Schur indicators, computed from the modular data of $\cat{g}{2}{3}$, we can see that the object $\Lambda_2$ is symmetrically self-dual. Thus we can form the planar algebra $ \text{PA}( \cat{g}{2}{3},  \Lambda_2)$. We begin the proof by identifying this planar algebra.

As $\operatorname{dimHom}(\mathbf{1} \to \Lambda_2 ^{\otimes n})$ is equal to $1,0,1,1,4,10$ for $0 \leq n \leq 5$, we have from \cite[Theorem B]{MR3624901} that there is a planar algebra injection $\overline{G_2(q)} \to \text{PA}( \cat{g}{2}{3},  \Lambda_2)$ for $q$ some primitive root of unity. 

As the categorical dimension of  $\Lambda_2$ is equal to 
\[e^\frac{10 i\pi}{21} +e^\frac{8 i\pi}{21}+ e^\frac{2 i\pi}{21} + 1 + e^\frac{-2 i\pi}{21}+e^\frac{-8 i\pi}{21} e^\frac{-10 i\pi}{21} ,\]    
we must have that $q =  e^\frac{N i\pi}{21}$ for some $N \in \{1,4,5,16,17,20,22,25,26,37,38,41\}$. The quantum twist of $\Lambda_2$ is $e^{\frac{8 i \pi}{7}}$, which implies that $N$ has to be either $5$ or $26$. By Lemma~\ref{lem:g2iso} we know that the planar algebras $\overline{G_2(e^\frac{5 i\pi}{21})}$ and $\overline{G_2(e^\frac{26 i\pi}{21})}$ are isomorphic, thus we have a planar algebra injection
\[  \overline{G_2(e^\frac{5 i\pi}{21})} \to \text{PA}( \cat{g}{2}{3},  \Lambda_2). \]

Aiming towards a contradiction of the statement of the Lemma, suppose there was a monoidal auto-equivalence $\Gamma$ of $\cat{g}{2}{3}$ that sends $\Lambda_1  \mapsto \Lambda_2$. Then $\Gamma$ induces a monoidal equivalence of based categories $( \cat{g}{2}{3}, \Lambda_1) \to ( \cat{g}{2}{3}, \Lambda_2)$, and hence by \cite[Theorem A]{1607.06041} an isomorphism of planar algebras
\[   \text{PA}( \cat{g}{2}{3},  \Lambda_1) \to \text{PA}( \cat{g}{2}{3},  \Lambda_2)  \].

We know from Equation~\eqref{eq:g2} that $\text{PA}( \cat{g}{2}{3},  \Lambda_1) \cong \overline{G_2\left(e^\frac{i\pi}{21}\right)}$, so we have an injection of planar algebras  
\[     \overline{G_2(e^\frac{5 i\pi}{21})}\to \overline{G_2(e^\frac{i\pi}{21})}. \]
As the 3-box space of both these planar algebras is 1-dimensional, we must have that the trivalent vertex of $\overline{G_2(e^\frac{5 i\pi}{21})}$ is mapped to a scaler multiple of the trivalent vertex of $\overline{G_2(e^\frac{i\pi}{21})}$. However, the bubble and triangle popping values of the trivalent vertex in each of these planar algebras are different, and there exists no rescaling of the trivalent vertex to arrange that these two values are the same for both planar algebras. Hence we have a contradiction, and thus the fusion ring automorphism $\Lambda_1  \mapsto \Lambda_2 \mapsto 3\Lambda_2 \mapsto \Lambda_1$ does not lift to a monoidal auto-equivalence of $\cat{g}{2}{3}$.
\end{proof}

As the non-trivial fusion ring automorphism of $\cat{g}{2}{3}$ does not lift, we see
\[ \BrAut(\cat{g}{2}{3} ) =  \TenAut(\cat{g}{2}{3} )  =   \text{Gauge}(\cat{g}{2}{3}) = \{e\},\]
as in the statement of Theorem~\ref{thm:main}.

\subsection*{Case: $k=4$} \hspace{1em}

When $k = 4$ there is a single non-trivial fusion ring automorphism, which sends
\[   \Lambda_1 \leftrightarrow 2 \Lambda_2  \quad \text{and}\quad \Lambda_2 \leftrightarrow 4\Lambda_1,  \]
and fixes the other simple objects.
\begin{lemma}
This non-trivial fusion ring automorphism of $\cat{g}{2}{4}$ lifts to a braided auto-equivalence.
\end{lemma}  
\begin{proof}
Using the modular data of $\cat{g}{2}{4}$, we can use Frobenius-Schur indicators to verify that the object $2\Lambda_2$ is symmetrically self-dual. Thus we can form the planar algebra $\text{PA}(\cat{g}{2}{4}, 2\Lambda_2)$. Using the same argument as in the proof of Lemma~\ref{lem:G2L3} we can show that there is an injection of planar algebras
\[   \overline{G_2(q)} \to \text{PA}(\cat{g}{2}{4}, 2\Lambda_2) \]
for $q$ some primitive root of unity.

The categorical dimension of  $2\Lambda_2$ is equal to 
\[e^\frac{10 i\pi}{24} +e^\frac{8 i\pi}{24}+ e^\frac{2 i\pi}{24} + 1 + e^\frac{-2 i\pi}{24}+e^\frac{-8 i\pi}{24} e^\frac{-10 i\pi}{24},\] 
so we must have that $q =  e^\frac{N i\pi}{24}$ for some $N \in \{1,5,19,23,25,29,43,47\}$. The quantum twist of $2\Lambda_2$ is $i$, which allows us to further deduce that $N \in \{1,5,25,29\}$. The Hopf link of $4\Lambda_2$ is $3 \sqrt{2}+\sqrt{6 \left(2 \sqrt{2}+3\right)}+2$, which shows that $N \in \{1,25\}$. By Lemma~\ref{lem:g2iso} we know that the planar algebras $\overline{G_2\left(e^\frac{i\pi}{24}\right)}$ and $\overline{G_2\left(e^\frac{25 i\pi}{24}\right)}$ are isomorphic as braided planar algebras, hence there is an injection of braided planar algebras
\[  \overline{G_2\left(e^\frac{i\pi}{24}\right)} \to \text{PA}(\cat{g}{2}{4}, 2\Lambda_2). \]
  
As the map exchanging $ \Lambda_1 \leftrightarrow 2 \Lambda_2$ and $\Lambda_2 \leftrightarrow 4\Lambda_1$ is a fusion ring homomorphism, we have the dimensions of the spaces
\[\operatorname{Hom}(\mathbf{1} \to (4\Lambda_2)^{n}) \quad \text{ and } \quad \operatorname{Hom}(\mathbf{1} \to (\Lambda_1)^{n})\]
agree. Hence the planar algebras $\overline{G_2\left(e^\frac{i\pi}{24}\right)}$ and $\text{PA}(\cat{g}{2}{4}, 2\Lambda_2)$ have box spaces of the same dimensions, and so the injection 
\[  \overline{G_2\left(e^\frac{i\pi}{24}\right)} \to \text{PA}(\cat{g}{2}{4}, 2\Lambda_2)\]
is in fact an isomorphism of braided planar algebras.

By \cite[Theorem A]{1607.06041}, the above isomorphism of planar algebras lifts to a braided equivalence of based categories
\[  (\cat{g}{2}{4}, \Lambda_1) \to (\cat{g}{2}{4}, 2\Lambda_2).\]
Forgetting the basing gives a braided auto-equivalence of $\cat{g}{2}{4}$ that sends $\Lambda_1$ to $2\Lambda_2$. As there is a unique non-trivial fusion ring automorphism of $\cat{g}{2}{4}$, we must have that this automorphism lifts.
\end{proof}
As the non-trivial fusion ring automorphism of $\cat{g}{2}{4}$ lifts, we thus have
\[  \TenAut(\cat{g}{2}{4}) = \BrAut( \cat{g}{2}{4}) =\Z{2}\]
as in the statement of Theorem~\ref{thm:main}.

\section{Future Directions and Applications}
\subsection*{A new family of modular tensor categories from the charge-conjugation auto-equivalence of $\cat{sl}{r+1}{k}$}\hspace{1em}

For the modular tensor categories $\cat{sl}{r+1}{k}$ we show the existence (for $r\geq 2$) of the semi-exceptional charge conjugation braided auto-equivalence. This auto-equivalence has order $2$, hence we get a group action
\[ \Z{2} \to \BrAut(\cat{sl}{r+1}{k} ).\]

The technique of gauging \cite{MR3555361} allows us to take a group action on a modular tensor category, and construct a new modular tensor category, given that two certain cohomological obstructions vanish. For the example of $\cat{sl}{r+1}{k}$, the first of these cohomological obstructions will be an element of
\[  H^3(\Z{2} , \Inv(\cat{sl}{r+1}{k})) = H^3(\Z{2} , \Z{r+1}),\]
with $\Z{2}$ acting on $\Z{r+1}$ via the inverse map, and the second will be an element of
\[  H^4(\Z{2} , \mathbb{C}^\times) = \{e\}.\]

As latter obstruction group is trivial, the corresponding obstruction vanishes for free. The former obstruction group is isomorphic to \[\Z{r+1}/( 2 \Z{r+1}) = \begin{cases}  \{e\} \text{ if $r$ is even},\\		
						 		\Z{2} \text{ if $r$ is odd}.
						\end{cases}\]
 Thus for $r$ even, we can gauge	$\cat{sl}{r+1}{k}$ by the charge conjugation auto-equivalence to obtain a new family of modular tensor categories. While we have the abstract existence of these categories, almost all other information about them remains mysterious.
 
\begin{ques}	
For even $r$, determine the structure of the modular tensor categories obtained by gauging $\cat{sl}{r+1}{k}$ by the charge conjugation auto-equivalence.
\end{ques}

Even for the simplest non-trivial case of $\cat{sl}{3}{3}$ it is hard to say anything about the resulting modular tensor category. Hence, it appears new techniques will be required to answer this question.

In another direction, it would be interesting to study the obstruction living in $H^3(\Z{2} , \Z{r+1})$ when $r$ is odd.					 		

\subsection*{Reverse-braided auto-equivalences of $\cat{so}{2r+1}{k}$, and connections with modular grafting}\hspace{1em}

When $k = 2$ and every prime $p$ dividing $2r+1$ is congruent to $1$ mod $4$, we show the existence of a distinguished exceptional monoidal auto-equivalence of $\cat{so}{2r+1}{2}$. This monoidal auto-equivalence has the peculiar property that it is reverse braided in the sense that 
\[     \mathcal{F}\left(  \raisebox{-.5\height}{ \includegraphics[scale = .5]{BMWpositivecrossing}}\right )\quad = \quad \raisebox{-.5\height}{ \includegraphics[scale = .5]{BMWnegitivecrossing}}.\]
The existence of such auto-equivalences is rare for general modular tensor categories, as it implies restrictive conditions on the central charge of the category, and on the twists of objects which are fixed points.

The first few examples of categories $\cat{so}{2r+1}{2}$ with such reverse braided auto-equivalences are

\[ \cat{so}{5}{2},\quad   \cat{so}{5}{13} ,\quad   \cat{so}{17}{2} ,\quad  \cat{so}{25}{2} ,\text{ and } \quad  \cat{so}{29}{2}.\]

Playing numerology, we notice that that (apart from $2r+1 = 25$), these categories all appear as ingredients of the hypothetical ``modular grafting'' construction \cite{MR2837122}. It is conjectured that these categories can be grafted with certain dihedral groups to obtain the Drinfeld centres of the Haagerup-Izumi categories. The difficulty in this conjecture is that there is no idea of what grafting should look like at the level of modular tensor categories. The appearance of these reverse braided auto-equivalences for these examples leads one to suspect that such auto-equivalences may play important role in making the modular grafting construction rigorous.

While this is all wild speculation, we do have some evidence for a connection. We first notice that the objects from $\cat{so}{2r+1}{2}$ that appear in the modular grafting are exactly the non-fixed simple objects of the reverse braided auto-equivalence. By this we mean precisely, that if $\{X_i: 1\leq i \leq m \}$ are the non-fixed simple objects of $\cat{so}{2r+1}{2}$, then there exists an $m\times m$ block in the S-matrix of the Drinfeld centre of the hypothetically corresponding Haagerup-Izumi category that exactly corresponds to the $m \times m$ sub-matrix of the S-matrix for $\cat{so}{2r+1}{2}$ obtained by restricting to the objects $\{X_i: 1\leq i \leq m \}$. The same fact holds true for the T-matrix as well.

We also notice that the only other quantum group modular tensor category that could possibly have a reverse braided auto-equivalence is $\cat{f}{4}{4}$. This auto-equivalence has 12 non-fixed simple objects, and we find that the $12 \times 12$ block of restricted modular data for $\cat{f}{4}{4}$ exactly appears in the modular data for the centre of extended Haagerup \cite{MR3719545}. Further, we can identify $\Z{4}\ltimes \Z{5}$ as the group component of the modular grafting, and construct a fairly natural looking formula for the modular data of centre of extended Haagerup, starting from the modular data of $\cat{f}{4}{4}$ and the Drinfeld centre of $\Z{4}\rtimes \Z{5}$. This strongly suggests that the centre of extended Haagerup should fit into the hypothetical modular grafting procedure, and that reverse braided auto-equivalences will play a key role in making this procedure precise.

While of course this could all just be coincidence, we believe there is sufficient evidence to investigate the following question.

\begin{ques}
Can reverse braided auto-equivalences be used to give a rigorous definition of modular grafting.
\end{ques}

\subsection*{A new family of planar algebras from the level-rank duality auto-equivalence of $\cat{sp}{2r}{r}$}\hspace{1em}

Recall that $\cat{sp}{2r}{r}$ has exceptional order two level-rank duality auto-equivalence, which we constructed from an exceptional planar algebra automorphism $\phi$ of the planar algebra $\overline{{\operatorname{BMW}}(q, i)}$. This exceptional planar algebra automorphism was defined by
\[      \phi\left (   \raisebox{-.5\height}{ \includegraphics[scale = .5]{BMWpositivecrossing}}\right) \quad = \quad -  \raisebox{-.5\height}{ \includegraphics[scale = .5]{BMWnegitivecrossing}} .\]
By taking fixed points of the planar algebra $\overline{{\operatorname{BMW}}(q, i)}$ under this $\Z{2}$-action, we can construct a new planar algebra $\overline{{\operatorname{BMW}}(q, i)}^{\Z{2}}$. 

While this planar algebra is easily described as the fixed points of $\phi$, is always desirable to have a presentation of a planar algebra by generators and relations. For the planar algebra $\overline{{\operatorname{BMW}}(q, i)}^{\Z{2}}$ there is a canonical 6-box generator given by 
\begin{align*}
  \raisebox{-.5\height}{ \includegraphics[scale = .5]{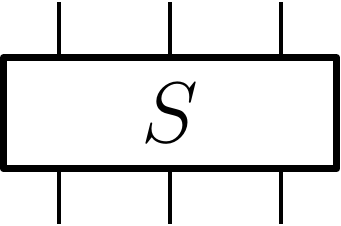}} &:= \raisebox{-.5\height}{ \includegraphics[scale = .5]{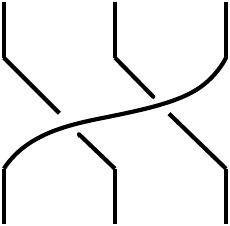}} + \raisebox{-.5\height}{ \includegraphics[scale = .5]{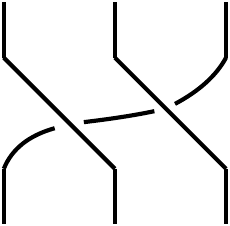}} +\raisebox{-.5\height}{ \includegraphics[scale = .5]{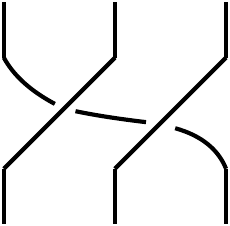}} +\raisebox{-.5\height}{ \includegraphics[scale = .5]{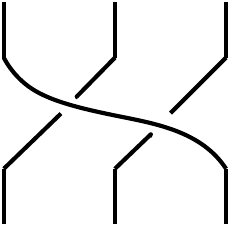}} + \raisebox{-.5\height}{ \includegraphics[scale = .5]{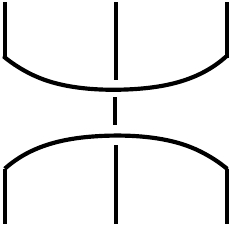}}+\raisebox{-.5\height}{ \includegraphics[scale = .5]{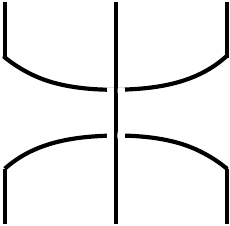}} \\
&  + \frac{2 (q-1) (q+1) \left(2 i q^2+q-i\right) (-2+q (q-i))}{q (1+q (q (2+q (q+4 i))-4 i))} \left(   \raisebox{-.5\height}{ \includegraphics[scale = .5]{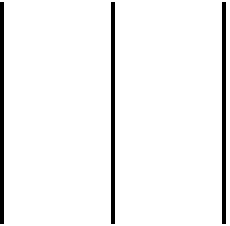}} + \raisebox{-.5\height}{ \includegraphics[scale = .5]{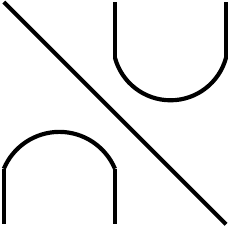}} +\raisebox{-.5\height}{ \includegraphics[scale = .5]{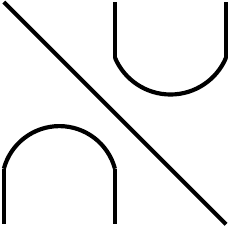}}\right) \\
&  + -\frac{2 i \left(q^2-1\right)^2 (-3+q (3 q-4 i))}{q (1+q (q (2+q (q+4 i))-4 i))} \left( \raisebox{-.5\height}{ \includegraphics[scale = .5]{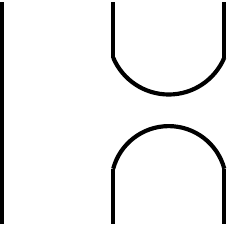}} + \raisebox{-.5\height}{ \includegraphics[scale = .5]{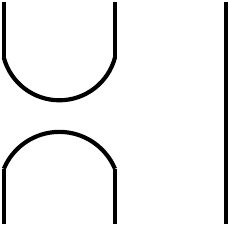}}\right).
\end{align*}
which is uniquely characterised (up to scaler) by the fact that it is uncappable in every position, and invariant under rotation by $\pi$ degrees. Finding more complicated relations that this generator satisfies however proves to be an exhausting task. However, finding such relations could reveal new ideas for skein theories involving 6-boxes, and hopefully pave the way for more exotic planar algebras based on similar skein theories.

\begin{ques}
Find a generators and relations presentation for the sub-planar algebra of  $\overline{{\operatorname{BMW}}(q, i)}^{\Z{2}}$ generated by $S$.
\end{ques}

%
%\[ \cat{sp}{4}{2}^{\Z{2}} : \raisebox{-.5\height}{ \includegraphics[scale = .5]{C2L2}} \]
%%\[ C_{2,2} : \raisebox{-.5\height}{ \includegraphics[scale = .5]{C2L2_2}} \]
%
%\[ \cat{sp}{6}{3}^{\Z{2}} : \raisebox{-.5\height}{ \includegraphics[scale = .5]{C3L3}} \]
%%\[ C_{3,3} : \raisebox{-.5\height}{ \includegraphics[scale = .6]{C3L3_2}} \]
%
%
%\[ \cat{sp}{8}{4}^{\Z{2}} : \raisebox{-.5\height}{ \includegraphics[scale = .5]{C4L4}} \]
%%\[ C_{4,4} : \raisebox{-.5\height}{ \includegraphics[scale = .55]{C4L4_2}} \]

\subsection*{ A non-commutative algebra object in $\cC( \lie{g}_2 , 4) $} \hspace{1em}

%\nn{
%\begin{lemma}
%The object $\mathbf{1} \oplus g$ has the structure of an algebra object in $\mathcal{E}_4(G_2)$.
%\end{lemma}
%\begin{cor}
%The object 
%\[ (0) \oplus (3 \Lambda_1)\oplus (4 \Lambda_1) \oplus (\Lambda_2) \]
%in $(G_2)_4$ has the structure of an algebra object.
%\end{cor}
%}
For any modular tensor category $\cC$, there exists an isomorphism from the group of braided auto-equivalences of $\cC$, to the group of invertible modules over $\cC$ \cite{MR2677836}. The goal of this Section will be to determine the invertible module over $\cC( \lie{g}_2 , 4) $ corresponding to the exceptional braided auto-equivalence
\[   \Lambda_1 \leftrightarrow 2 \Lambda_2  \quad \text{and}\quad \Lambda_2 \leftrightarrow 4\Lambda_1.  \]

We begin by describing how to take a braided auto-equivalence of a modular tensor category $\cC$, and construct an invertible module over $\cC$.

\begin{cons}
Let $\mathcal{F} \in \BrAut(\cC)$, then the object 
\[   A_{\mathcal{F}} :=  \bigoplus_{X\in \operatorname{Irr}(C)} X \otimes \mathcal{F}^{-1}(X^*) \]
has the structure of an algebra object in $\cC$. Except in the trivial case of $\cC  = \operatorname{Vec}$, the algebra $ A_{\mathcal{F}}$ is not simple. Let 
\[   A_{\mathcal{F}} = \oplus  A_i \]
be a decomposition of $A_{\mathcal{F}}$ into simple algebra objects. The invertible module corresponding to $\mathcal{F}$ is $\operatorname{Mod}_{\cC}(A_i)$ for any choice of algebra $A_i$. As the simple algebra objects $\{ A_i \}$ are pairwise Morita equivalent, the choice of $A_i$ does not affect the resulting module category. 
\end{cons}

For a given braided auto-equivalence $\mathcal{F}$, it is easy to determine the algebra object $A_{\mathcal{F}}$. In the case of the exceptional braided auto-equivalence of $\cC( \lie{g}_2 , 4) $ we have that
\begin{align*}   &A_{(\Lambda_1 \leftrightarrow 2 \Lambda_2  \hspace{ .15cm} \text{and} \hspace{ .15cm} \Lambda_2 \leftrightarrow 4\Lambda_1)} = \\ (0)^{\oplus 5} \oplus (\Lambda_1)^{\oplus 4} \oplus (2 \Lambda_1)^{\oplus 8} \oplus (3 \Lambda_1)^{\oplus 9}& \oplus (4 \Lambda_1)^{\oplus 5} \oplus (\Lambda_2)^{\oplus 5} \oplus (\Lambda_1 + \Lambda_2)^{\oplus 12} \oplus (2\Lambda_1 + \Lambda_2)^{\oplus 10} \oplus (2\Lambda_2)^{\oplus 4}. 
\end{align*}
However it is difficult in practice to determine how the algebra $A_{\mathcal{F}}$ decomposes into simple algebras. To determine how the algebra $A_{(\Lambda_1 \leftrightarrow 2 \Lambda_2  \hspace{ .15cm} \text{and} \hspace{ .15cm} \Lambda_2 \leftrightarrow 4\Lambda_1)}$ decomposes into simple algebra objects $\{A_i\}$, we apply the techniques of \cite{MR2909758} to compute an exhaustive list of possible simple algebra objects in $\cC( \lie{g}_2 , 4) $. Via \cite{1807.06131}, we know that the rank of $\operatorname{Mod}_{\cC}(A_i)$ is equal to 5 (the number of fixed points of the exceptional braided auto-equivalence). Thus we can restrict our attention to simple algebras in $\cC( \lie{g}_2 , 4) $ whose corresponding module category has rank 5. This reduction significantly reduces the complexity of the computation.

\begin{lemma}\label{lem:possibleAlgs}
Let $A \in \cC( \lie{g}_2 , 4) $ be an algebra whose corresponding module category has rank 5. Then $A$ is contained in the list
\begin{align*}
     A_1 &:= (0) \oplus (3 \Lambda_1) \oplus (4 \Lambda_1)\oplus (\Lambda_2) \\
     A_2 &:= (0) \oplus (2 \Lambda_1) \oplus (3 \Lambda_1) \oplus (4 \Lambda_1) \oplus (\Lambda_2) \oplus (\Lambda_1 + \Lambda_2)^{\oplus 2} \oplus (2\Lambda_1 + \Lambda_2) \\
     A_3 &:= (0) \oplus (2 \Lambda_1)^{\oplus 2} \oplus (3 \Lambda_1) \oplus (4 \Lambda_1) \oplus (\Lambda_2) \oplus (\Lambda_1 + \Lambda_2)^{\oplus 2} \oplus (2\Lambda_1 + \Lambda_2)^{\oplus 2} \\
     A_4 &:= (0) \oplus (\Lambda_1)\oplus (2 \Lambda_1) \oplus (3 \Lambda_1) \oplus (4 \Lambda_1) \oplus (\Lambda_2) \oplus (\Lambda_1 + \Lambda_2)^{\oplus 2} \oplus (2\Lambda_1 + \Lambda_2)^{\oplus 2} \oplus (2\Lambda_2) \\
     A_5 &:= (0) \oplus (\Lambda_1)\oplus (2 \Lambda_1)^{\oplus 2} \oplus (3 \Lambda_1)^{\oplus 3} \oplus (4 \Lambda_1) \oplus (\Lambda_2) \oplus (\Lambda_1 + \Lambda_2)^{\oplus 3} \oplus (2\Lambda_1 + \Lambda_2)^{\oplus 2} \oplus (2\Lambda_2) \\
     A_6 &:= (0) \oplus (\Lambda_1)\oplus (2 \Lambda_1)^{\oplus 3} \oplus (3 \Lambda_1)^{\oplus 3} \oplus (4 \Lambda_1)^{\oplus 2} \oplus (\Lambda_2)^{\oplus 2} \oplus (\Lambda_1 + \Lambda_2)^{\oplus 4} \oplus (2\Lambda_1 + \Lambda_2)^{\oplus 3} \oplus (2\Lambda_2) \\
     A_7 &:= (0) \oplus (\Lambda_1)^{\oplus 2} \oplus (2 \Lambda_1)^{\oplus 4} \oplus (3 \Lambda_1)^{\oplus 5} \oplus (4 \Lambda_1)^{\oplus 3} \oplus (\Lambda_2)^{\oplus 3} \oplus (\Lambda_1 + \Lambda_2)^{\oplus 4} \oplus (2\Lambda_1 + \Lambda_2)^{\oplus 4} \oplus (2\Lambda_2)^{\oplus 2} \\
     A_8 &:= (0) \oplus (\Lambda_1)^{\oplus 2} \oplus (2 \Lambda_1)^{\oplus 4} \oplus (3 \Lambda_1)^{\oplus 5} \oplus (4 \Lambda_1)^{\oplus 3} \oplus (\Lambda_2)^{\oplus 3} \oplus (\Lambda_1 + \Lambda_2)^{\oplus 5} \oplus (2\Lambda_1 + \Lambda_2)^{\oplus 4} \oplus (2\Lambda_2)^{\oplus 2} \\
     A_9 &:= (0) \oplus (\Lambda_1)^{\oplus 2} \oplus (2 \Lambda_1)^{\oplus 5} \oplus (3 \Lambda_1)^{\oplus 5} \oplus (4 \Lambda_1)^{\oplus 3} \oplus (\Lambda_2)^{\oplus 3} \oplus (\Lambda_1 + \Lambda_2)^{\oplus 6} \oplus (2\Lambda_1 + \Lambda_2)^{\oplus 5} \oplus (2\Lambda_2)^{\oplus 2}.
     \end{align*}
\end{lemma}
\begin{proof}
We use the techniques of \cite{MR2909758} to find all potential simple algebras in $\cC( \lie{g}_2 , 4)$. We then throw out all potential simple algebras whose corresponding potential module category does not have rank 5.
\end{proof}

With this finite list of possible simple algebra objects, it now becomes a combinatorial problem to determine how $A_{(\Lambda_1 \leftrightarrow 2 \Lambda_2  \hspace{ .15cm} \text{and} \hspace{ .15cm} \Lambda_2 \leftrightarrow 4\Lambda_1)}$ decomposes. While \textit{a-priori} there should be no reason to expect a unique decomposition, in this case we are lucky, and can exactly determine the decomposition.
\begin{lemma}
The algebra $A_{(\Lambda_1 \leftrightarrow 2 \Lambda_2  \hspace{ .15cm} \text{and} \hspace{ .15cm} \Lambda_2 \leftrightarrow 4\Lambda_1)} \in \cat{g}{2}{4}$ decomposes into simple algebras as
\[   A_3 \oplus A_4 \oplus A_4 \oplus A_5 \oplus A_5.\]
\end{lemma}
\begin{proof}
We brute force check all possible 1287 different ways of choosing 5 simple algebras from the list in Lemma~\ref{lem:possibleAlgs}, and see that $A_3 \oplus A_4 \oplus A_4 \oplus A_5 \oplus A_5$ is the only one that is equal to $A_{(\Lambda_1 \leftrightarrow 2 \Lambda_2  \hspace{ .15cm} \text{and} \hspace{ .15cm} \Lambda_2 \leftrightarrow 4\Lambda_1)}$.
\end{proof}
\begin{cor}
The objects $A_3$, $A_4$, and $A_5$ in $\cC( \lie{g}_2 , 4)$ have the structure of simple algebra objects.
\end{cor}
The module fusion graph of $\operatorname{Mod}_{\cC}(A_3)$ is

\begin{center} \begin{tikzpicture}[scale = .45,baseline={([yshift=-.5ex]current bounding box.center)}]
  \vertex(fe) at (-4,0) {};
  
        \vertex(f0) at (0,0) {};
        
        \vertex(f1) at (4,5) {};
         \vertex(f2) at (4,3) {};
          \vertex(f3) at (4,1) {};
           \vertex(f4) at (4,-1) {};
            \vertex(f5) at (4,-3) {};
            \vertex(f6) at (4,-5) {};

            \vertex(f7) at (8,3) {};
            \vertex(f8) at (8,1) {};
          \vertex(f9) at (8,-1) {};
           \vertex(f10) at (8,-3) {};

            \vertex(f11) at (12,1) {};
            \vertex(f12) at (12,-1) {};

            \Edge(fe)(f0)
            
            \Edge(f0)(f1)
            \Edge(f0)(f2)
            \Edge[style  = {double}](f0)(f3)
            \Edge[style  = {double}](f0)(f4)
            \Edge[style  = {double}](f0)(f5)
            \Edge(f0)(f6)
            
            \Edge[style  = {double}](f1)(f7)
            \Edge[style  = {double}](f1)(f8)
            \Edge[style  = {double}](f1)(f9)
            \Edge[style  = {double}](f1)(f10)
            
            \Edge(f2)(f7)
            \Edge(f2)(f8)
            \Edge(f2)(f9)
            \Edge(f2)(f10)
            
            \Edge[style  = {double}](f3)(f7)
            \Edge[style  = {double}](f3)(f8)
            \Edge[style  = {double}](f3)(f9)
            \Edge[style  = {double}](f3)(f10)
            
            \Edge[style  = {double}](f4)(f7)
            \Edge(f4)(f8)
            \Edge(f4)(f9)
            \Edge[style  = {double}](f4)(f10)
            
            \Edge[style  = {double}](f5)(f7)
            \Edge(f5)(f8)
            \Edge(f5)(f9)
            \Edge[style  = {double}](f5)(f10)
            
            \Edge(f6)(f7)
            \Edge(f6)(f8)
            \Edge(f6)(f9)
            \Edge(f6)(f10)
            
            \Edge(f11)(f7)
            \Edge(f11)(f8)
            \Edge(f11)(f9)
            \Edge(f11)(f10)
            \Edge(f12)(f7)
            \Edge(f12)(f8)
            \Edge(f12)(f9)
            \Edge(f12)(f10)
             
    \end{tikzpicture}
\end{center}

The category $\cC( \lie{g}_2 , 4)$ is unitary, hence this graph is the principal graph of a subfactor via Popa's embedding theorem \cite{MR1055708}. The existence of this subfactor was hypothesised in \cite{MR3339174}.

\appendix	

\section{The group structure of $\TenAut( \cat{sl}{r+1}{k})$.}\label{app:terry}

\begin{center}   \textbf{By Terry Gannon} \end{center}

In this appendix we explicitly determine the group structure of $\TenAut( \cat{sl}{r+1}{k})$. The goal will be to prove the following Theorem.
\begin{theorem}\label{thm:terry}
Let $r,k \geq 1$. Write $n  = r+1$ and $d =\operatorname{gcd}(n,k)$. Then we have
\[  \TenAut( \cat{sl}{r+1}{k}) \cong \Z{2}\times \{ b \in \Z{dn}^\times : b \equiv 1 \pmod d\}\]
where $ c :=  \begin{cases}
0 \text{ if } k \leq 2 \text{ or } r = 1 \\
1 \text{ if } k \geq 3 \text{ and } r \geq 2.
\end{cases}$

More explicitly, this subgroup of $\Z{dn}^\times$ is isomorphic to 
\[ \begin{cases}
 \Z{n'}^\times \times \Z{2} \times \Z{\frac{n''}{2}} \text{ if $2$ exactly divides $d$}\\
 \Z{n'}^\times \times \Z{n''} \text{ otherwise,}
\end{cases}\]
where we write $n = n'n''$ for $n'' = \operatorname{gcd}(n, k^\infty)$, so $n'$ is coprime to $k$.

\end{theorem}

We have seen that auto-equvialences of $\cat{sl}{r+1}{k}$ come in two forms, the charge conjugation auto-equivalences, and the simple current auto-equivalences. The charge conjugation auto-equivalences of $\cat{sl}{r+1}{k}$ are already well understood. They exist and are distinct from the simple current auto-equivalences when both $k\geq 3$ and $r\geq 2$. Each of these charge conjugation auto-equivalences has order $2$, and commutes with all other auto-equivalences. More complicated is the composition of the simple current auto-equivalences. Let us study these in detail.

Recall that the simple current auto-equivalences of $ \cat{sl}{r+1}{k}$ are parameterised by the set $G := \{ a \in  \Z{n} : 1+ka \text{ is coprime to } r+1    \}$ via the map $a \mapsto  \cF_{k\Lambda_1, a}$. We begin by determining the group structure on this set which corresponds to the composition of simple current auto-equivalences.

Let $a,b$ two elements in the set $G$. Let $X \in \cat{sl}{r+1}{k}$, and $n$ the unique integer (modulo $r+1$) such that 
\[ \sigma_{X,k\Lambda_1}\circ \sigma_{k\Lambda_1,X} = e^{2\pi i \frac{n}{r+1}}\id_{(k\Lambda_1)\otimes X}.\]
 Then by the definition of the simple current auto-equivalence we have $ \cF_{k\Lambda_1, b}  (X) = (k\Lambda_{-bn})\otimes X$, and so $ \cF_{k\Lambda_1, a}(  \cF_{k\Lambda_1, b}  (X)) =  (k\Lambda_{-am})\otimes (k\Lambda_{-bn})\otimes X$, where $m$ is the unique integer (modulo $r+1$) such that 
\[\sigma_{ (k\Lambda_{-bn})\otimes X,k\Lambda_1}\circ \sigma_{k\Lambda_1, (k\Lambda_{-bn})\otimes X} = e^{2\pi i \frac{m}{r+1}}\id_{(k\Lambda_1)\otimes  (k\Lambda_{-bn})\otimes X}.\]
A direct computation gives that $m = n + nbk$. Hence
\[  \cF_{k\Lambda_1, a}(  \cF_{k\Lambda_1, b}  (X)) =\left(k\Lambda_{-an(1 + bk)}\right)\otimes (k\Lambda_{-bn})\otimes X= \left(k\Lambda_{n( -a - b - kab)}\right)\otimes X.\]
Thus the auto-equivalence $\cF_{k\Lambda_1, a}\circ  \cF_{k\Lambda_1, b}$ behaves the same on objects as the auto-equivalence $\cF_{k\Lambda_1, a+b+kab}$. As the category $\cat{sl}{r+1}{k}$ has no non-trivial gauge auto-equivalences, this implies that 
\[    \cF_{k\Lambda_1, a}\circ  \cF_{k\Lambda_1, b} \cong \cF_{k\Lambda_1, a+b+kab}.\] 

Hence we can endow the set $G$ with the group structure 
\[a \cdot a' := a + a' + aa'k\]
to capture the composition of the simple current auto-equivalences.

\begin{lemma}
The group $G$ is isomorphic to the subgroup $ G(n,d) := \{b \in \Z{dn}^\times : b \equiv 1 \pmod d\}$ of $\Z{dn}^\times$.
\end{lemma}
\begin{proof}
There exists integers $\ell, m$ such that $\ell k + mn = d$. Since $\ell \frac{k}{d} \equiv 1 \pmod {\frac{n}{d}}$, we have $\ell$ is coprime to $\frac{n}{d}$. In fact, we can choose $\ell$ so that it is coprime to $n$: choose an integer $L$ so that 
\[   L \equiv \begin{cases}
1 \pmod p \text{ if } p \divides \operatorname{gcd}(\ell, n) \\
0 \pmod  p \text{ if } p \divides n \text{ and } p \not\divides \ell 
\end{cases}
\]
then $\ell' := \ell +L\frac{n}{d}$ and $m' :=m - L\frac{k}{d}$ also satisfy $d = \ell'k + m'n$, though $\ell'$ is coprime to $n$.

Since $\ell$ is coprime to $n$, the map $b \mapsto \ell b$ on $\Z{n}$ is a bijection. Because $1 + db \equiv 1 + k \ell n \pmod n$, it restricts to a bijection $G(n,d) \to G$ as sets. Comparing the product in $G(n,d)$, i.e $(1 + db)(1 + db') = 1 + d(b + b') + d^2bb'$, with that of $G$, i.e. $(\ell b) \cdot (\ell b') = \ell b + \ell b' + d\ell bb' \pmod n$, we see then that this map defines a group isomorphism.
\end{proof}

All that remains is to determine $G(n,d)$ as a product of cyclic groups. Let us write $\nu_p$ for the multiplicity of the prime $p$ in $n$, and $\eta_p$ for the multiplicity of the prime $p$ in $d$. Then $\Z{n}^\times \cong  \prod_p \Z{p^{\nu_p}}^\times$ and $G(n,d) = \prod_p G(p^{\nu_p}, p^{\eta_p})$, and it suffices to work locally. Trivially, $G(p^\nu,1) \cong  \Z{p^\nu}^\times$, and so $G(n,d) \cong \Z{n'} \times G(n'',d)$.

\begin{lemma}
Assume $ 0 < \eta < \nu$ and $p$ is prime. As groups,
\[   G(p^\nu, p^\eta) \cong  \begin{cases}
\Z{2}\times \Z{2^{\nu - 1}} \quad \text{ if $p= 2$ and $\eta =1$}\\
\Z{p^{\nu}} \quad \text{ otherwise.}
\end{cases}\] 
\end{lemma}
\begin{proof}
First note that for any $p$, $b \mapsto 1 + p^\eta b : \Z{p^\nu} \to G(p^\nu, p ^\eta)$ is a bijection of sets, so $| G(p^\nu, p^\eta)| = p^\nu$.

If $p$ is odd, then it is a classical group theory result that $\Z{p^{\nu+\eta}}^\times \cong \Z{p-1} \times  \Z{p^{\nu + \eta - 1}}$. Then $G(p^\nu, p^\eta)$ will be a subgroup of the cyclic subgroup $\Z{p^{\nu + \eta - 1}}$, hence is itself cyclic with $G(p^\nu, p^\eta) \cong \Z{p^\nu}$ as groups.

When $p = 2$ we have that $\Z{2^{\nu+\eta}}^\times \cong \Z{2}\times \Z{2^{\nu + \eta - 2}}$, where the $\Z{2}$ factor can be taken to be generated by $-1 \in \Z{2^{\nu+\eta}}^\times$, and $\Z{2^{\nu  +\eta-2}}$ taken to be all $c \in \Z{\nu  +\eta}$ with $c \equiv 1 \pmod 4$. When $\eta \geq 2$, $G(2^\nu, 2^\eta)$ must then be a subgroup of the cyclic subgroup $\Z{2^{\nu + \eta - 2}}$, and so must be isomorphic to $\Z{2^\nu}$. When $\eta =1$, the condition $c \equiv 1 \pmod 2^\eta$ in $G(2^\nu, 2^\eta)$ is automatically satisfied, so $G(2^\nu, 2) = \Z{\nu +1}^\times$ and we're done.  
\end{proof}

\bibliography{bibliography} 
\bibliographystyle{plain}
\end{document}